\documentclass[a4paper]{amsart}

\usepackage{amsmath}
\usepackage{amsfonts}
\usepackage{amssymb}
\usepackage{epsfig}
\newtheorem{theorem}{Theorem}[section]

\newtheorem{lm}[theorem]{Lemma}
\newtheorem{tr}[theorem]{Theorem}
\newtheorem{cor}[theorem]{Corollary}

\newtheorem{rem}[theorem]{Remark}
\newtheorem{pr}[theorem]{Proposition}

\usepackage{color}

\newtheorem{ex}[theorem]{Example}

\begin{document}
	
	\title{Almost-simple algebraic supergroups}
	\author{S. Bouarroudj}
	\address{New York University Abu Dhabi, Division of Science and Mathematics, Po Box. 129188, Abu Dhabi, United Arab Emirates}
	\email{sofiane.bouarroudj@nyu.edu}
	\author{A. N. Zubkov}
	\address{Sobolev Institute of Mathematics, Omsk Branch, Pevtzova 13, 644043 Omsk, Russian Federation}
	\email{a.zubkov@yahoo.com}
	\begin{abstract}
		We investigate certain almost-simple algebraic supergroups over an algebraically closed field of characteristic zero or odd characteristic. Beyond those whose Lie superalgebras are simple and appear in Kac's classification, we construct new families of almost-simple supergroups whose Lie superalgebras are either non-simple or simple but fall outside Kac's classification.
	\end{abstract}
	\maketitle
	
	\section*{introduction}
	The purpose of this paper is to initiate the study of almost-simple supergroups over an algebraically closed field $\Bbbk$ of zero or odd characteristic. On the one hand, the authors are motivated by the recent progress in the study of simple modular Lie superalgebras (cf. \cite{BGL3, BGL4, CE, eld, K}), on the other hand, by the problem of describing the structure of quasi-reductive supergroups, in which the case of positive characteristic is very far from even formulating plausible conjectures. It is worth noting that in the case of zero characteristic, such supergroups have more or less satisfactory description (cf. \cite{grzub, sergan2}). Essentially, almost-simple supergroups can be defined in either strong or weak sense, or in brief, SAS-supergroups and WAS-supergroups, respectively. More precisely, a smooth connected supergroup $\mathbb{G}$ is called \emph{almost-simple in the weak sense}, if it is semi-simple, non-commutative and its only smooth connected normal supersubgroups are $\mathbb{G}$ or $e$. If $\mathbb{G}$ satisfies the stronger condition that any proper normal supersubgroup is finite, then $\mathbb{G}$ is called \emph{almost-simple in the strong sense}. Note that if $\mathbb{G}$ is purely even, then these two conditions are equivalent to each other.
	
	The class of SAS-supergroups is contained in the class of WAS-supergroups. Furthermore, if the characteristic of the ground field is positive, then any SAS-supergroup is quasi-reductive and we hope that in this case describing the structure of  SAS-supergroups will not turn out to be a wild task as describing the structure of quasi-reductive supergroups. 
	
	If $\mathrm{char}\Bbbk=0$, then these two classes of supergroups  are the same. Moreover, a supergroup $\mathbb{G}$ is almost-simple if and only if its Lie superalgebra $\mathfrak{G}$ is simple (see Proposition \ref{SAS in zero char} below). 
	In other words, the classification of almost-simple supergroups in zero characteristic is analogous to Kac's classification of simple Lie superalgebras. It is worth noting that
	in this particular case, a SAS-supergroup is not necessary quasi-reductive.
	Indeed, let $n\geq 3$ and $\mathbb{G}\simeq\mathrm{Aut}(\Lambda(n))$.
	Then $\mathrm{Lie}(\mathbb{G})\simeq\mathfrak{Der}(\Lambda(n))\simeq\mathrm{W}(n)$ is simple, hence $\mathbb{G}$ is SAS, but $\mathbb{G}_{ev}\simeq\mathrm{GL}_n\ltimes\mathrm{U}_n$, where
	$\mathrm{GL}_n$ consists of all invertible linear substitutions of free odd generators
	of $\Lambda(n)$ and $\mathrm{U}_n$ is the normal unipotent subgroup consisting of automorphisms acting trivially on $\Lambda(n)^+/(\Lambda(n)^+)^2$, where $\Lambda(n)^+$ consists of exterior tensors without constant terms.    
	
	As opposed to this, if 
	$\mathrm{char}\Bbbk=p>0$, then the class of SAS-supergroups is properly contained in the class of WAS-supergroups. For example, if $p \nmid (m-n)$ or $m=n$, then the \emph{projective special linear supergroup} $\mathrm{PSL}(m|n)$ is a WAS-supergroup, but not SAS. If $p|(m-n)$ and $p\nmid m$ or $p\nmid n$, then the derived supersubgroup
	$\mathcal{D}(\mathrm{PSL}(m|n))$ is the unique proper normal smooth connected  supersubgroup of $\mathrm{PSL}(m|n)$ and WAS-supergroup itself. Moreover, it is SAS if and only if $m=1$ or $n=1$.
	The case $p|m, p|n$ is even more curious and we refer to Proposition \ref{when PSL is SAS or WAS} for more details.
	Furthermore, we provide many examples of SAS-supergroups with non-simple Lie superalgebras.
	
	The central result, around which further analysis of the structure of SAS-supergroups is built, is Proposition \ref{SAPS in positive char}. This proposition states that over a field of positive characteristic a supergroup $\mathbb{G}$ is SAS if and only if its largest even (super)subgroup $G=\mathbb{G}_{ev}$ is almost-simple and the odd component $\mathfrak{G}_{\bar 1}$ of its Lie superalgebra $\mathfrak{G}$, regarded as a $G$-module with respect to the adjoint action, has neither nonzero submodule $\mathfrak{V}$ with $[\mathfrak{G}_{\bar 1}, \mathfrak{V}]=0$, nor submodule $\mathfrak{W}\neq\mathfrak{G}_{\bar 1}$, such that $\mathfrak{G}_{\bar 1}/\mathfrak{W}$  is a trivial $G$-module.
	
	It has already been proved by Gavarini and Fioresi that all Lie superalgebras in the Kac's list are algebraic (see \cite{gavfi1, gavfi2}). They proved even more, any such superalgebra $\mathfrak{G}$, with the exception of the family $\mathfrak{osp}_{\alpha}(4|2)$, has a \emph{Chevalley basis}, that is a basis of a free $\mathbb{Z}$-submodule $\mathfrak{G}_{\mathbb{Z}}$, such that  the rank of $\mathfrak{G}_{\mathbb{Z}}$ is equal to the dimension of $\mathfrak{G}$ and $[\mathfrak{G}_{\mathbb{Z}}, \mathfrak{G}_{\mathbb{Z}}]\subseteq \mathfrak{G}_{\mathbb{Z}}$. The family $\mathfrak{osp}{\alpha}(4|2)$ requires a separate treatment, since its structure constants depend on the parameter $\alpha$. In this case, suitable integral forms are constructed over rings containing $\mathbb{Z}[\alpha]$, rather than uniformly over $\mathbb{Z}$.  Using this $\mathbb{Z}$-form (or $\mathbb{Z}[\alpha]$-form, respectively) of $\mathfrak{G}$, one can construct the associated \emph{Chevalley supergroup} (over a field $\Bbbk$ of arbitrary characteristic), such that its Lie superalgebra is isomorphic to $\mathfrak{G}_{\mathbb{Z}}\otimes_{\mathbb{Z}} \Bbbk$. We will re-prove this result for all simple Lie superalgebras from Kac's list, except for those that belong to the \emph{Cartan series}, using only the Harish-Chandra pair technique, focusing on whether SAS or WAS are the corresponding supergroups, respectively. We do not address the more subtle issue of uniqueness, since it is not essential to our purposes. There are, however, many cases where it can be determined from the relevant auxiliary statements. 
	For example, we show that the simple Lie superalgebra $\mathfrak{brj}(2; 5)$ (defined only in the characteristic $5$), which does not belong to the Kac's list, is algebraic, by proving that the couple $(\mathrm{Sp}_4, \mathfrak{brj}(2; 5))_{\bar 1})$ has the unique (up to an isomorphism) nontrivial structure of Harish-Chandra pair. In particular, there is a unique (up to a supergroup isomorphism) SAS-supergroup $\mathbb{G}$, such that $\mathrm{Lie}(\mathbb{G})\simeq\mathfrak{brj}(2; 5)$. 
	
	By analyzing the structure of the supergroup $\mathrm{AB}(3)$ (from Kac's list) in characteristic $3$, we will show that it contains a normal (smooth and connected) subsupergroup $\mathbb{H}$, whose Lie superalgebra is isomorphic to the simple Lie superalgebra ${\mathfrak g}(3,3)^{(1)}/{\mathfrak z}$ introduced in \cite[Theorem 4.22(ii)]{eld} (see also \cite{BGL4} for the description of the corresponding Cartan matrices). 
	
	Returning to Proposition \ref{SAPS in positive char} and dealing with SAS-supergroups, note that if $\mathrm{char}\Bbbk=p>0$, then the fact that $\mathbb{G}_{ev}=G$ is almost-simple does not imply that $\mathfrak{G}_{\bar 0}$ is simple. More generally, let $G$ be an almost-simple algebraic group and let $\mathfrak{g}$ be its Lie algebra. Assume also that if $G$ of type $\mathrm{G}_2$, then $p\neq 3$. Then there are four mutually exclusive alternatives:
	\begin{enumerate}
		\item $\mathfrak{g}$ is simple;
		\item the unique nonzero proper ideal of $\mathfrak{g}$ is its one-dimensional center $\mathfrak{z}$, so that $\mathfrak{g}/\mathfrak{z}$ is a simple Lie algebra;
		\item the unique nonzero proper ideal of $\mathfrak{g}$ is its derived ideal $[\mathfrak{g}, \mathfrak{g}]$ of codimension one, which is also simple as a Lie algebra;
		\item $\mathfrak{g}=\mathfrak{z}\oplus [\mathfrak{g}, \mathfrak{g}],$ where $\mathfrak{z}$  is the one-dimensional center and the derived ideal $[\mathfrak{g}, \mathfrak{g}]$ is simple as a Lie algebra.
	\end{enumerate}   
	Finally, if $G$ of type $\mathrm{G}_2$ and $p=3$, then $\mathfrak{g}$ has an ideal $I$ isomorphic to $\mathfrak{psl}_3$, such that 
	$\mathfrak{g}/I\simeq\mathfrak{psl}_3$. Recall that $\mathfrak{psl}_3$ is not algebraic. 
	In what follows we do not consider the latter case. 
	
	In each of the first four cases, we partially describe the Lie superalgebra $\mathfrak{G}$. For example, if $\mathfrak{G}_{\bar 0}$ is simple,  then any nonzero proper superideal of $\mathfrak{G}$ contains the derived superideal 
	$[\mathfrak{G}, \mathfrak{G}]=\mathfrak{G}_{\bar 0}\oplus [\mathfrak{G}_{\bar 0}, \mathfrak{G}_{\bar 1}]$, that is very close to being simple. Nevertheless, we construct a SAS-supergroup $\mathbb{G}$, such that $G\simeq\mathrm{SL}_2$
	and $0\neq [\mathfrak{G}, \mathfrak{G}]\neq\mathfrak{G}$ (see Theorem \ref{certain couples} below). We also prove that this supergroup is isomorphic to $\mathrm{Aut}(\mathfrak{spo}(2|1))^0$.
	
	Further, if $\mathfrak{G}_{\bar 0}$ satisfies the alternatives $(2)$ or $(4)$, then the structure of the corresponding SAS-supergroup $\mathbb{G}$ is completely described in Proposition \ref{SAS with small square of odd component}, provided $[\mathfrak{G}_{\bar 1}, \mathfrak{G}_{\bar 1}]=\mathrm{Z}(\mathfrak{G}_{\bar 0})$. In the remaining cases, we prove that if $\mathbb{G}$ is not purely even, then it contains a finite normal supersubgroup $\mathbb{H}$, such that $\mathbb{L}=\mathbb{G}/\mathbb{H}$ is not purely even and also a SAS-supergroup, and 
	one of the following statements holds: 
	\begin{enumerate}
		\item $\mathbb{L}$ is as in Proposition \ref{SAS with small square of odd component} ;
		\item Any nonzero superideal of $\mathfrak{L}$ contains $[\mathfrak{L},  \mathfrak{L}]$;
		\item $\mathbb{L}$ satisfies  the alternatives $(3)$ or $(4)$ and $[\mathfrak{L}_{\bar 0}, \mathfrak{L}_{\bar 0}]=[\mathfrak{L}_{\bar 1}, \mathfrak{L}_{\bar 1}]$.
	\end{enumerate}
	Additionally, $\mathbb{H}$ can be chosen being infinitesimal of height at most one.
	
	In the next article, we are going to explore the second case. This case poses an interesting problem of describing Lie superalgebras $\mathfrak{L}$, which are not simple, but every nonzero superideal of $\mathfrak{L}$ contains $[\mathfrak{L}, \mathfrak{L}]$. It seems natural to call such Lie superalgebras \emph{almost-simple}. The question immediately arises whether the derived superideal of an almost-simple Lie superalgebra $\mathfrak{L}$ is simple? 
	Is it true if $\mathfrak{L}$ is the Lie superalgebra of some SAS-supergroup?
	
	In conclusion, we note that this article is mainly devoted to the study of the properties of SAS-supergroups. The wider class of WAS-supergroups will be investigated in an upcoming paper. 
	
	\section{Basic definitions and results}
	
	Throughout this article, $\Bbbk$ is an algebraically closed field of zero or odd characteristic. Let $\mathsf{SAlg}_{\Bbbk}$ denote the category of supercommutative superalgebras. An \emph{affine supergroup} is a representable functor $\mathbb{G}$ from $\mathsf{SAlg}_{\Bbbk}$ to the category of groups. 
	In other words, for any $A\in \mathsf{SAlg}_{\Bbbk}$
	we have
	\[\mathbb{G}(A)=\mathrm{Hom}_{\mathsf{SAlg}_{\Bbbk}}(\Bbbk[\mathbb{G}], A),\]
	where $\Bbbk[\mathbb{G}]$ is a \emph{Hopf superalgebra} representing $\mathbb{G}$. 
	
	All (super)groups are assumed to be affine and algebraic. If $\mathbb{G}, \mathbb{H}, \ldots , $ are supergroups, then their Lie superalgebras are denoted by
	$\mathfrak{G}, \mathfrak{H}, \ldots$.
	
	Similarly, their largest even supersubgroups of $\mathbb{G}_{ev}, \mathbb{H}_{ev}, \ldots$, are denoted by $G, H, \ldots$, and regarded
	as ordinary algebraic groups, unless stated otherwise. 
	
	Let $G$ be an algebraic group and $V$ be a finite dimensional $G$-module.   
	The pair $(G, V)$ is said to be a \emph{Harish-Chandra pair}, if the following conditions hold:
	\begin{enumerate}
		\item There is a symmetric bilinear map $V\times V\to\mathrm{Lie}(G)$, denoted by $[ \ , \ ]$;
		\item This map is $G$-equivariant with respect to the diagonal action of $G$ on $V\times V$ and the adjoint action of $G$ on $\mathrm{Lie}(G)$;
		\item The induced action of $\mathrm{Lie}(G)$ on $V$, denoted by the same symbol $[ \ , \ ]$,  satisfies $[[v, v], v]=0$ for all $v\in V$. 
	\end{enumerate}
	The morphism of Harish-Chandra pairs $(G, V)\to (H, W)$ is a couple $f : G\to H$ and $u : V\to W$, where $f$ is a morphism of algebraic groups and $u$ is a morphism of vector spaces, such that 
	\begin{enumerate}
		\item $u (gv)= f(g)u(v), \text{ for all } g\in G, \text{ and } v\in V$;
		\item $[u(v), u(v')]=\mathrm{d}_e(f)([v, v']), \text{ for all } v, v'\in V$. 
	\end{enumerate}

	With each algebraic supergroup $\mathbb{G}$ we associate its Harish-Chandra pair $(G, \mathfrak{G}_{\bar 1})$, where $G=\mathbb{G}_{ev}$ acts on $\mathfrak{G}$ via the adjoint action
	and the bilinear map $\mathfrak{G}_{\bar 1}\times \mathfrak{G}_{\bar 1}\to\mathrm{Lie}(G)=\mathfrak{G}_{\bar 0}$ is just a restriction of the Lie bracket. 
	The functor $\mathbb{G}\mapsto (G, \mathfrak{G}_{\bar 1})$ is an equivalence of the category of algebraic supergroups to the category of Harish-Chandra pairs (see \cite[Theorem 5.4 and Theorem 6.1]{mas-shib} or \cite[Theorem 12.10]{maszub}). 
	
	Given two supergroups $\mathbb{H}$ and $\mathbb{G}$. Claiming that $\mathbb{H}\leq \mathbb{G}$ is equivalent to  $H\leq G$, $\mathfrak{H}_{\bar 1}$ is an $H$-submodule of $(\mathfrak{G}_{\bar 1})|_{H}$ and the Lie bracket on $\mathfrak{H}_{\bar 1}$ is the restriction of the Lie bracket on $\mathfrak{G}_{\bar 1}$. 
	Further, if a (closed) supersubgroup $\mathbb{H}$ of  $\mathbb{G}$ is represented by its Harish-Chandra (sub)pair $(H, \mathfrak{H}_{\bar 1})$, then  
	$\mathbb{H}$ is normal in $\mathbb{G}$ if and only if
	\begin{enumerate}
		\item $H$ is normal in $G$;
		\item $\mathfrak{H}_{\bar 1}$ is a $G$-submodule of $\mathfrak{G}_{\bar 1}$;
		\item $H\leq\ker(G\to\mathrm{GL}(\mathfrak{G}_{\bar 1}/\mathfrak{H}_{\bar 1}))$ (this is actually just the induced action of $G$ on $\mathfrak{G}_{\bar 1}/\mathfrak{H}_{\bar 1}$, i.e. $g(x+\mathfrak{H}_{\bar 1})=gx+\mathfrak{H}_{\bar 1}, g\in G, x\in\mathfrak{G}_{\bar 1}$);
		\item $[\mathfrak{G}_{\bar 1}, \mathfrak{H}_{\bar 1}]\subseteq\mathrm{Lie}(H)$.
	\end{enumerate}

	Besides, the supergroup $\mathbb{G}/\mathbb{H}$ is represented by the Harish-Chandra pair $(G/H, \mathfrak{G}_{\bar 1}/\mathfrak{H}_{\bar 1})$.
	For more details, we refer to \cite[Theorem 6.6(1)]{mas-shib} and \cite[Theorem 12.11]{maszub}. We call the conditions $(1)-(4)$ a \emph{normality criterion}. 
	\begin{lm}\label{product of normal supersubgroups}
		If $\mathbb{H}$ and $\mathbb{K}$ are normal supersubgroups of $\mathbb{G}$, then $\mathbb{H}\mathbb{K}$ is a normal supersubgroup again, and its Harish-Chandra pair 
		is equal to $(H K, \mathfrak{H}_{\bar 1}+\mathfrak{K}_{\bar 1})$.	
	\end{lm}
	\begin{proof}
		This is a trivial consequence of the normality criterion. 	
	\end{proof}
	Recall that a supergroup $\mathbb{G}$ is smooth or connected, if $\mathbb{G}_{ev}=G$ is. If $\mathrm{char}\Bbbk=p>0$, then the $r$-th \emph{Frobenius kernel} 
	$\mathbb{G}_r$ of $\mathbb{G}$ is uniquely determined by the Harish-Chandra subpair $(G_r, \mathfrak{G}_{\bar 1})$, where $G_r$ is the $r$-th Frobenius kernel of $G$ (see \cite{jan} and \cite{zub1} for the equivalent definition). A finite supergroup $\mathbb{G}$ is said to be \emph{infinitesimal of height at most one}, if $\mathbb{G}=\mathbb{G}_1$. 
	\begin{lm}
		Let $\mathbb{G}$ be an algebraic supergroup. Then $\mathbb{G}$ has the largest normal smooth connected solvable supersubgroup, called the \emph{solvable radical} of $\mathbb{G}$, and it is denoted by $\mathrm{R}(\mathbb{G})$.  
	\end{lm}
	\begin{proof} Use \cite[Corollary 6.4]{maszub1} and Lemma \ref{product of normal supersubgroups}.
	\end{proof}
	If $\mathrm{R}(\mathbb{G})=e$, then $\mathbb{G}$ is said to be \emph{semisimple}.
	
	Recall also that if $\mathbb{G}$ is purely odd, that is if $G=e$, then $\mathbb{G}\simeq (\mathbb{G}_a^-)^{\dim\mathfrak{g}_{\bar 1}}$, where $\mathbb{G}^-_a$ is the $0|1$-dimensional purely odd unipotent supergroup. More precisely, $\Bbbk[\mathbb{G}^-_a]$ is freely generated by an odd primitive element, or equivalently, for any superalgebra $A$ there is $\mathbb{G}^-_a(A)=(A_{\bar 1}, +)$. Any purely odd supergroup is obviously smooth, connected and abelian.
	
	A supergroup $\mathbb{G}$ is said to be \emph{split}, if the natural embedding $\mathbb{G}_{ev}\to\mathbb{G}$ is split (i.e., there is a supergroup morphism $\mathbb{G}\to\mathbb{G}_{ev}$, such that its composition with this embedding is $\mathrm{id}_{\mathbb{G}_{ev}}$). This is equivalent to the condition 
	$[\mathfrak{G}_{\bar 1}, \mathfrak{G}_{\bar 1}]=0$ (cf. \cite[Proposition 6.1]{bz}). If $\mathbb{G}$ is split, then $\mathbb{G}\simeq\mathbb{G}_{ev}\ltimes \mathbb{G}_{odd}$, where
	$\mathbb{G}_{odd}\simeq (\mathbb{G}_a^-)^{\dim\mathfrak{G}_{\bar 1}}$ is a normal purely odd supersubgroup. 
	
	\begin{lm}\label{Lie superalgebra is trivial} 
		If $\mathrm{Lie}(\mathbb{G})=0$, then $\mathbb{G}$ is purely even  and \'etale. In particular, any finite smooth connected supergroup is purely odd.
	\end{lm}
	\begin{proof}
		The first statement is obvious. Next, if $\mathbb{G}=G$, then \cite[Proposition 10.15]{milne}, applied to $e\leq G^0$, infers $G^0=e$. 
		If $\mathbb{G}$ is a finite smooth connected supergroup, then $0=\dim G=\dim\mathrm{Lie}(G)$ implies $G$ is \'etale, hence trivial.
	\end{proof}	
	
	We now define the classes of algebraic supergroups considered in this paper:
	
	\noindent{{\bf WAPS.}} A smooth, connected, noncommutative supergroup whose only smooth connected normal subsupergroups are itself and the trivial supergroup. 
	
	\noindent{{\bf SAPS.}} A smooth, connected, noncommutative supergroup every proper normal subsupergroup of which is finite.
	
	A smooth connected supergroup $\mathbb{G}$ is said to be \emph{almost-simple in the weak sense} (shortly, WAS-supergroup), if it is WAPS and semi-simple. It is called \emph{almost-simple in the strong sense} (shortly, SAS-supergroup) if it is SAPS and semi-simple (compare with \cite[Definitions 19.7 and 19.8]{milne}). 
	
	Observe that if $\mathbb{G}$ is purely even, then these definitions are equivalent to each  other. Moreover, any WAPS-group is certainly WAS. In fact, if $G$ is WAPS and $H$ is a proper normal subgroup of $G$, then 
	by \cite[Corollary 1.39 and Corollary 1.87]{milne} there is a normal smooth subgroup $H_{red}$ in $G$, such that $H_{red}(\Bbbk)=H(\Bbbk)$ and $H_{red}\leq H$.	More precisely, by Hilbert Nullstellensatz there is $\Bbbk[H_{red}]=\Bbbk[H]/\mathsf{rad}(\Bbbk[H])$. By \cite[Proposition 1.52 and Corollary 1.85]{milne}, $H_{red}^0$ is normal, smooth and connected. Since $H_{red}^0\neq G$, $H_{red}^0$ is trivial. In particular, $H_{red}$ is finite, therefore $H$ is finite as well.
	The proof of the converse statement is similar. 
	
	Finally, if $\mathrm{R}(G)\neq e$,  then $\mathrm{R}(G)=G$. The derived subgroup
	$\mathcal{D}(G)$ is a proper normal smooth connected subgroup by \cite[Corollary 6.19(b)]{milne}, hence $\mathcal{D}(G)=e$, a contradiction! 
	\begin{lm}\label{when a quotient of SAS is SAS}
		Let $G$ be a SAS-group. If $H$ is a finite normal subgroup of $G$, then $G/H$ is SAS.	
	\end{lm}	
	\begin{proof}
		It is clear that any proper normal subgroup of $G/H$ is finite. Assume that $\mathrm{R}(G/H)\neq e$. Let $R$ be the preimage of $\mathrm{R}(G/H)$ in $G$. Since $R$ is infinite, we have $R=G$ since $G$ is a SAS group. Thus $G/H$ is solvable and therefore, a $k$-th
		derived subgroup $\mathcal{D}^k(G)$ is contained in $H$ for sufficiently large positive integer $k$. By \cite[Corollary 6.19(b)]{milne}, $\mathcal{D}^k(G)$ is smooth and connected, 
		hence trivial. In particular, $\mathcal{D}(G)$ is a proper normal subgroup, whence trivial by the same reasoning, meaning that $G$ is SAS, $\mathcal{D}(G)$ is finite, but it is also smooth and connected. This contradiction completes the proof.
	\end{proof}	
	\section{Auxiliary results on WAPS and SAPS-supergroups}
	
	\begin{lm}\label{SAS is WAS}
		Any SAS-supergroup is a WAS-supergroup.	If $\mathrm{char}\Bbbk=0$, then the converse is also true.
	\end{lm}
	\begin{proof}
		Let $\mathbb{G}$ be a SAS-supergroup. If $\mathbb{H}$ is a normal smooth connected supersubgroup and $\mathbb{H}\neq \mathbb{G}$, then $\mathbb{H}$ is finite. Lemma \ref{Lie superalgebra is trivial} implies that $\mathbb{H}$ is purely odd, hence abelian. Therefore,  
		$\mathbb{H}\leq\mathrm{R}(\mathbb{G})=e$.
		
		If $\mathrm{char}\Bbbk=0$, then any algebraic supergroup is smooth. In particular, if $\mathbb{H}$ is a proper normal supersubgroup of $\mathbb{G}$, then its connected component $\mathbb{H}^0$ is represented
		by the pair $(H^0, \mathfrak{H}_{\bar 1})$, hence normal again by \cite[Proposition 1.52]{milne} and the normality condition. If $\mathbb{G}$ is WAS, then $\mathbb{H}^0=e$.
	\end{proof}
	\begin{rem} {\em 
			Observe that SAPS groups are not necessarily WAPS groups and vise versa. }
	\end{rem}
	\begin{lm}\label{some simple case}
		Let $\mathbb{G}$ be a smooth connected supergroup. If $\mathfrak{G}$ is simple, then $\mathbb{G}$ is a WAS-supergroup.	
	\end{lm}
	\begin{proof}
		If $\mathbb{H}$ is its normal smooth connected supersubgroup, then either $\mathfrak{H}=\mathfrak{G}$ or $\mathfrak{H}=0$ since $\mathfrak{G}$ is simple. In the first case $\mathfrak{H}_{\bar 0}=\mathfrak{G}_{\bar 0}$ infers $H=G$ (use again \cite[Proposition 10.15]{milne}), hence $\mathbb{H}=\mathbb{G}$. In the second case $\mathbb{H}$ is purely even and \'etale by Lemma \ref{Lie superalgebra is trivial}, hence trivial. Similarly, if $\mathbb{K}=\mathrm{R}(\mathbb{G})$ is the solvable radical of $\mathbb{G}$, then its Lie superalgebra $\mathfrak{K}$ is solvable as well, hence it equals to zero. Again, $\mathbb{K}$ is purely even and \'etale, hence trivial. 	
	\end{proof}

	Let $(G-1)\mathfrak{G}_{\bar 1}$ denote the smallest $G$-submodule $\mathfrak{V}$ in $\mathfrak{G}_{\bar 1}$ such that $G$ acts trivially on $\mathfrak{G}_{\bar 1}/\mathfrak{V}$.
	The normality criterion immediately implies that $(G, (G-1)\mathfrak{G}_{\bar 1})$ represents a normal supersubgroup of $\mathbb{G}$. In particular, if $\mathbb{G}$ is 
	WAPS or SAPS, then there should be $(G-1)\mathfrak{G}_{\bar 1}=\mathfrak{G}_{\bar 1}$. 
	
	\begin{lm}\label{a very particular case}
		Assume that $G$ is a WAPS-group. Then $\mathbb{G}$ is a WAPS-supergroup if and only if the following conditions hold : 
		\begin{enumerate}
			\item $\mathfrak{G}_{\bar 1}=(G-1)\mathfrak{G}_{\bar 1}$;
			\item There are no nonzero $G$-submodules  $\mathfrak{W}$ of $\mathfrak{G}_{\bar 1}$, such that $[\mathfrak{G}_{\bar 1}, \mathfrak{W}]=0$.
		\end{enumerate}
		Similarly, if $G$ is SAPS (or equivalently, if $G$ is WAS), then $\mathbb{G}$ is SAPS if and only if it satisfies the condition $(1)$.
	\end{lm}
	\begin{proof}
		By the normality criterion, if $\mathbb{H}$ is a proper smooth connected normal supersubgroup of $\mathbb{G}$, then either $H=G, \mathfrak{H}_{\bar 1}\neq \mathfrak{G}_{\bar 1}$ and 
		$G=\ker(G\to\mathrm{GL}(\mathfrak{G}_{\bar 1}/\mathfrak{H}_{\bar 1}))$, or $H=1$ and $[\mathfrak{G}_{\bar 1}, \mathfrak{H}_{\bar 1}]=0$, by the normality criterion. The proof of the second statement is similar.  
	\end{proof}
	\begin{lm}\label{a converse}
		Assume that $\mathbb{G}$ is a smooth connected supergroup  and $\mathbb{G}/\mathbb{H}$ is a WAPS-supergroup, where $\mathbb{H}$ is a finite normal supersubgroup of $\mathbb{G}$.	Then $\mathbb{G}$ is WAPS
		if and only if the conditions $(1-2)$ of Lemma \ref{a very particular case} hold. 
	\end{lm}
	\begin{proof}
		Let $\mathbb{R}$ be a smooth connected normal supersubgroup of $\mathbb{G}$. Then $\mathbb{R}\mathbb{H}/\mathbb{H}$ is a smooth connected normal supersubgroup of $\mathbb{G}/\mathbb{H}$, hence either $\mathbb{R}\mathbb{H}=\mathbb{G}$ or $\mathbb{R}\leq\mathbb{H}$. In the first case, $\mathbb{G}/\mathbb{R}=\mathbb{R}\mathbb{H}/\mathbb{R}\simeq\mathbb{H}/(\mathbb{R}\cap\mathbb{H})$ is a finite smooth connected supergroup, hence it is purely odd by  lemma \ref{Lie superalgebra is trivial}.
		The latter implies $G=R$ and $G$ acts trivially on $\mathfrak{G}_{\bar 1}/\mathfrak{R}_{\bar 1}$, by the normality criteria. 
		In the second case $\mathbb{R}$ is purely odd, hence $[\mathfrak{G}_{\bar 1}, \mathfrak{R}_{\bar 1}]=0$.
	\end{proof}
	\begin{lm}\label{if quotient is SAPS}
		Assume that $\mathbb{H}$ is a finite normal supersubgroup of smooth connected supergroup $\mathbb{G}$. If $\mathbb{G}$ is SAPS and $\mathbb{G}$ is not solvable, then $\mathbb{G}/\mathbb{H}$ is SAPS too. The converse is true
		if and only if $\mathbb{G}$ satisfies the condition $(1)$ from Lemma \ref{a very particular case}.	
	\end{lm}
	
	\begin{proof}
		It is clear that any proper normal supersubgroup of $\mathbb{G}/\mathbb{H}$ is finite. So, all we need is to show that $\mathbb{G}/\mathbb{H}$ is noncommutative. Assume the contrary.
		Arguing as in Lemma \ref{when a quotient of SAS is SAS}, we conclude that $G$ is commutative. By \cite[Corollary 6.4]{maszub1} the supergroup $\mathbb{G}$ is solvable, a contradiction.
		
		If $\mathbb{G}/\mathbb{H}$ is SAPS and $\mathbb{R}$ is a proper infinite normal supersubgroup
		in $\mathbb{G}$, then the same arguments as in Lemma \ref{a converse} imply $R=G$, hence $G$ acts trivially on $\mathfrak{G}_{\bar 1}/\mathfrak{R}_{\bar 1}$. 
	\end{proof}
	
	\begin{pr}\label{SAPS in positive char}
		If $\mathrm{char}\Bbbk =p>0$ and $\mathbb{G}$ is smooth,  connected and not solvable, then $\mathbb{G}$ is SAPS if and only if $G$ is SAPS and $\mathbb{G}$ satisfies the condition $(1)$ from Lemma \ref{a very particular case}. Moreover, $\mathbb{G}$ is SAS
		if and only if $G$ is SAS and $\mathbb{G}$ satisfies the conditions $(1-2)$ from Lemma \ref{a very particular case}.
	\end{pr}
	\begin{proof}
		Assume that $\mathbb{G}$ is a SAPS-supergroup and $\mathbb{G}$ is not solvable. Since $\mathbb{G}_1$ is finite and $G$ is smooth, $\mathbb{G}/\mathbb{G}_1\simeq G/G_1\simeq G^{(1)}$ is a SAPS-group by Lemma \ref{if quotient is SAPS}. On the other hand, $\Bbbk[G^{(1)}]=\Bbbk[G]^{(1)}$ coincides with $\Bbbk[G]$ as a Hopf ring, but its
		$\Bbbk$-space structure is given by $a\cdot f=a^{\frac{1}{p}} f, a\in\Bbbk, f\in \Bbbk[G]$. In particular, any \emph{conormal} or \emph{cofinite} Hopf ideal in $\Bbbk[G]$ remains the same in its Frobenius twist $\Bbbk[G]^{(1)}$, and vice versa. In other words, if $H$ is a proper normal subgroup in $G$, then $H^{(1)}$ is the same in $G^{(1)}$, that implies $H$ is finite and $G$ is SAPS. Lemma \ref{a very particular case} infers the first statement.
		
		As above, one easily sees that a smooth connected algebraic group $R$ is solvable or semisimple if and only if $R^{(1)}$ is. Indeed, 
		$\Bbbk[R/\mathcal{D}(R)]$ is isomorphic to the largest cocommutative Hopf subalgebra of $\Bbbk[R]$. Since this property is invariant under the transition to $\Bbbk[R]^{(1)}$,
		we have $(R/\mathcal{D}(R))^{(1)}\simeq R^{(1)}/\mathcal{D}(R^{(1)})$. Using \cite[Corollary 6.19(b)]{milne}, we derive both statements.
		
		Assume that $\mathbb{G}$ is SAS.
		Let $\mathbb{H}$ be the preimage of $\mathrm{R}(G^{(1)})\neq e$ in $\mathbb{G}$.  Since
		$\mathrm{R}(G^{(1)})$ is infinite,  $\mathbb{H}$ is infinite too, that in turn infers $\mathbb{H}=\mathbb{G}$. Thus $\mathrm{R}(G^{(1)})\simeq \mathbb{H}/\mathbb{G}_1\simeq G^{(1)}$ is solvable, hence $G$ is. By \cite[Corollary 6.4]{maszub1} the supergroup $\mathbb{G}$ is solvable as well, which is a contradiction. It remains to note that if the codition $(2)$ from Lemma \ref{a very particular case} is violated, then $\mathbb{G}$ contains
		a normal purely odd supersubgroup, which is smooth, connected and abelian. 
		
		Conversely, if $G$ is SAS and $\mathbb{G}$ satisfies conditions $(1-2)$ from  Lemma \ref{a very particular case}, then $\mathbb{G}$ is at least SAPS by Lemma \ref{if quotient is SAPS}. Let $\mathbb{H}=\mathrm{R}(\mathbb{G})$. Then $H\leq \mathrm{R}(G)=e$, hence $[\mathfrak{G}_{\bar 1}, \mathfrak{H}_{\bar 1}]=0$ implies
		$\mathfrak{H}_{\bar 1}=0$.
	\end{proof}
	
	\subsection{Projective special linear supergroups}
	The following lemma is folklore, however, we found it important to prove it.
	\begin{lm}\label{when sl is simple}
		If $\mathrm{char}\Bbbk =p>0$ and $m+n\geq 3$, then the Lie superalgebra $\mathfrak{sl}(m|n)$ is simple if and only if $p$ does not divide $m-n$. If $p$  divides $m-n$, then $\mathfrak{sl}(m|n)$ contains the unique proper
		central superideal $\Bbbk I_{m+n}$, so that the Lie superalgebra $\mathfrak{psl}(m|n)=\mathfrak{sl}(m|n)/\Bbbk I_{m+n}$ is simple.
	\end{lm}
	\begin{proof}
		Let $\mathfrak{I}$ be a non-zero superideal in $\mathfrak{sl}(m|n)$. We show that either $\mathfrak{I}=\mathfrak{sl}(m|n)$, or $\mathfrak{I}$ is contained in the subspace of diagonal matrices. 
		
		So, let $z=\sum_{1\leq k, l\leq m+n}a_{kl}E_{kl}\in\mathfrak{I}$ and $a_{ij}\neq 0$ for some $i\neq j$. We have 
		\[
		\begin{array}{cll}z_1=[z, E_{ji}] & = &
			(a_{jj}-a_{ii})E_{ji}+a_{ij}(E_{ii}-(-1)^{|i|+|j|}E_{jj})\\[2mm]
			&& \displaystyle +\sum_{k\neq i, j}a_{kj}E_{ki}-\sum_{l\neq i, j} (-1)^{(|i|+|j|)(|i|+|l|)}a_{il}E_{jl}.
		\end{array}
		\]
		If $|i|+|j|={\bar 0} \pmod 2$, then $[z_1, E_{ji}]=-2a_{ij}E_{ji}\in\mathfrak{I}$, hence $E_{ji}\in\mathfrak{I}$. 
		Otherwise, if $|i|+|j|={\bar 1}  \pmod 2$, then choose $s\neq i, j$ so that  $z_2=[z_1, E_{si}]=-a_{ij}E_{si}-(-1)^{(|i|+|j|)(|i|+|s|)}a_{is}E_{ji}\in\mathfrak{I}$, and
		$[E_{ii}+E_{jj}, z_2]=a_{ij}E_{si}$. It follows that  $E_{si}\in\mathfrak{I}$. Thus, $\pm [E_{si}, E_{js}]=E_{ji}\in\mathfrak{I}$ as well. To sum it all up, if $z$ contains a nonzero term
		$a_{ij} E_{ij}$ with $i\neq j$, then $E_{ji}$ belongs to $\mathfrak{I}$, and by the same reasoning, $E_{ij}$ does as well.
		
		Now, let $E_{kl}\in\mathfrak{I}$, such that $k\neq l$. We claim that arbitrary $E_{uv}$, for $u\neq v$, belongs to $\mathfrak{I}$. Indeed, if $\{k, l\}\cap\{u, v\}=\emptyset$, then
		$[[E_{vk}, E_{kl}], E_{lu}]=E_{vu}\in\mathfrak{I}$. Otherwise, and without loss of generality, we can assume that $k=u$, for $l\neq v$. In this case, $[E_{kl}, E_{lv}]=E_{kv}\in\mathfrak{I}$.
		Finally, if $\mathfrak{I}$ contains all $E_{kl}$, for $k\neq l$, then it must contain all $[E_{kl}, E_{lk}]=E_{kk}-(-1)^{|k|+|l|}E_{ll}$, hence $\mathfrak{I}=\mathfrak{sl}(m|n)$. 
		
		It remains to consider the case where $\mathfrak{I}$ is contained in the subspace of diagonal matrices. If $z=\sum_{1\leq k\leq m+n}a_k E_{kk}$, then for any couple of indices
		$i\neq j$, we have $[z, E_{ij}]=(a_{i}-a_{j})E_{ij}\in\mathfrak{I}$, hence $a_{i}=a_{j}$. The lemma is proved. 
	\end{proof}
	\begin{rem}\label{m=n=1}
		If $m=n=1$, then the Lie superalgebra $\mathfrak{sl}(1|1)$ is obviously solvable. Note also that $[\mathfrak{gl}(m|n), \mathfrak{gl}(m|n)]=\mathfrak{sl}(m|n)$ for arbitrary $m, n$.   
	\end{rem}
	\begin{lm}\label{some exact sequence}
		Let \[1\to R\to G\to H\to 1\] be an exact sequence of algebraic groups. If $R$ and $G$ are smooth, then the induced sequence of Lie algebras
		\[0\to\mathrm{Lie}(R) \to \mathrm{Lie}(G)\to \mathrm{Lie}(H)\to 0\]	
		is exact too. The similar statement is valid for supergroups.
	\end{lm}
	\begin{proof}
		It is known that the above sequence is at least left exact. In addition, $H$ is smooth by \cite[Proposition 1.62]{milne}. It remains to combine the fact that an algebraic group $S$ is smooth if and only if $\dim S=\dim\mathrm{Lie}(S)$ (cf. \cite[Corollary 12.2]{water}) with $\dim G=\dim R+\dim H$ (cf. \cite[Proposition 5.23]{milne}).	
		The second statement is obvious being reformulated in terms of Harish-Chandra pairs (cf. \cite[Theorem 12.11]{maszub}).
	\end{proof}
	
	We suppose that $m, n>0$ and $m+n\geq 3$, unless stated otherwise. The supergroup $\mathrm{SL}(m|n)$ contains the central (purely even) supersubgroup consisting of scalar matrices, that is isomorphic to $\mu_{m-n}$. Recall that $\mu_l$ is the group of roots of unity of order $l\geq 1$, that is $\mu_l(A)=\{a\in A_{\bar 0}\mid a^l=1\}, A\in\mathsf{SAlg}_{\Bbbk}$. Slightly abusing notations, let $\mu_0$ denote the one-dimensional torus $\mathbb{G}_m\simeq\mathrm{G}_m$.
	
	The supergroup
	$\mathrm{PSL}(m|n)=\mathrm{SL}(m|n)/\mu_{m-n}$ is called a \emph{projective special linear} supergroup. Its even part can be described as follows.
	Let $\chi : \mathrm{GL}_m\times \mathrm{GL}_n\to \mathrm{G}_m$ be a character, determined by 
	\[(A, B)\mapsto \det(A)\det(B)^{-1}, \text{ where } (A, B)\in \mathrm{GL}_m\times \mathrm{GL}_n .\]
	Then $\mathrm{SL}(m|n)_{ev}=\ker\chi$ and $\mathrm{PSL}(m|n)_{ev}\simeq \ker\chi/\mu_{m-n}$. Furthermore, we have $\mathrm{SL}_m\times \mathrm{SL}_n\unlhd \mathrm{SL}(m|n)_{ev}$,
	so that $\mathrm{SL}(m|n)_{ev}/(\mathrm{SL}_m\times \mathrm{SL}_n)\simeq \mathrm{G}_m$. In particular, $\mathrm{SL}(m|n)$ is smooth and connected (use \cite[Proposition 1.62 and Proposition 8.1]{milne}), hence $\mathrm{PSL}(m|n)$ is too. 
	
	We have the commutative diagram 
	\[\begin{array}{ccccccccc}
		1 & \to & D & \to & \mathrm{GL}(m|n) & \to & \mathrm{PGL}(m|n) & \to & 1 \\
		&     & \uparrow & & \uparrow     &     &  \uparrow      &     & \\
		1 & \to & \mu_{m-n} & \to & \mathrm{SL}(m|n) & \to & \mathrm{PSL}(m|n) & \to & 1 	
	\end{array}\]
	with exact rows, where $D\simeq \mathbb{G}_m$ is the center of $\mathrm{GL}(m|n)$ consisting of scalar  matrices. The vertical arrows are closed embeddings by
	\cite[Lemma 10.2]{bz}. The supergroup $\mathrm{PGL}(m|n)$ is also smooth and connected, and we have
	$\dim\mathrm{PGL}(m|n)_{ev}=m^2+n^2-1$. Since $\dim\mu_{m-n}=0$ provided $m\neq n$, otherwise $\dim\mu_0=1$, we have $\dim\mathrm{PSL}(m|n)_{ev}=m^2+n^2-1$ provided $m\neq n$, otherwise $\dim\mathrm{PSL}(m|n)_{ev}=m^2+n^2-2$. By \cite[Theorem 12.4]{water},  $\mathrm{PGL}(m|n)\simeq \mathrm{PSL}(m|n)$ whenever $m\neq n$. In particular, in this case
	$\mathrm{Lie}(\mathrm{PSL}(m|n))\simeq\mathfrak{pgl}(m|n)$. Finally, if $m=n$, then $\mu_0=D$ is smooth, hence 
	$\mathrm{Lie}(\mathrm{PSL}(m|n))\simeq\mathfrak{psl}(m|n)$.
	
	Let $H$
	denote the normal smooth connected subgroup 
	$(\mathrm{SL}_m\times\mathrm{SL}_n)\mu_{m-n}/\mu_{m-n}\simeq (\mathrm{SL}_m\times\mathrm{SL}_n)/\mu_{\mathrm{gcd}(m, n)}$ of ${PSL}(m|n)_{ev}$,
	which has codimension one if and only if $m\neq n$, otherwise they are the same.
	Recall that if $G$ is an algebraic group, then $X(G)$ denotes its character group.
	\begin{lm}\label{X(PSL)}
		Let $m\neq n$. Any character 
		of $\mathrm{PSL}(m|n)_{ev}$ is induced by a character 
		of $\mathrm{GL}_m\times\mathrm{GL}_n$, defined as $\eta_l : (A, B)\mapsto \det(A)^{ln'}\det(B)^{-l m'},$ where $m'=\frac{m}{\mathrm{gcd}(m, n)}, \ n'=\frac{n}{\mathrm{gcd}(m, n)}, l\in\mathbb{Z}$. In particular, $H=\ker\eta_1$.
	\end{lm}
	\begin{proof}
		Since $\mathrm{PSL}(m|n)_{ev}\simeq (\mathrm{GL}_m\times\mathrm{GL}_n)/D$, the character group $X(\mathrm{PSL}(m|n)_{ev})$
		is naturally identified with the subgroup
		$\{\eta\in X(\mathrm{GL}_m\times\mathrm{GL}_n)\mid \eta(D)=1\}$. Note that any $\chi\in X(\mathrm{GL}_m\times\mathrm{GL}_n)$ has a form 
		\[\chi((A, B))=\det(A)^k\det(B)^s, k, s\in\mathbb{Z},\]
		hence $\chi(D)=1$ if and only if $km+sn=0$ if and only if $\chi=\eta_l$ for some $l\in\mathbb{Z}$.
		Since $\mathrm{PSL}(m|n)_{ev}/H\simeq\mathrm{G}_m$, there is $l\in\mathbb{Z}\setminus 0$, such that 
		$H=\ker\eta_l$. However, $\ker\eta_l/\ker\eta_1\simeq\mu_l$ and Lemma \ref{Lie superalgebra is trivial} imply $l=1$.
	\end{proof}
	\begin{lm}\label{superideals in pgl}
		The Lie superalgebra $\mathfrak{pgl}(m|n)$ is either simple, or its unique proper nonzero  superideal coincides with $[\mathfrak{pgl}(m|n), \mathfrak{pgl}(m|n)]\simeq\mathfrak{psl}(m|n)$. 
	\end{lm}	
	\begin{proof}
		Let $\mathfrak{I}$ be a proper nonzero superideal of
		$\mathfrak{pgl}(m|n)$. Then $\mathfrak{I}\simeq\mathfrak{J}/\Bbbk I_{m+n}$, where $\mathfrak{J}$ is a proper nonzero superideal of $\mathfrak{gl}(m|n)$. If $\mathfrak{J}\cap\mathfrak{sl}(m|n)=0$ or $\mathfrak{J}\cap\mathfrak{sl}(m|n)=\Bbbk I_{m+n}$, then $\mathfrak{J}_{\bar 1}=0$ implies $[\mathfrak{J}_{\bar 0}, \mathfrak{gl}(m|n)_{\bar 1}]=0$. Since $\mathfrak{J}_{\bar 0}\neq\Bbbk I_{m+n}$, reasoning as in Lemma \ref{when sl is simple} infers that there is $E_{kl}\in\mathfrak{J}_{\bar 0}$ with $k\neq l$. Choosing an element $E_{lu}\in\mathfrak{gl}(m|n)_{\bar 1}$, we obtain $[E_{kl}, E_{lu}]=E_{ku}\neq 0$, a contradiction.    
	\end{proof}
	\begin{lm}\label{filtration of pgl_0}
		Let $\mathrm{char}\Bbbk=p>0$ and $p|(m-n), m\neq n$. Then \[[\mathfrak{pgl}(m|n)_{\bar 0}, \mathfrak{pgl}(m|n)_{\bar 0}]=[\mathfrak{pgl}(m|n), \mathfrak{pgl}(m|n)]_{\bar 0}\simeq\mathfrak{psl}(m|n)_{\bar 0}\]    
		if and only if $p\nmid \mathrm{gcd}(m, n)$. Otherwise, $[\mathfrak{pgl}(m|n)_{\bar 0}, \mathfrak{pgl}(m|n)_{\bar 0}]$ has codimension two in $\mathfrak{pgl}(m|n)_{\bar 0}$.
	\end{lm}
	\begin{proof}
		Let $p|(m-n)$ and $p\nmid \mathrm{gcd}(m, n)$. The latter is equivalent to $p\nmid m, p\nmid n$. Using this remark, one easily sees that $\mathfrak{sl}(m|n)_{\bar 0}\simeq \mathfrak{sl}_m\oplus\mathfrak{sl}_n\oplus \Bbbk I_{m+n}$. In particular, $\mathfrak{psl}(m|n)_{\bar 0}\simeq\mathfrak{sl}_m\oplus\mathfrak{sl}_n\simeq [\mathfrak{pgl}(m|n)_{\bar 0}, \mathfrak{pgl}(m|n)_{\bar 0}]$.
		If $p|m, p|n$, then $[\mathfrak{pgl}(m|n)_{\bar 0}, \mathfrak{pgl}(m|n)_{\bar 0}]\simeq (\mathfrak{sl}_m\oplus\mathfrak{sl}_n)/\Bbbk I_{m+n}$ has codimension two in $\mathfrak{pgl}(m|n)_{\bar 0}$.
	\end{proof}
	\begin{pr}\label{when PSL is SAS or WAS}
		The following statements take place:
		\begin{enumerate}
			\item If $\mathrm{char}\Bbbk =0$, then $\mathrm{PSL}(m|n)$ is a SAS-supergroup;
			\item If $\mathrm{char}\Bbbk=p>0$ and $p\nmid  (m-n)$ or $m=n$, then $\mathrm{PSL}(m|n)$ is a WAS-supergroup, but not SAS;
			\item If $m\neq n, p|(m-n)$ and $p\nmid \ \mathrm{gcd}(m, n)$, then the derived supersubgroup $\mathcal{D}(\mathrm{PSL}(m|n))$ is a WAS-supergroup, which is also SAS if and only if $m=1$ or $n=1$;
			\item If $m\neq n$ and $p|m, p|n$, then $\mathrm{PSL}(m|n)$ is a WAS-supergroup if and only if $p\nmid (m'-n')$, where $m'=\frac{m}{\mathrm{gcd}(m, n)}, \ n'=\frac{n}{\mathrm{gcd}(m, n)}$; otherwise,  $\mathcal{D}(\mathrm{PSL}(m|n))$ is a WAS-supergroup, but both are not SAS. 
		\end{enumerate}
	\end{pr}
	\begin{proof}
		Under conditions of $(1)$ or $(2)$, $\mathrm{Lie}(\mathrm{PSL}(m|n))$ is always simple, hence these cases follow by Lemma \ref{some simple case}. Besides, if $p>0$ and $m\neq n$, then $\mathrm{PSL}(m|n)$ is not SAS, since 
		$\mathrm{PSL}(m|n)_{ev}$ contains a normal smooth connected subgroup $H$ of codimension one. Further we assume that $m\neq n$.
		
		By Lemma \ref{filtration of pgl_0}, the induced Lie algebra morphism $\mathrm{Lie}(H)\to \mathrm{Lie}(\mathrm{PSL}(m|n)_{ev})\simeq\mathfrak{pgl}(m|n)_{\bar 0}$ is an isomorphism onto $\mathfrak{psl}(m|n)_{\bar 0}$. Therefore, 
		by the normality criterion, $(H, \mathfrak{pgl}(m|n)_{\bar 1})$ is a Harish-Chandra subpair, which determines a normal smooth connected supersubgroup $\mathbb{H}$ in  
		$\mathrm{PSL}(m|n)$. Since $\mathrm{PSL}(m|n)/\mathbb{H}\simeq\mathrm{G}_m$ is abelian,
		$\mathcal{D}(\mathrm{PSL}(m|n))\leq\mathbb{H}$. 
		On the other hand, 
		the natural inclusion $[\mathfrak{pgl}(m|n), \mathfrak{pgl}(m|n)]\subseteq \mathrm{Lie}(\mathcal{D}(\mathrm{PSL}(m|n)))$ implies $\mathrm{Lie}(\mathcal{D}(\mathrm{PSL}(m|n)))=\mathfrak{H}\simeq\mathfrak{psl}(m|n)$. 
		Furthermore, the similar inclusions \[[\mathfrak{pgl}(m|n)_{\bar 0}, \mathfrak{pgl}(m|n)_{\bar 0}]\subseteq\mathrm{Lie}(\mathcal{D}(\mathrm{PSL}(m|n)_{ev}))\]
		and
		\[\mathcal{D}(\mathrm{PSL}(m|n)_{ev})\leq \mathcal{D}(\mathrm{PSL}(m|n))_{ev},\]
		combined with the fact that $\mathcal{D}(\mathrm{PSL}(m|n)_{ev})$ is smooth and connected, infer 
		\[\mathcal{D}(\mathrm{PSL}(m|n)_{ev})=H=\mathcal{D}(\mathrm{PSL}(m|n))_{ev}. \]
		Summing all up, we obtain $\mathcal{D}(\mathrm{PSL}(m|n))=\mathbb{H}$ is a WAS-supergroup. It remains to note that $H$ is a SAS-group if and only if $m=1$ or $n=1$, so that $H\simeq\mathrm{SL}_n$ or $H\simeq\mathrm{SL}_m$ respectively.
		
		Finally, let $p|m, p|n$ and $\mathbb{R}$ be a normal smooth connected supersubgroup in $\mathrm{PSL}(m|n)$. If $\mathbb{R}$ is proper and $\mathbb{R}\neq e$, then Lemma \ref{superideals in pgl} implies $\mathfrak{R}\simeq\mathfrak{psl}(m|n)$. The group $\mathrm{PSL}(m|n)/\mathbb{R}\simeq \mathrm{PSL}(m|n)_{ev}/R$ is one-dimensional, smooth and connected, hence isomorphic to $\mathrm{G}_m$ or $\mathrm{G}_a$ (see \cite[Corollary 16.16]{milne}). 
		Since \[(\mathrm{GL}_m\times \mathrm{GL}_n)/\mathcal{D}(\mathrm{GL}_m\times \mathrm{GL}_n)\simeq\mathrm{G}_m\times\mathrm{G}_m\] is a torus, the second case is impossible. 
		As in Lemma \ref{X(PSL)}, we conclude that $R=ker\eta_1=H$. Furthermore, $\mathfrak{R}_{\bar 0}=\ker d\eta_1$, where
		$d\eta_1((a, b))=n'\mathrm{tr}(a)-m'\mathrm{tr}(b), a\in\mathfrak{gl}_m, b\in \mathfrak{gl}_m$, so that $\mathfrak{psl}(m|n)_{\bar 0}\subseteq\mathfrak{R}_{\bar 0}$ if and only if $p|(m'-n')$. Reasoning as above, one sees that the Harish-Chandra subpair $(H, \mathfrak{pgl}(m|n)_{\bar 1})$ determines the derived supersubgroup $\mathcal{D}(\mathrm{PSL}(m|n))$, which is WAS bot not SAS.
	\end{proof}
	\begin{cor}\label{WAS not equal to SAS}
		If $\mathrm{char}\Bbbk=p>0$, then the class of SAS-supergroups is properly contained in the class of WAS-supergroups. 	
	\end{cor}

	\subsection{More properties of SAS-supergroups}
	
	\begin{pr}\label{SAS in zero char}
		Let $\mathbb{G}$ be a connected algebraic supergroup. If $\mathrm{char}\Bbbk=0$, then $\mathbb{G}$ is SAS if and only if $\mathfrak{G}$ is simple.	
	\end{pr}
	\begin{proof}
		If $\mathbb{G}$ is a SAS-supergroup, then by \cite[Proposition 11.1]{grzub} the Lie superalgebra $\mathfrak{G}$ is semisimple. Furthermore,  by \cite[Theorem 11.8]{grzub}, $\mathbb{G}$ is inserted in a \emph{sandwich pair}. More precisely, there are simple Lie superalgebras $\mathfrak{U}_1, \ldots, \mathfrak{U}_t$, some nonnegative integers $n_1, \ldots, n_t$ and 
		connected supergroups $\mathbb{U}\unlhd\mathbb{H}$, such that \[\mathfrak{U}=\oplus_{1\leq i\leq t}\mathfrak{U}_i\otimes\Lambda(n_i)\simeq\mathrm{InDer}(\mathfrak{U})\subseteq \mathfrak{H}=\mathrm{Der}(\mathfrak{U})=\]
		\[(\oplus_{1\leq i\leq t}\mathrm{Der}(\mathfrak{U}_i)\otimes\Lambda(n_i))\oplus (\oplus_{1\leq i\leq t}\mathrm{id}_{\mathfrak{U}_i}\otimes\mathrm{Der}(\Lambda(n_i)).\] Moreover, there is a supergroup morphism $\mathbb{G}\to\mathbb{H}$, such that its kernel $\mathbb{R}$ is finite and $\mathfrak{U}\subseteq \mathfrak{G}\subseteq\mathfrak{H}$ (since $\mathfrak{R}=0$, we identify $\mathfrak{G}$ with its image in $\mathfrak{H}$). Without loss of generality, one can assume that $\mathbb{U}\leq\mathbb{G}\leq\mathbb{H}$. Since $\mathbb{U}\neq e$, we have $\mathbb{U}=\mathbb{G}$. If $n_i>1$ for some $1\leq i\leq t$, then the projection of $\mathfrak{U}$ to $\mathrm{id}\otimes\mathrm{Der}(\Lambda(n_i))$ is zero, which contradicts \cite[Theorem 6(a)]{kac}. Finally, if $\mathfrak{U}=\oplus_{1\leq i\leq t}\mathfrak{U}_i$ and $t>1$, then each $\mathfrak{U}_i$ is the Lie superalgebra of some proper connected normal supersubgroup of $\mathbb{G}$ by \cite[Lemma 11.3]{grzub}, which contradicts the fact that $\mathbb{G}$ is a SAS-supergroup. \end{proof}
	
	\begin{rem}\label{Lie algebra of SAS, an example}{\rm 
			If $G$ is a SAS-group, then $\mathrm{Lie}(G)$ is not always simple. 
			Let us consider the commutative diagram with the exact rows :
			\[\begin{array}{ccccccccc}
				1 & \to & D & \to & \mathrm{GL}_n & \to & \mathrm{PGL}_n & \to & 1 \\
				&     & \uparrow & & \uparrow     &     &  \uparrow      &     & \\
				1 & \to & \mu_n & \to & \mathrm{SL}_n & \to & \mathrm{PSL}_n & \to & 1 	
			\end{array}\]
			where $D\simeq G_m$ is the center of $\mathrm{GL}_n$ consisting of scalar matrices. Besides, the vertical arrows are closed embeddings (cf. \cite[Theorem 5.34]{milne}), and $\mathrm{PSL}_n\to\mathrm{PGL}_n$ is an isomorphism. In fact, $\dim\mathrm{PSL}_n=\dim\mathrm{SL}_n=\dim\mathrm{PGL}_n$ and all these groups are smooth and connected. Therefore, \cite[Theorem 12.4]{water} implies the statement. 
			
			Lemma \ref{some exact sequence} infers $\mathrm{Lie}(\mathrm{PGL}_n)\simeq \mathfrak{gl}_n/\Bbbk I_n=\mathfrak{pgl}_n$.
			Furthermore, we have a commutative diagram of Lie algebras with the exact upper row
			\[\begin{array}{ccccccccc}
				0 & \to & \Bbbk I_n & \to & \mathfrak{gl}_n & \to & \mathfrak{pgl}_n & \to & 0 \\
				&     & \uparrow & & \uparrow     &     &  \uparrow      &     & \\
				0 & \to & \mathrm{Lie}(\mu_n) & \to & \mathfrak{sl}_n & \to & \mathrm{Lie}(\mathrm{PSL}_n) &  & 	
			\end{array}\]
			The right exactness of the bottom horizontal row depends on $char\Bbbk$. As has already been observed, $\mathrm{Lie}(\mu_n)=0$ if and only if $char\Bbbk=0$ or $char\Bbbk=p>0$ and $p\nmid n$. In this case $\mathfrak{sl}_n\simeq \mathrm{Lie}(\mathrm{PSL}_n)\simeq\mathfrak{pgl}_n$. If $p|n$, then $\mathrm{Lie}(\mu_n)=\Bbbk I_n$ and $\mathfrak{psl}_n=\mathfrak{sl}_n/\Bbbk I_n$ is a proper (simple) Lie subalgebra in $\mathrm{Lie}(\mathrm{PSL}_n)$. In other words, $\mathrm{SL}_n\to\mathrm{PSL}_n$ is an {\bf inseparable isogeny} (meaning that the quotient morphism $\mathrm{SL}_n\to \mathrm{PSL}_n$ is not separable, i.e. the induced morphism of Lie algebras $\mathfrak{sl}_n\to\mathfrak{pgl}_n$ is not surjective). Arguing as in Subsection 2.1, one can derive that $\mathfrak{psl}_n$ is an ideal in $\mathrm{Lie}(\mathrm{PSL}_n)$ (of codimension one), that coincides with
			$[\mathfrak{pgl}_n, \mathfrak{pgl}_n]$.
			
			Recall that $\mathrm{SL}_n$ is a SAS-group (cf. \cite[Summary 24.66(A)]{milne}). Thus, by Lemma \ref{when a quotient of SAS is SAS}, $\mathrm{PSL}_n$ is also a SAS-group. But the Lie algebras of both groups are not simple, provided $p|n$. Moreover, if $n\neq p$,  the whole situation can be more curious. 
			
			Let $n'$ be a divisor of $n$ and $p|n$. The group $G\simeq \mathrm{SL}_n/\mu_{n'}$ is again SAS.  We have the  exact sequence
			\[ 1\to\mu_n/\mu_{n'}\to G\to \mathrm{PSL}_n\to 1\]
			and the following alternatives:
			\begin{enumerate}
				\item $p\not |\frac{n}{n'}$,  $\mathfrak{g}\simeq\mathfrak{pgl}_n$ ;
				\item $p\not | n'$,  $\mathfrak{g}\simeq\mathfrak{sl}_n$;
				\item $p| n', p|\frac{n}{n'}$,  $\mathfrak{g}=\mathfrak{z}\oplus [\mathfrak{g}, \mathfrak{g}],$ where $\mathfrak{z}$ is the one-dimensional center of $\mathfrak{g}$ and $[\mathfrak{g}, \mathfrak{g}]\simeq\mathfrak{psl}_n$.
			\end{enumerate} 
			Only the last alternative is not obvious.  We have $\mathfrak{z}=\mathrm{Lie}(\mu_n/\mu_{n'})$ is the one-dimensional center of $\mathfrak{g}$, such that $\mathfrak{g}/\mathfrak{z}$ is an ideal in $\mathfrak{pgl}_n$ of codimension one, hence it coincides with $[\mathfrak{pgl}_n, \mathfrak{pgl}_n]\simeq\mathfrak{psl}_n$.  On the other hand,  the inseparable isogeny $\mathrm{SL}_n\to G$ induces a morphism $\mathfrak{sl}_n\to\mathfrak{g}$ of Lie algebras, whose image is an ideal of $\mathfrak{g}$ of codimension one,  also isomorphic to $\mathfrak{psl}_n$.  Summing all up, we obtain the statement.} 
	\end{rem}
	
	The situation described above is typical. 
	\begin{pr}\label{Lie algebra of SAS}
		Let $G$ be a SAS-group of type different from $\mathrm{G}_2$, or $G$ is of type $\mathrm{G}_2$ and $\mathrm{char}\Bbbk\neq 3$. Let $\mathfrak{g}$  be its Lie algebra. Then one of the following alternatives hold:
		\begin{enumerate}
			\item $\mathfrak{g}$ is simple;
			\item the unique proper nonzero ideal of $\mathfrak{g}$ is its one-dimensional center $\mathfrak{z}$, so that $\mathfrak{g}/\mathfrak{z}$ is a simple Lie algebra;
			\item the unique proper nonzero ideal of $\mathfrak{g}$ is its derived ideal $[\mathfrak{g}, \mathfrak{g}]$ of codimension one, which is also simple as a Lie algebra;
			\item $\mathfrak{g}=\mathfrak{z}\oplus [\mathfrak{g}, \mathfrak{g}],$ where $\mathfrak{z}$  is the one-dimensional center and the derived ideal $[\mathfrak{g}, \mathfrak{g}]$ is simple as a Lie algebra.
		\end{enumerate}
		Finally, if $G$ of type $\mathrm{G}_2$ and $\mathrm{char}\Bbbk=3$, then $\mathfrak{g}$ contains an ideal $I$, such that $I\simeq \mathfrak{g}/I\simeq\mathfrak{psl}_3$. 
	\end{pr}
	\begin{proof}
		There is an \emph{universal covering}
		$\widetilde{G}\to G$, where $\widetilde{G}$ is a smooth connected simply connected group (see \cite[Definition 18.7 and Remark 18.27]{milne}). Since $\Bbbk$ is algebraically closed, $G\simeq\widetilde{G}/D$, where $D$ is a finite diagonalizable group. The group $\widetilde{G}$ is obviously SAS and $D$ is central (see \cite[Corollary 12.37]{milne}).
		By \cite[Theorem 23.62]{milne}, $G$ and $\widetilde{G}$ have the isomorphic indecomposable root systems, or equivalently, their root systems have the same \emph{Dynkin diagrams}. Let $\widehat{G}$ be a \emph{Chevalley group}, determined by this Dynkin diagram (cf. \cite[Theorem 23.72]{milne}). Note that $\mathrm{Lie}(\widehat{G})=\mathfrak{g}(\mathbb{Z})\otimes_{\mathbb{Z}}\Bbbk$, where $\mathfrak{g}$ is the simple Lie algebra over $\mathbb{C}$, corresponding to the aforementioned diagram, and $\mathfrak{g}(\mathbb{Z})$ is its $\mathbb{Z}$-form (cf. \cite[Theorem 4.1.1]{strade}). By \cite[Theorem 23.62]{milne}, there is a central isogeny $\widetilde{G}\to \widehat{G}$ with the kernel $D'$. 
		
		If the Dynkin diagram of $\widetilde{G}$ is $\mathrm{A}_n$, then $\widetilde{G}\simeq\mathrm{SL}_{n+1}$ and this case has been already considered. 
		
		Assume that 
		this diagram is different from $\mathrm{E}_6$. In this case the center $\mathrm{Z}(\widetilde{G})$ is isomorphic to one of the group in the list $\{e, \mu_2, \mu_2\times\mu_2, \mu_4\}$ (see \cite[24.c]{milne}), hence $\mathrm{Lie}(D)=\mathrm{Lie}(D')=0$ and $\mathrm{Lie}(\widetilde{G})\simeq \mathrm{Lie}(\widehat{G})\simeq\mathrm{Lie}(G)$ is simple, except the case
		$G$ is of type $\mathrm{G}_2$ and $\mathrm{char}\Bbbk=3$ (see \cite[Sections 4.1 and 4.4]{strade}). 
		
		If this diagram is $\mathrm{E}_6$, then $\mathrm{Z}(\widetilde{G})\simeq\mu_3$. In particular, if $\mathrm{char}\Bbbk >3$, then again $\mathrm{Lie}(\widetilde{G})\simeq \mathrm{Lie}(\widehat{G})\simeq\mathrm{Lie}(G)$ is simple. Finally, let $\mathrm{char}\Bbbk=3$. If $D'\neq e$, then $D'\simeq\mu_3$ and $\mathfrak{d}=\mathrm{Lie}(D')$ is a  one-dimensional central ideal in $\mathrm{Lie}(\widetilde{G})$. As above, it follows that $\mathrm{Lie}(\widetilde{G})/\mathfrak{d}$ is an ideal in $\mathrm{Lie}(\widehat{G})$ of codimension one, that obviously contradicts to \cite[4.4(A)]{strade}. Thus $\widetilde{G}\simeq\widehat{G}$ and either $G\simeq\widehat{G}$ satisfies $(2)$, or we have the inseparable isogeny $\widehat{G}\to G$, that corresponds to the alternative $(3)$. 
	\end{proof}
	\begin{rem}\label{when g is not simple}
		Note that the alternatives $(2), (3), (4)$ take place only if $G$ is of type $\mathrm{A}_n$, where $p|(n+1),$ or $\mathrm{E}_6$ and $p=3$.    
	\end{rem}
	\begin{pr}\label{not algebraic}
		Let $\mathfrak{g}$ be a Lie algebra isomorphic to $\mathfrak{psl}_n,$ where $p|n$, or to $\mathfrak{e}_6/\mathfrak{z}$, where $p=3$ and $\mathfrak{z}$ is the one-dimensional center of $\mathfrak{e}_6$. Then there is no smooth connected algebraic group $G$, such that $\mathrm{Lie}(G)\simeq \mathfrak{g}$.   \end{pr}
	\begin{proof}
		Assume the contrary, say, $p|n$ and $\mathfrak{g}\simeq\mathfrak{psl}_n$ is the Lie algebra of some smooth connected group $G$. Since $\mathfrak{g}$ is simple,  Lemma \ref{some simple case} implies that $G$ is almost-simple, hence $G$ is uniquely (up to a central isogeny) determined by its indecomposable root system.  
		
		Recall that if $H$ is a reductive group and $T$ is its maximal torus, then 
		$\mathfrak{t}=\mathrm{Lie}(T)$ is a Cartan subalgebra of $\mathfrak{h}=\mathrm{Lie}(H)$. Let $\mathfrak{t}_n$ denote the subalgebra of $\mathfrak{sl}_n$ consisting of all traceless diagonal matrices. By the above remark, it is a Cartan subalgebra in $\mathfrak{sl}_n$. Since $\Bbbk I_n$ is a central  ideal in $\mathfrak{sl}_n$, $\mathfrak{st}_n=\mathfrak{t}_n/\Bbbk I_n$ is a Cartan subalgebra of $\mathfrak{psl}_n$ also. Furthermore, since the dimensions of Cartan subalgebras are the same, the rank of root system of $G$ is $n-2$. It remains to note that there is no indecomposable root system of rank $n-2$ with $(n^2-2)-(n-2)=n^2-n$ roots, except the case $n-2=2$ and the system $\mathrm{G}_2$. But the condition $p|n$ infers $p=2$, which contradicts our assumption. 
		
		The case $\mathfrak{g}\simeq\mathfrak{e}_6/\mathfrak{z}$ is similar and we leave it to the reader.
	\end{proof}
	
	\begin{pr}\label{Lie superalgebra of SAS}
		Let $\mathrm{char}\Bbbk=p >0$. If $\mathbb{G}$ is a not purely even SAS-supergroup  with $\mathfrak{G}_{\bar 0}$ being  simple, then:
		\begin{enumerate}
			\item Any nonzero superideal of $\mathfrak{G}$ contains $[\mathfrak{G}, \mathfrak{G}]=\mathfrak{G}_{\bar 0}\oplus [\mathfrak{G}_{\bar 0}, \mathfrak{G}_{\bar 1}]$, where $[\mathfrak{G}_{\bar 0}, \mathfrak{G}_{\bar 1}]\neq 0$;
			\item If $\mathfrak{S}$ is a nonzero $\mathfrak{G}_{\bar 0}$-submodule in $\mathfrak{G}_{\bar 1}$, then $[\mathfrak{G}_{\bar 1}, \mathfrak{S}]=\mathfrak{G}_{\bar 0}$; 
			\item If $\mathfrak{G}_{\bar 1}/[\mathfrak{G}_{\bar 0}, \mathfrak{G}_{\bar 1}]\neq 0$ (equivalently, if $\mathfrak{G}$ is not simple), then any composition factor of $G$-module $\mathfrak{G}_{\bar 1}/[\mathfrak{G}_0, \mathfrak{G}_{\bar 1}]$ has a form $L(\lambda)^{[1]}$ and not all $\lambda$ are zero. 
		\end{enumerate} 
	\end{pr}
	\begin{proof}
		Let $\mathfrak{I}$ be a nonzero superideal in $\mathfrak{G}$.  
		Then $\mathfrak{I}_{\bar 0}$ is either zero, or $\mathfrak{I}_{\bar 0}=\mathfrak{G}_{\bar 0}$. In the first case, we have $[\mathfrak{G}_{\bar 1}, \mathfrak{I}_{\bar 1}]=0$. Set $\mathfrak{J}=\sum_{g\in G(\Bbbk)} g\mathfrak{I}_{\bar 1}$. Since $G$ is smooth, it follows that $\mathfrak{J}$ is a $G$-submodule of $\mathfrak{G}_{\bar 1}$, such that $[\mathfrak{G}_{\bar 1}, \mathfrak{J}]=0$, and Proposition \ref{SAPS in positive char} implies $\mathfrak{I}_{\bar 1}\subseteq \mathfrak{J}=0$ because $\mathbb{G}$ is SAS. Therefore, any nonzero superideal $\mathfrak{I}$ contains $\mathfrak{G}_{\bar 0}\oplus [\mathfrak{G}_{\bar 0}, \mathfrak{G}_{\bar 1}]$. Furthermore, the same arguments show that if $\mathfrak{S}$
		is a nonzero $\mathfrak{G}_{\bar 0}$-submodule in $\mathfrak{G}_{\bar 1}$, then $[\mathfrak{G}_{\bar 1}, \mathfrak{S}]\neq 0$, hence it equals $\mathfrak{G}_{\bar 0}$. 
		If $[\mathfrak{G}_{\bar 0}, \mathfrak{G}_{\bar 1}]=0$, then $\mathfrak{G}_{\bar 0}=[\mathfrak{G}_{\bar 1}, \mathfrak{G}_{\bar 1}]$ infers that $\mathfrak{G}_{\bar 0}$ is abelian, a contradiction.
		
		Set $R=\ker(G\to\mathrm{GL}(\mathfrak{G}_{\bar 1}/[\mathfrak{G}_{\bar 0}, \mathfrak{G}_{\bar 1}]))$. Since $\mathrm{Lie}(R)=\mathfrak{G}_{\bar 0}$, the Harish-Chandra subpair $(R, [\mathfrak{G}_{\bar 0}, \mathfrak{G}_{\bar 1}])$
		represents a proper normal supersubgroup $\mathbb{R}$. Thus $R$ is finite, that infers $G_1\leq R$ by \cite[II, \S 7, Corollary 4.3(a)]{dg} and $G_1$ acts trivially on
		$\mathfrak{G}_{\bar 1}/[\mathfrak{G}_{\bar 0}, \mathfrak{G}_{\bar 1}]$. By Steinberg's Tensor Product Theorem (cf. \cite[Theorem II.3.17]{jan}), the first half of $(3)$ follows. Finally, if each
		$\lambda$ equals zero, then $\mathfrak{G}_{\bar 1}/[\mathfrak{G}_{\bar 0}, \mathfrak{G}_{\bar 1}]$ is a nonzero direct sum of several copies of trivial $G$-module $L(0)\simeq \Bbbk$ by \cite[Proposition II.2.14]{jan}, which is a contradiction.
	\end{proof}	
	This proposition shows that the Lie superalgebra $\mathfrak{G}$ of a SAS-supergroup $\mathbb{G}$, such that $\mathfrak{G}_0$ is simple, is very close to being simple too. Nevertheless, it leaves the opportunity that $\mathbb{G}$ can be SAS even if $\mathfrak{G}$ is not simple. The following example shows that the latter is indeed the case.
	
	\subsection{SAS-supergroup with not simple Lie superalgebra}
	
	Recall that $\mathrm{SL}_2$ is a SAS-group, provided $\mathrm{char}\Bbbk\neq 2$. Moreover, its Lie algebra $\mathfrak{sl}_2$ is simple. 
	
	Fix the maximal torus $T$ of $\mathrm{SL}_2$ consisting of diagonal matrices. The character group $X(T)$ can be identified with $\mathbb{Z}$, as well as
	the set of dominant weights $X(T)^+$ is identified with $\mathbb{N}$. Let $X_2$ and $X_{-2}$ denote the root subgroups
	\[X_2(R)=\{I_2+rE_{12}\mid r\in R \}, \ X_{-2}(R)=\{I_2+rE_{21}\mid r\in R \}, \ R\in\mathsf{Alg}_{\Bbbk}.\]
	
	An irreducible $\mathrm{SL}_2$-module $L(n)$ is the socle of the induced module $H^0(n)\simeq\mathrm{Sym}_n(V)$, where $V$ denotes the standard $\mathrm{SL}_2$-module with basis
	$v_1, v_2$. The $n$-th symmetric power $\mathrm{Sym}_n(V)$ 
	has a basis $s_i=v_1^i v_2^{n-i}, 0\leq i\leq n$. The weight of $s_i$ is equal to $2i-n$. 
	
	Recall that the structure of $\mathrm{SL}_2$-module on $\mathrm{Sym}_n(V)$ is equivalent to its induced structure of $\mathrm{Dist}(\mathrm{SL}_2)$-module, so that
	\[\left(\begin{array}{c}
		H \\
		t
	\end{array}\right)s_i=\left(\begin{array}{c}
		2i-n \\
		t
	\end{array}\right)s_i, \   E_{12}^{(k)}s_i=\left(\begin{array}{c}
		n-i \\
		k
	\end{array}\right)s_{i+k}, \ E_{21}^{(k)}s_i=\left(\begin{array}{c}
		i \\
		k
	\end{array}\right)s_{i-k}, \ t, k\geq 0, \]
	where $s_i$ is supposed to be zero, whenever $i> n$ or $i<0$. Besides, $H=E_{11}-E_{22}$ is the generator of \emph{Cartan subalgebra} $\mathrm{Lie}(T)$.
	\begin{rem}\label{Z-lattice}{\em 
			The above formulas can be obtained as follows. Assume that the ground field has zero characteristic. Then $V$ contains the free $\mathbb{Z}$-submodule $V_{\mathbb{Z}}=\mathbb{Z}v_1+\mathbb{Z}v_2$, that is also $(\mathfrak{sl}_2)_{\mathbb{Z}}$-module, where $(\mathfrak{sl}_2)_{\mathbb{Z}}=\mathbb{Z}H + \mathbb{Z}E_{12}+\mathbb{Z}E_{21}$ is a $\mathbb{Z}$-form of Lie algebra $\mathfrak{sl}_2$. Furthermore, the free $\mathbb{Z}$-submodule $\mathrm{Sym}_n(V)_{\mathbb{Z}}=\sum_{0\leq i\leq n}\mathbb{Z}s_i$ is a  $(\mathfrak{sl}_2)_{\mathbb{Z}}$-module as well, hence it is a $\mathrm{U}(\mathfrak{sl}_2)_{\mathbb{Z}}$-module, where $\mathrm{U}(\mathfrak{sl}_2)_{\mathbb{Z}}$ is the Kostant's $\mathbb{Z}$-form of the enveloping algebra 
			of $\mathfrak{sl}_2$ (cf. \cite[II.1.12]{jan}).	It remains to note that $\mathrm{Sym}_n(V)\simeq \mathrm{Sym}_n(V)_{\mathbb{Z}}\otimes_{\mathbb{Z}}\Bbbk$ and  $\mathrm{Dist}(\mathrm{SL}_2)\simeq \mathrm{U}(\mathfrak{sl}_2)_{\mathbb{Z}}\otimes_{\mathbb{Z}}\Bbbk$ over arbitrary field $\Bbbk$.}
		
	\end{rem}
	
	Since the longest element $w_0$ of the Weyl group of $\mathrm{SL}_2$ is equal to $-1$, we have $L(n)^*\simeq L(-w_0(n))=L(n)$. In particular, the Weyl module $V(n)$ is isomorphic to
	$H^0(-w_0(n))^*=\mathrm{Sym}_n(V)^*$. Moreover, if $0\leq n< p$, then $L(n)=H^0(n)\simeq V(n)$. 
	
	The $\mathrm{SL}_2$-module $\mathrm{Sym}_n(V)^*$ has the (dual) basis 
	$s_i^*,  0\leq i\leq n$. The induced action of Lie algebra $\mathfrak{sl}_2$ on $\mathrm{Sym}_n(V)^*$ is determined by  
	\[H s^*_i=(n-2i)s_i^*, \ E_{12}s^*_i=-(n-i+1)s^*_{i-1}, \ E_{21}s_i^*=-(i+1)s^*_{i+1}, 0\leq i\leq n .\]

	Any $\mathrm{SL}_2$-equivariant symmetric bilinear map $\mathrm{Sym}_n(V)^*\times \mathrm{Sym}_n(V)^*\to\mathfrak{sl}_2$ is determined by the identities:  
	\[[s^*_i, s^*_{n-i}]=a_i H, \ [s^*_{j}, s^*_{n-1-j}]=b_j E_{12}, \ [s^*_k, s^*_{n+1-k}]=c_k E_{21}, \]
	\[a_i=a_{n-i}, \ b_j=b_{n-1-j}, \ c_k=c_{n+1-k},\]
	\[0\leq i\leq n, \ 0\leq j\leq n-1, 1\leq k\leq n.\]
	Besides, if $i+i'\neq n, n\pm 1$, then $[s^*_i, s^*_{i'}]=0$.
	
	\begin{pr}\label{the unique form}
		The space $\mathrm{Hom}_{\mathrm{SL}_2}(\mathrm{Sym}_2(\mathrm{Sym}_n(V)^*), \mathfrak{sl}_2)$ is nonzero if and only if $n$ is odd. Moreover, in the latter case it is one-dimensional and if $n\leq p$ or $p=0$, then we additionally have
		\[\mathrm{Hom}_{\mathrm{SL}_2}(\mathrm{Sym}_2(\mathrm{Sym}_n(V)^*), \mathfrak{sl}_2)=\mathrm{Hom}_{\mathfrak{sl}_2}(\mathrm{Sym}_2(\mathrm{Sym}_n(V)^*), \mathfrak{sl}_2).\]	
	\end{pr}
	\begin{proof}
		Let $G$ be an algebraic group and $M$ be a finite dimensional $G$-module. Let $\mathfrak{g}$ denote $\mathrm{Lie}(G)$. Since $\mathrm{char}\Bbbk\neq 2$, we have 
		$M^{\otimes 2}=\mathrm{Sym}_2(M)\oplus\Lambda^2(M)$, where $\mathrm{Sym}_2(M)\simeq \frac{1}{2}(\mathrm{id}_{M^{\otimes 2}}+\tau)(M^{\otimes 2})$, $\Lambda^2(M)\simeq \frac{1}{2}(\mathrm{id}_{M^{\otimes 2}}-\tau)(M^{\otimes 2})$ and $\tau(m\otimes n)=n\otimes m, m, n\in M$. Let $(M^*)^{\otimes 2}\otimes M^{\otimes 2}\to \Bbbk$ be the natural 
		nondegenerate $G$-invariant pairing \[\langle \phi\otimes\psi, m\otimes n\rangle =\phi(m)\psi(n), m, n\in M, \phi, \psi\in M^*.\]
		Then $\langle \mathrm{Sym}_2(M^*), \Lambda^2(M)\rangle =0$ infers the isomorphism $\mathrm{Sym}_2(M^*)\simeq \mathrm{Sym}_2(M)^*$ of $G$-modules, and of $\mathfrak{g}$-modules as well. 
		
		In particular, $\mathrm{Sym}_2(\mathrm{Sym}_n(V)^*)\simeq \mathrm{Sym}_2(\mathrm{Sym}_n(V))^*$, regarded as $\mathrm{SL}_2$-module and $\mathfrak{sl}_2$-module simultaneously. 	
		Note also that $\mathrm{sl}_2$, regarded as a $\mathrm{SL}_2$-module with respect to the adjoint action, is isomorphic to $L(2)$. In particular,  $\mathfrak{sl}_2^*\simeq\mathfrak{sl}_2$. Summing all up, we have
		\[\mathrm{Hom}_{\mathrm{SL}_2}(\mathrm{Sym}_2(\mathrm{Sym}_n(V)^*), \mathfrak{sl}_2)\simeq \mathrm{Hom}_{\mathrm{SL}_2}(\mathfrak{sl}_2, \mathrm{Sym}_2(\mathrm{Sym}_n(V)))\]
		and 
		\[\mathrm{Hom}_{\mathfrak{sl}_2}(\mathrm{Sym}_2(\mathrm{Sym}_n(V)^*), \mathfrak{sl}_2)\simeq \mathrm{Hom}_{\mathfrak{sl}_2}(\mathfrak{sl}_2, \mathrm{Sym}_2(\mathrm{Sym}_n(V))).\]
		We claim that the spaces in the first line are one-dimensional. Indeed, by Donkin-Mathieu theorem \cite[Proposition II.4.21]{jan}, the $\mathrm{SL}_2$-module $\mathrm{Sym}_n(V)^{\otimes 2}$ has a good filtration, and being its direct summand, $\mathrm{Sym}_2(\mathrm{Sym}_n(V))$ also does.
		Moreover, \cite[Proposition II.4.16(a)]{jan} implies that the dimension of $\mathrm{Hom}_{\mathrm{SL}_2}(\mathfrak{sl}_2, \mathrm{Sym}_2(\mathrm{Sym}_n(V)))$ equals the number of factors $H^0(2)$ in a good filtration of $\mathrm{Sym}_2(\mathrm{Sym}_n(V))$. The formal character of $\mathrm{Sym}_2(\mathrm{Sym}_n(V))$, as well as the formal characters of induced modules, do not depend on the characteristic of the ground field. Therefore,  $\dim \mathrm{Hom}_{\mathrm{SL}_2}(\mathfrak{sl}_2, \mathrm{Sym}_2(\mathrm{Sym}_n(V)))$ also does not. 
		Recall that
		\[\mathrm{Hom}_{\mathrm{SL}_2}(\mathfrak{sl}_2, \mathrm{Sym}_2(\mathrm{Sym}_n(V)))\subseteq \mathrm{Hom}_{\mathfrak{sl}_2}(\mathfrak{sl}_2, \mathrm{Sym}_2(\mathrm{Sym}_n(V))),\]
		with the equality, provided $\mathrm{char}\Bbbk=0$. Thus all we need is to prove that the space $ \mathrm{Hom}_{\mathfrak{sl}_2}(\mathfrak{sl}_2, \mathrm{Sym}_2(\mathrm{Sym}_n(V)))$ is always one-dimensional, whenever $\mathrm{char}\Bbbk=0$ or $n\leq \mathrm{char}\Bbbk=p>0$. 
		
		Since $\mathfrak{sl}_2\simeq V(2)$ is irreducible, the morphisms from $\mathrm{Hom}_{\mathfrak{sl}_2}(\mathfrak{sl}_2, \mathrm{Sym}_2(\mathrm{Sym}_n(V)))$ are 
		in one-to-one correspondence with the elements from $(\mathrm{Sym}_n(V)^{\otimes 2})^{<\tau>}$, those are killed by $E_{12}$ and of weight $2$. Any such element has a form \[z=\sum_{1\leq k\leq n} e_k s_k\otimes s_{n+1-k}.\] Then $E_{12}z=0$ implies 
		\[e_{k-1}(n-k+1)+e_k(k-1)=0, 1\leq k\leq n, \ \mbox{where} \ e_0=0.\]
		One can easily see that
		\[e_k=(-1)^{k-1}\left(\begin{array}{c}
			n-1 \\
			k-1
		\end{array}\right)e_1, \ 1\leq k\leq n,\]
		is the unique solution of this linear system of equations, provided $\mathrm{char}\Bbbk=0$ or $n\leq\mathrm{char}\Bbbk=p>0$. Finally, $z$ satisfies $\tau(z)=z$ if and only if $(-1)^{k-1}=(-1)^{n-k}$ for any $1\leq k\leq n$.  We conclude that $n-2k+1$ is even, and hence $n$ is odd. Proposition is proved.
	\end{proof}
	So, if $n$ is odd, there exists (nontrivial) $\mathrm{SL}_2$-equivariant symmetric bilinear map $\mathrm{Sym}_n(V)^*\times \mathrm{Sym}_n(V)^*\to\mathfrak{sl}_2$ and it is unique up to a scalar multiple. It remains to find its structural constants. 
	
	The identites $[E_{21}, [s^*_i, s^*_{n-i}]]=2a_i E_{21}$ are equivalent to 
	\[-(i+1)c_{i+1}-(n-i+1)c_i=2a_i, 0\leq i\leq n, \ \mbox{where} \ c_0=c_{n+1}=0.\]
	Similarly, the identities 
	\[[E_{12}, [s^*_i, s^*_{n-i}]]=-2a_i E_{12}, \  [E_{21}, [s^*_{j}, s^*_{n-1-j}]]=-b_j H, \ [E_{12}, [s^*_k, s^*_{n+1-k}]]=c_k H\]
	are equivalent to 
	\[-(n-i+1)b_{i-1}-(i+1)b_i=-2a_i, 0\leq i\leq n, \ \mbox{where} \ b_{-1}=b_n=0,\]
	\[-(j+1)a_{j+1}-(n-j)a_j=-b_j, 0\leq j\leq n-1, \]
	\[-(n-k+1)a_{k-1}-ka_k=c_k, 1\leq k\leq n.\]
	It is easy to check that this system of equations has the solution 
	\[a_i=\frac{(-1)^i}{2} (\left(\begin{array}{c}
		n-1 \\
		i
	\end{array}\right)-\left(\begin{array}{c}
		n-1 \\
		i-1
	\end{array}\right))a, \ 0\leq i\leq n,\]
	\[b_j=(-1)^j \left(\begin{array}{c}
		n-1 \\
		j
	\end{array}\right)a, \ 0\leq j\leq n-1, \]
	\[c_k=(-1)^k \left(\begin{array}{c}
		n-1 \\
		k-1
	\end{array}\right)a, \ 1\leq k\leq n, \ a\in\Bbbk.\]
	The supergroups whose Harish-Chandra pairs have the forms  $(\mathrm{SL}_2, \mathfrak{G}_{\bar 1})$, $(\mathrm{PSL}_2, \mathfrak{G}_{\bar 1})$ or $(\mathrm{GL}_2, \mathfrak{G}_{\bar 1})$, play a significant role in the study of the structure of quasi-reductive groups. Here we do not provide  more detailed comments, but only note that they appear as centralizers of certain tori in quasi-reductive supergroups. 
	
	The problem of complete classification of such supergroups seems too complicated for now. So, let us consider the particular case when $\mathfrak{G}_{\bar 1}$ is a \emph{highest weight} 
	$\mathrm{SL}_2$-module of weight $n$. Since the Weyl module $V(n)$ is the \emph{universal highest weight} module of weight $n$, there is a surjective morphism $V(n)\to \mathfrak{G}_{\bar 1}$ of $\mathrm{SL}_2$-modules. If the Lie bracket $\mathfrak{G}_{\bar 1}\times\mathfrak{G}_{\bar 1}\to\mathfrak{sl}_2$ is nonzero, then the induced morphism of $\mathrm{SL}_2$-modules (symmetric and bilinear)  $V(n)\times V(n)\to\mathfrak{sl}_2$ is nonzero as well, and it is completely described in Proposition \ref{the unique form}.  In particular, it implies that the Lie bracket $\mathfrak{G}_{\bar 1}\times\mathfrak{G}_{\bar 1}\to\mathfrak{sl}_2$ is also uniquely (up to a nonzero scalar multiple) defined by the above formulas.  
	\begin{lm}\label{when g_1 is V(n)} Let $\mathbb{G}$ be a supergroup, such that $G\simeq\mathrm{SL}_2$. If $\mathbb{G}$ is not split and $\mathfrak{G}_{\bar 1}\simeq V(n)$, then
		either $n=1$ or $n=p=3$.
		
	\end{lm}
	\begin{proof} If $n=1$, then the above formulas infer 
		\[[s_0^*, s_1^*]=\frac{a}{2} H, \ [s_0^*, s_0^*]=aE_{12}, \ [s_1^*, s_1^*]=-aE_{21}, a\in \Bbbk\setminus 0.\]	
		Set $v=\alpha_0 s_0^*+\alpha_1 s^*_1$. Then $[v, v]=a(\alpha_0^2 E_{12}-\alpha_1^2 E_{21}+\alpha_0\alpha_1 H)$  and the straightforward computation shows that $[[v, v], v]=0$.
		
		Let $n>3$. Set $v=\alpha s_0^*+\beta s^*_{n-1}, \text{ where } \alpha, \beta\neq 0$. Then 
		\[[[v, v], v]=[2a\alpha\beta E_{12}, v]= -4a\alpha\beta^2 s^*_{n-2}\neq 0,\]
		which is a contradiction. 
		It remains to consider the case $n=3$.
		
		Using the above formulas, we obtain
		\[[s^*_0, s^*_3]=\frac{a}{2}H, \  [s^*_1, s^*_2]=-\frac{a}{2}H,\, [s^*_0, s^*_2]=aE_{12}, \ [s^*_1, s^*_1]=-2aE_{12}, \]
		\[[s^*_1, s^*_3]=-aE_{21}, \ [s^*_2, s^*_2]=2aE_{21},\, [s^*_0, s^*_0]=[s^*_0, s^*_1]=[s^*_2, s^*_3]=[s^*_3, s^*_3]=0.\]
		Set $v=\sum_{0\leq i\leq 3}\alpha_i s^*_i$. For simplicity, let $a=1$. Then
		\[
		\begin{array}{lcl}
			[v, v] & = & -2\alpha_1^2 E_{12}+2\alpha_2^2 E_{21}+2\alpha_0\alpha_2 E_{12}+\alpha_0\alpha_3 H-\alpha_1\alpha_2 H-2\alpha_1\alpha_3 E_{21}\\[2mm]
			&=& (-2\alpha_1^2+2\alpha_0\alpha_2)E_{12}+(2\alpha_2^2-2\alpha_1\alpha_3)E_{21}+(-\alpha_1\alpha_2+\alpha_0\alpha_3)H\end{array}
		\]
		and  the straightforward computation shows that 
		\[
		\begin{array}{lcl}
			[[v, v], v] & = & (6\alpha_1^3-9\alpha_0\alpha_1\alpha_2+3\alpha_0^2\alpha_3)s_0^*+(3\alpha_1^2\alpha_2-6\alpha_0\alpha_2^2+3\alpha_0\alpha_1\alpha_3)s_1^*\\[2mm]
			&&+ (6\alpha_1^2\alpha_3-3\alpha_0\alpha_2\alpha_3-3\alpha_1\alpha_2^2)s_2^* +(-6\alpha_2^3+9\alpha_1\alpha_2\alpha_3-3\alpha_0\alpha_3^2)s_3^*.
		\end{array}\]
		Thus $[[v, v], v]=0$ for arbitrary $v$ if and only if $p=3$. \end{proof}
	\begin{tr}\label{certain couples}
		Let $\mathbb{G}$ be a non-split supergroup, such that $G\simeq\mathrm{SL}_2$ and $\mathrm{R}_u(\mathbb{G})=e$. If  $\mathfrak{G}_{\bar 1}$ is a highest weight $\mathrm{SL}_2$-module of highest weight $n$, then
		either $n=1$ and $\mathfrak{G}_{\bar 1}\simeq V(1)\simeq L(1)$, or $n=p=3$ and $\mathfrak{G}_{\bar 1}\simeq V(3)$. In the first case $\mathbb{G}\simeq\mathrm{SpO}(2|1)$ and in the second case $\mathbb{G}\simeq\mathrm{Aut}(\mathfrak{spo}(2|1))^0$. In particular, its Lie superalgebra  $\mathfrak{G}$  is  isomorphic to $\mathfrak{Der}(\mathfrak{spo}(2|1))$ which is not simple.
		
	\end{tr}
	\begin{proof}
		Recall that $\mathfrak{G}_{\bar 1}$ is an epimorphic image of the module $V(n)$.   First,  assume that $n>3$. 
		For any $v\in V(n)$ let $\bar v$  denote its image in $\mathfrak{G}_{\bar 1}$.
		
		If $p\not|n$, then consider the element $v=s_0^*+s_n^*$. It is clear that ${\bar v}\neq 0$ in $\mathfrak{G}_{\bar 1}$,\ since $\overline{s_0^*}$ and $\overline{s_n^*}$ are linearly independent in $\mathfrak{G}_{\bar 1}/\mathrm{rad} \mathfrak{G}_{\bar 1}\simeq V(n)/\mathrm{rad} V(n)\simeq L(n)$. 
		Since $n>3$, we have $[s_0^*, s_0^*]=[s_n^*, s_n^*]=0$ and \[[[\bar v, \bar v], \bar v]=\overline{ [[v, v], v]}=\overline{[a H, v]}=\overline{a(n s_0^*-n s_n^*)}\neq 0\] in $\mathfrak{G}_{\bar 1}$ as well, which is a contradiction. 
		
		Now, let $p|n$. Note that $\mathrm{rad} \mathfrak{G}_{\bar 1}\neq 0$, otherwise
		$\mathfrak{G}_{\bar 1}\simeq L(n)\simeq L(\frac{n}{p})^{[1]}$ and therefore, all weights of the nonzero ideal $[\mathfrak{G}_{\bar 1}, \mathfrak{G}_{\bar 1}]\subseteq\mathfrak{sl}_2$ are multiples of $p$, which is obviously impossible. 
		As was observed, $\overline{s_0^*}$ is nonzero in $\mathfrak{G}_{\bar 1}$. If $[\overline{s_0^*}, \mathrm{rad} \mathfrak{G}_{\bar 1}]=0$, then since $\overline{s_0^*}$ generates $\mathfrak{G}_{\bar 1}$,  it would imply that $[\mathfrak{G}_{\bar 1}, \mathrm{rad} \mathfrak{G}_{\bar 1}]=0$, which contradicts the fact that $\mathrm{R}_u(\mathbb{G})=e$. The unique element $\overline{s_i^*}\in \mathrm{rad} \mathfrak{G}_{\bar 1}$ for which $[\overline{s_0^*}, \overline{s_i^*}]\neq 0$ is $\overline{s^*_{n-1}}$, and hence $\overline{s^*_{n-1}}$ is nonzero in $\mathfrak{G}_{\bar 1}$.  
		Set $v=s_0^*+s^*_{n-1}$. Again, $[s_{n-1}^*, s_{n-1}^*]=0$ and if $[[\bar v, \bar v], \bar v]=-4a \overline{s^*_{n-2}}$ is zero in $\mathfrak{G}_{\bar 1}$, then $E_{21}\overline{s^*_{n-2}}=-(n-1)\overline{s^*_{n-1}}=0$ in $\mathfrak{G}_{\bar 1}$, which is a contradiction.  
		
		It remains to consider the case $n=3$. If $p> 3$, then $V(n)\simeq L(n)$, and hence $\mathfrak{G}_{\bar 1}\simeq V(n)$ as well. By Lemma \ref{when g_1 is V(n)}, there should be $p=3$, a contradiction.
		Finally, if $p=3$, then $V(3)$ has the socle isomorphic to $L(1)$, such that $V(3)/L(1)\simeq L(3)\simeq L(1)^{[1]}$. Since $[L(1), L(1)]\neq 0$, we conclude that $\mathfrak{G}_{\bar 1}\simeq V(3)$.
		The first statement is  proven. 
		
		Let $n=p=3$ and $\mathfrak{G}_{\bar 1}\simeq V(3)$. It is easy to see that \[\mathfrak{G}\neq [\mathfrak{G}, \mathfrak{G}]=\mathfrak{G}_{\bar 0}\oplus L(1)\simeq\mathfrak{sl}_2\oplus L(1)\simeq \mathfrak{spo}(2|1).\] 
		The adjoint action induces a supergroup morphism  $\mathbb{G}\to\mathrm{Aut}(\mathfrak{spo}(2|1))^0$. Let $\mathbb{R}$ denote its kernel. Since the centralizer of $[\mathfrak{G}, \mathfrak{G}]$ in $\mathfrak{G}$ is zero, Lemma \ref{Lie superalgebra is trivial} implies that $\mathbb{R}$ is finite and etale, therefore contained in the center of $G\simeq\mathrm{SL}_2$. On the other hand, the center $\mu_2$ of $\mathrm{SL}_2$ does not act trivially on $L(1)$, hence $\mathbb{R}=1$. 
		
		As a byproduct, we obtain the embedding $\mathfrak{G}\to\mathfrak{Der}(\mathfrak{spo}(2|1))$ of Lie superalgebras. Note that the image of $\mathfrak{G}_{\bar 0}$ in $\mathfrak{Der}(\mathfrak{spo}(2|1))_{\bar 0}$ is isomorphically  mapped onto $\mathfrak{Der}(\mathfrak{sl}_2)\simeq\mathfrak{sl}_2$. Furthermore, the kernel of the canonical surjection $\mathfrak{Der}(\mathfrak{spo}(2|1))_{\bar 0}\to\mathfrak{sl}_2$ is contained in
		$\mathrm{End}_{\mathfrak{sl}_2}(L(1))\simeq\Bbbk$, that infers either
		$\mathfrak{Der}(\mathfrak{spo}(2|1))_{\bar 0}\simeq\mathfrak{sl}_2$ or
		$\mathfrak{Der}(\mathfrak{spo}(2|1))_{\bar 0}\simeq \mathfrak{sl}_2\oplus\Bbbk$. In the latter case
		there is $\delta\in \mathfrak{Der}(\mathfrak{spo}(2|1))_{\bar 0}$, such that $\delta|_{\mathfrak{sl}_2}=0$ and $\delta |_{L(1)}=\alpha\mathrm{id},  \alpha\in\Bbbk\setminus 0$. Then for any $u, v\in L(1)$, such that $[u, v]\neq 0$,  we have $\delta([u, v])=0=2\alpha [u, v]$, hence $\alpha=0$. This contradiction implies
		$\mathfrak{Der}(\mathfrak{spo}(2|1))_{\bar 0}\simeq\mathfrak{sl}_2$.
		In particular, the Lie superalgebra of outer superderivations of $\mathfrak{spo}(2|1)$ is purely odd.
		
		Further, the algebraic group $\mathrm{Aut}(\mathfrak{spo}(2|1))^0_{ev}$ contains $G\simeq\mathrm{SL}_2$. Comparing dimensions, we conclude that $\mathrm{Aut}(\mathfrak{spo}(2|1))^0_{ev}$ is smooth and that by \cite[Proposition 10.15]{milne}  isomorphic to $G$.
		
		Recall that, as in the purely even case, for any Lie superalgebra $\mathfrak{H}$ we have the standard isomorphism (of superspaces)
		$\mathrm{H}^1(\mathfrak{H}, \mathfrak{H})\simeq\mathfrak{Der}(\mathfrak{H})/\mathfrak{InDer}(\mathfrak{H})$ (cf. \cite{l2}). By \cite[Lemma 2.6]{BGL5} the space $\mathrm{H}^1(\mathfrak{spo}(2|1), \mathfrak{spo}(2|1))$ is two-dimensional, hence 
		$\mathfrak{G}\simeq\mathfrak{Der}(\mathfrak{spo}(2|1))$. Summing all up, we obtain $\mathbb{G}\simeq \mathrm{Aut}(\mathfrak{spo}(2|1))^0$.
		
		Finally, if $n=1$,  then one notes that  $\mathrm{SpO}(2|1)_{ev}\simeq\mathrm{Sp}_2\simeq\mathrm{SL}_2$ and $\mathfrak{spo}(2|1)_{\bar 1}$ is obviously isomorphic to $L(1)$ (see below).  Proposition \ref{the unique form}
		concludes the proof.  
	\end{proof}	
	
	\subsection{About other alternatives from Proposition \ref{Lie algebra of SAS}}
	We say that a supergroup $\mathbb{G}$ satisfies an alternative from Proposition \ref{Lie algebra of SAS}, whenever $G$ does. From now on we also assume that if $G$ is of type $\mathrm{G}_2$, then $p\neq 3$.	
	
	The arguments of Remark \ref{Lie algebra of SAS, an example} and Proposition \ref{Lie algebra of SAS} show that a SAS-group $G$ satisfies the alternative $(2)$ if and only if :
	\begin{enumerate}
		\item $G\simeq\mathrm{SL}_n/\mu_{n'}$, where $p|n, p\not | n'$;
		\item $G$ is a simply connected group of type $\mathrm{E}_6$ and $p=3$.
	\end{enumerate}
	Similarly, $G$ satisfies the alternative $(3)$ if and only if :
	\begin{enumerate}
		\item $G\simeq\mathrm{SL}_n/\mu_{n'}, p|n', p\not |\frac{n}{n'}$;
		\item $p=3$ and $G\simeq\widetilde{G}/\mu_3$, where $\widetilde{G}$ is a simply connected group of type $\mathrm{E}_6$. 
	\end{enumerate}
	Finally, $G$ satisfies the alternative $(4)$ if and only if 
	\[G\simeq\mathrm{SL}_n/\mu_{n'}, \ \mbox{where} \ p|\frac{n}{n'}, \ p|n' .\]
	
	Assume that $G$ satisfies $(2)$ and $(4)$. In both cases, $G$ has the center $D\simeq\mu_l$ with $p|l$. 
	In particular, $\mathfrak{Z}=\mathrm{Z}(\mathrm{Lie}(G))=\mathrm{Lie}(D)=\Bbbk z$. 
	Let $X(D)\simeq\mathbb{Z}_l$ be the character group of $D$. 
	\begin{pr}\label{SAS with small square of odd component}
		Let $\mathbb{G}$ be a supergroup, which satisfies the alternative $(2)$ or $(4)$ from Proposition \ref{Lie algebra of SAS} and $[\mathfrak{G}_{\bar 1}, \mathfrak{G}_{\bar 1}]=\mathfrak{Z}$. In particular, $\mathfrak{Z}\oplus \mathfrak{G}_{\bar 1}$ is a proper superideal and therefore, $\mathfrak{G}$ is not simple. Then $\mathbb{G}$ is SAS if and only if :
		\begin{enumerate}
			\item $\mathfrak{G}_{\bar 1}\simeq\mathfrak{G}_{\bar 1}^*$ (as a $G$-module) so that \[[v, w]=\langle v, w \rangle z, \text{ for all } v, w\in \mathfrak{G}_{\bar 1},\] where $\langle \ , \ \rangle $ is the nondegenerate symmetric bilinear $G$-invariant form on $\mathfrak{G}_{\bar 1}$, naturally determined by this isomorphism; 
			\item The subgroup $\mu_p\leq D$ acts trivially on $\mathfrak{G}_{\bar 1}$; 	
			\item $\mathfrak{G}_{\bar 1}^G =0$.
		\end{enumerate}
	\end{pr}
	\begin{proof}
		If $V$ is a subspace of $\mathfrak{G}_{\bar 1}$, such that $[\mathfrak{G}_{\bar 1}, V]=0$, then $\mathfrak{V}=\sum_{g\in G(\Bbbk)} gV$ is a $G$-submodule and $[\mathfrak{G}_{\bar 1}, \mathfrak{V}]=0$
		implies $V\subseteq\mathfrak{V}=0$.  Thus, the first condition obviously follows.
		
		Next, there is the weight decomposition $\mathfrak{G}_{\bar 1}=\oplus_{h\in X(D)} (\mathfrak{G}_{\bar 1})_h$, where each term $(\mathfrak{G}_{\bar 1})_h$ is a $G$-submodule of $\mathfrak{G}_{\bar 1}$. Since $\langle \, , \, \rangle$ is nondegenerate, 	then $(\mathfrak{G}_{\bar 1})_h\neq 0$ implies that  $[(\mathfrak{G}_{\bar 1})_h, (\mathfrak{G}_{\bar 1})_g]\neq 0$ if and only if $g=h^{-1}$. In particular, $(\mathfrak{G}_{\bar 1})_h^*\simeq (\mathfrak{G}_{\bar 1})_{h^{-1}}$ for any $h\in X(D)$.
		
		Choose $x\in (\mathfrak{G}_{\bar 1})_h\setminus 0, y\in (\mathfrak{G}_{\bar 1})_{h^{-1}}\setminus 0$, such that $\langle x, y \rangle \neq 0$, and set $v=x+y$. Then 
		\[[[v, v], v]=\langle v, v \rangle [z, v]=2\langle x, y \rangle [z, v]=2\langle x, y\rangle (hx-hy)=0\] implies $h\in p\mathbb{Z}_l$, or equivalently, the subgroup $\mu_p\leq D$ acts trivially on $\mathfrak{G}_{\bar 1}$.
		Since $\mathfrak{G}_{\bar 1}$ is self-dual, the condition $\mathfrak{G}_{\bar 1}=(G-1)\mathfrak{G}_{\bar 1}$ is equivalent to $\mathfrak{G}_{\bar 1}^G=0$. 
		
		To complete the proof just note that the conditions $(1)$ and $(2)$ guarantee that $(G, \mathfrak{G}_{\bar 1})$ is a Harish-Chandra pair, and the corresponding supergroup satisfies the condition $(2)$ from Lemma \ref{a very particular case}. 
	\end{proof} 
	
	\begin{ex}\label{2.18}{\em 
			Set $G=\mathrm{SL}_n$ and $p|n$. Let $V=L(p\lambda)\oplus L(p\lambda)^*$, where $\lambda$ is a nonzero dominant weight. Then $V\simeq V^*$ and the map $(v, w)\mapsto \langle v, w \rangle I_n, v, w\in V,$ determines the structure of Harish-Chandra pair on $(\mathrm{SL}_n, V)$, such that the corresponding supergroup is SAS. }
	\end{ex}  
	\begin{rem}\label{curious property}{\em 
			The SAS-supergroup $\mathbb{G}$ from Proposition \ref{SAS with small square of odd component} has the following curious properties.  The subpair $(D, \mathfrak{G}_{\bar 1})$ corresponds to a normal supersubgroup $\mathbb{H}$ in $\mathbb{G}$, such that $\mathbb{G}/\mathbb{H}\simeq G/D$ is a SAS-group. Furthermore, $\mathbb{H}$ is finite and solvable by \cite[Corollary 6.3]{maszub1}. However, $\mathbb{H}$  is semisimple! Indeed, if $(S, W)$ is the Harish-Chandra pair corresponding to the solvable radical of $\mathbb{H}$, then $S=e$ by Lemma \ref{Lie superalgebra is trivial}. Thus $[\mathfrak{G}_{\bar 1}, W]=0$ by normality criterion, that contradicts to the non-degeneracy of $[ \ , \ ]$.} 
	\end{rem}
	\begin{pr}\label{alternative 2}
		Let $\mathbb{G}$ be a SAS-supergroup, which satisfies the alternative $(2)$ from Proposition \ref{Lie algebra of SAS}. Assume also that $[\mathfrak{G}_{\bar 1}, \mathfrak{G}_{\bar 1}]=\mathfrak{G}_{\bar 0}$. Then one of the following statements takes place :
		\begin{enumerate}
			\item Any nonzero superideal of $\mathfrak{G}$ contains $[\mathfrak{G}, \mathfrak{G}]=\mathfrak{G}_{\bar 0}\oplus [\mathfrak{G}_{\bar 0}, \mathfrak{G}_{\bar 1}]$;
			\item There is a finite solvable  normal supersubgroup $\mathbb{H}$, represented by a subpair $(\mu_{l'}, \mathfrak{H}_{\bar 1})$, such that
			$p|l'$	and $\mathbb{L}=\mathbb{G}/\mathbb{H}$ is  not purely even and also a SAS-supergroup satisfying the alternatives $(3)$ or $(4)$. Besides, $[\mathfrak{L}_{\bar 1}, \mathfrak{L}_{\bar 1}]=[\mathfrak{L}_{\bar 0}, \mathfrak{L}_{\bar 0}]$.
		\end{enumerate}
	\end{pr}
	
	\begin{proof}
		Let $\mathfrak{I}$ be a nonzero superideal in $\mathfrak{G}$. It has been already observed that $\mathfrak{I}_{\bar 0}\neq 0$. Then either $\mathfrak{I}_{\bar 0}=\mathfrak{G}_{\bar 0}$, hence $\mathfrak{I}$ contains $[\mathfrak{G}, \mathfrak{G}]$, or $\mathfrak{I}_{\bar 0}=\mathfrak{Z}$. In the latter case $\mathfrak{I}_{\bar 1}\neq \mathfrak{G}_{\bar 1}$. Without loss of generality, one can assume that $G\mathfrak{I}_{\bar 1}\subseteq \mathfrak{I}_{\bar 1}$ and $\mathfrak{I}_{\bar 1}$ is the largest $G$-submodule of $\mathfrak{G}_{\bar 1}$ with $[\mathfrak{G}_{\bar 1}, \mathfrak{I}_{\bar 1}]\subseteq\mathfrak{Z}$. Set $R=\ker (G\to \mathrm{GL}(\mathfrak{G}_{\bar 1}/\mathfrak{I}_{\bar 1}))$. Then $\mathrm{Lie}(R)=\{x\in \mathfrak{G}_{\bar 0}\mid [x, \mathfrak{G}_{\bar 1}]\subseteq \mathfrak{I}_{\bar 1}\}$ contains $\mathfrak{I}_{\bar 0}=\mathfrak{Z}$. Thus $R\cap D\simeq\mu_{l'}$, where $p|l' , l'|l$, and $(\mu_{l'}, \mathfrak{I}_{\bar 1})$ is a subpair of $(G, \mathfrak{G}_{\bar 1})$ corresponding to a normal supersubgroup $\mathbb{H}$ of $\mathbb{G}$. 
		
		The supergroup $\mathbb{H}$ is finite and solvable by \cite[Corollary 6.3]{maszub1}. By Lemma \ref{if quotient is SAPS} the supergroup $\mathbb{G}/\mathbb{H}$ is SAPS.
		Let $\mathbb{R}$ denote the preimage of $\mathrm{R}(\mathbb{G}/\mathbb{H})\neq e$ in $\mathbb{G}$. If $\mathbb{R}$ is infinite, then $\mathbb{R}=\mathbb{G}$ and thus $\mathbb{G}$ is solvable, a contradiction. Finally, if $\mathbb{R}$ is finite, then 
		$\mathrm{R}(\mathbb{G}/\mathbb{H})$ is purely odd, which implies that there is a $G$-submodule $\mathfrak{V}$ in $\mathfrak{G}_{\bar 1}$, such that $\mathfrak{I}_{\bar 1}$ is properly contained in $\mathfrak{V}$ and $[\mathfrak{G}_{\bar 1}, \mathfrak{V}]\subseteq \mathfrak{Z}$. The latter contradicts the maximality of $\mathfrak{I}_{\bar 1}$. 
		To complete the proof just note that $[\mathfrak{L}_{\bar 1}, \mathfrak{L}_{\bar 1}]$ is contained in the image of $\mathfrak{G}_{\bar 0}$ in $\mathfrak{L}_{\bar 0}$, which in turn  is contained in  $[\mathfrak{L}_{\bar 0}, \mathfrak{L}_{\bar 0}]$. 
	\end{proof}
	
	\begin{pr}\label{alternative 3}
		Let $\mathbb{G}$ be a not purely even SAS-supergroup, which satisfies the alternative $(3)$ or $(4)$ from Proposition \ref{Lie algebra of SAS}. In the latter case we also assume that $[\mathfrak{G}_{\bar 1}, \mathfrak{G}_{\bar 1}]\neq\mathfrak{Z}$. Then one of the following statements takes place :
		\begin{enumerate}
			\item $[\mathfrak{G}_{\bar 1}, \mathfrak{G}_{\bar 1}]=[\mathfrak{G}_{\bar 0}, \mathfrak{G}_{\bar 0}]$;
			\item Any nonzero superideal of $\mathfrak{G}$ contains $[\mathfrak{G}, \mathfrak{G}]$;
			\item There is a finite normal supersubgroup $\mathbb{H}$ of $\mathbb{G}$, such that $\mathfrak{H}_{\bar 0}=[\mathfrak{G}_{\bar 0}, \mathfrak{G}_{\bar 0}]$ or $\mathfrak{H}_{\bar 0}=\mathfrak{Z}$ and $\mathbb{L}=\mathbb{G}/\mathbb{H}$ is not purely even and also a SAS-supergroup, which satisfies the alternative $(2), (4)$ or $(3), (4)$ from Proposition \ref{Lie algebra of SAS}, respectively. Moreover, in the first case  
			$[\mathfrak{L}_{\bar 1}, \mathfrak{L}_{\bar 1}]=\mathfrak{Z}$, and in the second case $[\mathfrak{L}_{\bar 1}, \mathfrak{L}_{\bar 1}]=[\mathfrak{L}_{\bar 0}, \mathfrak{L}_{\bar 0}]$.
			Additionally, $\mathbb{H}$ can be chosen being infinitesimal of height at most one.
		\end{enumerate}
	\end{pr}
	\begin{proof}
		Assume that the first statement does not take place. Then $[\mathfrak{G}_{1}, \mathfrak{G}_{\bar 1}]=\mathfrak{G}_{\bar 0}$.
		
		Let $\mathfrak{I}$ be a nonzero superideal of $\mathfrak{G}$, such that $\mathfrak{I}\not\supseteq [\mathfrak{G}, \mathfrak{G}]$. Since $\mathfrak{I}_{\bar 0}\neq 0$, therefore it follows that $\mathfrak{I}_{\bar 0}=[\mathfrak{G}_{\bar 0}, \mathfrak{G}_{\bar 0}]$ or
		$\mathfrak{I}_{\bar 0}=\mathfrak{Z}$.
		
		Consider the first case. We have $0\neq [[\mathfrak{G}_{\bar 0}, \mathfrak{G}_{\bar 0}], \mathfrak{G}_{\bar 1}]\subseteq \mathfrak{I}_{\bar 1}\neq\mathfrak{G}_{\bar 1}$ and $[G\mathfrak{I}_{\bar 1}, \mathfrak{G}_{\bar 1}]\subseteq [\mathfrak{G}_{\bar 0}, \mathfrak{G}_{\bar 0}]$. The latter also implies that $G\mathfrak{I}_{\bar 1}\neq\mathfrak{G}_{\bar 1}$, therefore, without loss of generality, one can assume that $\mathfrak{I}_{\bar 1}$ is a maximal $G$-submodule of $\mathfrak{G}_{\bar 1}$, such that $[\mathfrak{I}_{\bar 1}, \mathfrak{G}_{\bar 1}]\subseteq [\mathfrak{G}_{\bar 0}, \mathfrak{G}_{\bar 0}]$
		
		Set $H=\ker(G\to\mathrm{GL}(\mathfrak{G}_1/[[\mathfrak{G}_{\bar 0}, \mathfrak{G}_{\bar 0}], \mathfrak{G}_{\bar 1}]))$. Then \[[\mathfrak{G}_{\bar 0}, \mathfrak{G}_{\bar 0}]\subseteq \mathrm{Lie}(H)=\{x\in\mathfrak{G}_{\bar 0}\mid [x, \mathfrak{G}_{\bar 1}]\subseteq [[\mathfrak{G}_{\bar 0}, \mathfrak{G}_{\bar 0}], \mathfrak{G}_{\bar 1}]\}.\]
		We obtain that $(H, \mathfrak{I}_{\bar 1})$ is a Harish-Chandra subpair of $(G, \mathfrak{G}_{\bar 1})$, which corresponds to a normal  supersubgroup $\mathbb{H}$ in $\mathbb{G}$. Since $\mathbb{H}\neq \mathbb{G}$, the supersubgroup $\mathbb{H}$ is finite. Without loss of generality, one can replace $\mathbb{H}$ by its first Frobenius kernel $\mathbb{H}_1=\mathbb{H}\cap\mathbb{G}_1$. Furthermore, observe that $[\mathfrak{G}_{\bar 0}, \mathfrak{G}_{\bar 0}]$ is a $p$-subalgebra of $\mathfrak{G}_{\bar 0}$ as a factor-algebra of a $p$-algebra over its center.  By \cite[Corollary 11.38]{milne}, there is a finite normal subgroup $K$ of $G$, such that $\mathrm{Lie}(K)=[\mathfrak{G}_{\bar 0}, \mathfrak{G}_{\bar 0}]$. Replacing $H$ by $H\cap K$, one can assume now that $\mathfrak{H}_{\bar 0}=[\mathfrak{G}_{\bar 0}, \mathfrak{G}_{\bar 0}]$.
		
		By Lemma \ref{if quotient is SAPS} the supergroup $\mathbb{L}=\mathbb{G}/\mathbb{H}$ is SAPS and not purely even. Since $G/H$ is SAS, the solvable radical of $\mathbb{L}$ is a purely odd supersubgroup. Therefore,  there is a proper $G$-submodule $\mathfrak{V}\subseteq \mathfrak{G}_{\bar 1}$, such that $\mathfrak{I}_{\bar 1}\subseteq\mathfrak{V}$ and $[\mathfrak{V}, \mathfrak{G}_{\bar 1}]\subseteq [\mathfrak{G}_{\bar 0}, \mathfrak{G}_{\bar 0}]$. The maximality of $\mathfrak{I}_{\bar 1}$
		implies that the solvable radical of $\mathbb{L}$ is trivial.
		
		The Lie algebra $\mathfrak{L}_{\bar 0}$ contains a one-dimensional ideal, which is just its center, isomorphic to $\mathfrak{G}_{\bar 0}/\mathfrak{H}_{\bar0}$. Thus $G/H$ satisfies the alternative $(2)$ or $(4)$ from Proposition \ref{Lie algebra of SAS} and additionally, $[\mathfrak{L}_{\bar 1}, \mathfrak{L}_{\bar 1}]\subseteq\mathfrak{Z}$. 
		
		Finally, let $\mathfrak{I}_{\bar 0}=\mathfrak{Z}$. Then $G\simeq\mathrm{SL}_{n}/\mu_{n'}, p|n', p|\frac{n}{n'}$ and the center of $G$ is isomorphic to $\mu_n/\mu_n'\simeq\mu_{\frac{n}{n'}}$. Arguing as above, one can find a finite normal subgroup $\mathbb{H}$ of $\mathbb{G}$, such that $H\simeq\mu_p\simeq\mu_{pn'}/\mu_{n'}$
		and $\mathbb{L}=\mathbb{G}/\mathbb{H}$ is a SAS-supergroup, which is not purely even as well. Furthermore, the group $G/H$ satisfies the alternatives $(3)$ or $(4)$ and $[\mathfrak{L}_{\bar 0}, \mathfrak{L}_{\bar 0}]\simeq\mathfrak{G}_{\bar 0}/\mathfrak{Z}$ equals $[\mathfrak{L}_{\bar 1}, \mathfrak{L}_{\bar 1}]$.
	\end{proof}
	\begin{tr}\label{1 or 2}
		Let $\mathbb{G}$ be a not purely even SAS-supergroup, which satisfies one of the alternatives $(2), (3)$ or $(4)$ from Proposition \ref{Lie algebra of SAS}.  	Then there is a finite normal supersubgroup 
		$\mathbb{H}$ of $\mathbb{G}$, such that $\mathbb{L}=\mathbb{G}/\mathbb{H}$ is not purely even and SAS, and one of the following statements hold:
		\begin{enumerate}
			\item $\mathbb{L}$ is as in Proposition \ref{SAS with small square of odd component};
			\item Any nonzero superideal of $\mathfrak{L}$ contains $[\mathfrak{L},  \mathfrak{L}]$;
			\item $\mathbb{L}$ satisfies  the alternatives $(3)$ or $(4)$ from Proposition \ref{Lie algebra of SAS} and $[\mathfrak{L}_{\bar 0}, \mathfrak{L}_{\bar 0}]=[\mathfrak{L}_{\bar 1}, \mathfrak{L}_{\bar 1}]$.
		\end{enumerate}
		The supersubgroup $\mathbb{H}$ can be chosen being infinitesimal of height at most one.
	\end{tr}
	\begin{proof}
		Assume that $\mathbb{G}$ is an SAS-supergroup that satisfies none of the above statements.   If $\mathbb{G}$ satisfies the alternative $(2)$ of Proposition \ref{Lie algebra of SAS}, then Proposition \ref{alternative 2} entails our theorem, except the last statement. In fact, $H\simeq\mu_{l'}$ is not necessarily infinitesimal. However, we can replace $H$ with its subgroup $\mu_p$, which will not change the proof in any way. 
		Similarly, if $\mathbb{G}$ satisfies $(3)$ or $(4)$, then Proposition \ref{alternative 3} does. 
	\end{proof}

	\section{Almost-simple supergroups from Kac's classification list}
	
	In this section, we demonstrate how the method of Harish-Chandra pairs can be used to prove that the simple Lie superalgebras from Kac's classification theorem, except those belong to the \emph{Cartan series}, are algebraic, i.e. they are Lie superalgebras of some almost-simple algebraic supergroups. It may seem incorrect to refer to Kac's classification, since the ground field may have a nonzero characteristic. However, the definitions of all simple Lie superalgebras from Kac's list  easily extends to the case of positive characteristic. In some cases, we briefly discuss their simplicity, but our main objective is to construct the corresponding supergroups and determine whether they are SAS or WAS. 
	
	We have already discussed the projective special linear supergroups in the previous section. The Lie superalgebras from the Cartan series will be discussed in an upcoming paper. 
	
	\subsection{Ortho-symplectic supergroups}
	
	The ortho-symplectic supergroups are of two types $\mathrm{SpO}(2m|2n+1)$ and $\mathrm{SpO}(2m|2n)$. We consider the first type only, the second one is similar.
	Recall that for any superalgebra $R$, the group  $\mathrm{SpO}(2m|2n+1)(R)$ consists of all matrices $g\in\mathrm{SL}(2m|2n+1)(R)$, such that $g^{st}Jg=J$, where 
	\[g=\left(\begin{array}{cc}
		A & B \\
		C & D
	\end{array}\right), \ g^{st}=\left(\begin{array}{cc}
		A^t & C^t \\
		-B^t & D^t
	\end{array}\right),\]
	\[ A\in\mathrm{GL}_{2m}(R_{\bar 0}), \ B\in\mathrm{Mat}_{2m\times (2n+1)}(R_{\bar 1}), \ C\in\mathrm{Mat}_{(2n+1)\times 2m}(R_{\bar 1}), D\in\mathrm{GL}_{2n+1}(R_{\bar 0}),\]
	\[J=\left(\begin{array}{ccccc}
		0 & I_m & 0 & 0 & 0 \\
		-I_m & 0 & 0 & 0 & 0 \\
		0 & 0 & 0 & I_n & 0 \\
		0 & 0 & I_n & 0 & 0 \\
		0 & 0 & 0 & 0 & 1
	\end{array}\right) .\]
	It is clear that $\mathrm{SpO}(2m|2n+1)_{ev}\simeq \mathrm{Sp}_{2m}\times\mathrm{SO}(2n+1)$. In particular, if $\mathrm{char}\Bbbk=p>0$, then $\mathrm{SpO}(2m|2n+1)$ is not
	SAS-supergroup, provided $m, n>0$. Nevertheless, since its Lie superalgebra $\mathfrak{spo}(2m|2n+1)$ is simple, $\mathrm{SpO}(2m|2n+1)$ is a WAS-supergroup. 
	
	For the reader's convenience, we will show that $\mathfrak{spo}(2m|2n+1)$ is simple in the modular case;  the case of zero characteristic is well known (cf. \cite{chengwang}). 
	Let $J_s$ and $J_o$ denote the upper and the bottom diagonal blocks of the matrix $J$ respectively. Then $\mathfrak{spo}(2m|2n+1)_1$ consists of matrices 
	\[ \left(\begin{array}{cc}
		0 & B \\
		C & 0
	\end{array}\right), \ B\in\mathrm{Mat}_{2m\times (2n+1)}(\Bbbk), \ C\in\mathrm{Mat}_{(2n+1)\times 2m}(\Bbbk),\]
	such that 
	\[J_s B+C^t J_o =0, \ \mbox{or equivalently} \ C= J_o B^t J_s.\]
	Furthermore, $\mathfrak{spo}(2m|2n+1)_{\bar 0}\simeq \mathfrak{sp}_{2m}\oplus \mathfrak{so}_{2n+1}$ acts on $\mathfrak{spo}(2m|2n+1)_{\bar 1}$ by
	\[ [\left(\begin{array}{cc}
		A & 0 \\
		0 & D
	\end{array}\right), \left(\begin{array}{cc}
		0 & B \\
		C & 0
	\end{array}\right)]=\left(\begin{array}{cc}
		0 & AB-BD \\
		DC-CA & 0
	\end{array}\right),\]
	so that $\mathfrak{spo}(2m|2n+1)_{\bar 1}$ can be identified with the space $\mathrm{Mat}_{2m\times (2n+1)}(\Bbbk)$, on which  $\mathfrak{spo}(2m|2n+1)_{\bar 0}$ acts as above. The latter implies that
	$\mathfrak{spo}(2m|2n+1)_{\bar 1}\simeq V\otimes W^*$, where $V$ and $W$ are the standard (irreducible) $\mathfrak{sp}_{2m}$-module and $\mathfrak{so}_{2n+1}$-module respectively. In particular, $\mathfrak{spo}(2m|2n+1)_{\bar 1}$ is an irreducible $\mathfrak{spo}(2m|2n+1)_{\bar 0}$-module and therefore, any nonzero superideal $\mathfrak{I}$ in $\mathfrak{spo}(2m|2n+1)$ satisfies
	either $\mathfrak{I}_{\bar 1}=0$ or $\mathfrak{I}_{\bar 1}=\mathfrak{spo}(2m|2n+1)_{\bar 1}$. If $\mathfrak{I}_{\bar 1}=0$, then the ideal $\mathfrak{I}_{\bar 0}$ contains $\mathfrak{sp}_{2m}$ or $\mathfrak{so}_{2n+1}$ and it acts trivially on  $\mathfrak{spo}(2m|2n+1)_{\bar 1}$, a contradiction. In the second case we note that $[\mathfrak{spo}(2m|2n+1)_{\bar 1}, \mathfrak{spo}(2m|2n+1)_{\bar 1}]$ is non-trivially projected into both $\mathfrak{sp}_{2m}$ and $\mathfrak{so}_{2n+1}$, hence $\mathfrak{I}=\mathfrak{spo}(2m|2n+1)$. 
	
	\subsection{Periplectic supergroups}
	
	Let ${\bf P}(n)$ be a periplectic supergroup, where $n\geq 2$. Then for any (supercommutative) superalgebra $R$, we have
	\[{\bf P}(n)(R) = \{g\in \mathrm{GL}(n|n)(R)\,|\,g^{st}Jg = J\}, \ \mbox{where} \
	J=\left(\begin{array}{cc}
		0 & I_n \\
		I_n & 0
	\end{array}\right).\]	
	The Lie superalgebra $\mathfrak{p}(n)=\mathrm{Lie}({\bf P}(n))$ consists of matrices
	\[\left(\begin{array}{cc}
		A & B \\
		C & D
	\end{array}\right)\in\mathfrak{gl}(n|n), \ \mbox{such that} \ B=B^t, \ C=-C^t, \ A=-D^t.\]
	Let $\mathrm{Sym}_{n\times n}$ and $\mathrm{SSym}_{n\times n}$ denote the subspaces of symmetric and skew-symmetric $n\times n$ matrices.
	The group ${\bf P}(n)_{ev}\simeq\mathrm{GL}_n$ acts on $\mathfrak{p}(n)_{\bar 1}\simeq\mathrm{Sym}_{n\times n}\oplus\mathrm{SSym}_{n\times n}$ by the rule
	\[A\cdot (B, C)= (AB A^t, (A^t)^{-1}C A^{-1}),\text{ for all } B\in \mathrm{Sym}_{n\times n}, \ C\in\mathrm{SSym}_{n\times n}, A\in\mathrm{GL}_n.\]
	Consider the normal supersubgroup $\mathbb{H}={\bf P}(n)\cap\mathrm{SL}(n|n)$ in ${\bf P}(n)$. Recall that $\mathrm{SL}(m|n)=\ker\mathrm{Ber}$, where
	$\mathrm{Ber}$ is the character $\mathrm{GL}(m|n)\to\mathbb{G}_m\simeq G_m$, called \emph{Berezinian},  that is defined as
	\[\mathrm{Ber}(g)=\det(A-BD^{-1}C)\det(D)^{-1}, \ g=\left(\begin{array}{cc}
		A & B \\
		C & D
	\end{array}\right)\in\mathrm{GL}(m|n).\]
	Thus $H=\mathbb{H}_{ev}=\{g\in\mathrm{GL}_n \mid \det(g)^2=1\}$, $\mathfrak{H}_{\bar 1}=\mathfrak{p}(n)_{\bar 1}$ and $H/\mathrm{SL}_n\simeq\mu_2$ is etale. The latter also infers that $H^0=\mathrm{SL}_n$ and $[\mathfrak{H}_{\bar 1}, \mathfrak{H}_{\bar 1}]\subseteq\mathfrak{H}_{\bar 0}=\mathfrak{sl}_n$. In other words, $\mathbb{H}^0$ is a normal (smooth and connected) supersubgroup
	in ${\bf P}(n)$, corresponding to the Harish-Chandra (sub)pair $(\mathrm{SL}_n, \mathfrak{p}(n)_{\bar 1})$. Let ${\bf SP}(n)$ denote $\mathbb{H}^0$.
	\begin{pr}\label{periplectic SAS}
		We have ${\bf SP}(n)=\mathcal{D}({\bf P}(n))$ and it is a SAS-supergroup, provided $n> 2$. 	
	\end{pr}
	\begin{proof}
		Since ${\bf P}(n)/{\bf SP}(n)$ is abelian, there is $\mathcal{D}({\bf P}(n))\leq {\bf SP}(n)$ and \[[\mathfrak{p}(n), \mathfrak{p}(n)]=\mathfrak{sl}_n\oplus\mathfrak{p}(n)_{\bar 1}\subseteq\mathrm{Lie}(\mathcal{D}({\bf P}(n))).\]
		Since, $\mathcal{D}({\bf P}(n))_{ev}$ contains $\mathrm{SL}_n=\mathcal{D}(\mathrm{SL}_n)$, the first statement follows. 	
		As a $\mathrm{SL}_n$-module, $\mathfrak{p}(n)_{\bar 1}$ can be identified with $\mathrm{Sym}_2(V)\oplus\Lambda^2(V^*)$, where $V$ is the standard $n$-dimensional $\mathrm{SL}_n$-module.
		Since both of them are nontrivial irreducible $\mathrm{SL}_n$-modules, the condition $(1)$ from Lemma \ref{a very particular case} immediately follows. Further, any nonzero $\mathrm{SL}_n$-submodule of $\mathfrak{p}(n)_{\bar 1}$ contains either $\mathrm{Sym}_2(V)$ or $\Lambda^2(V^*)$. Since $[\mathrm{Sym}_2(V), \Lambda^2(V^*)]\neq 0$, the condition $(2)$
		from Lemma \ref{a very particular case} follows as well. 
	\end{proof}
	\begin{rem}\label{the case n=2} {\rm 
			Note that if $n=2$, then $\Lambda^2(V^*)\simeq\Bbbk$ and therefore, the Harish-Chandra subpair $(\mathrm{SL}_2, \mathrm{Sym}_2(V))$ determines a smooth connected normal supersubgroup $\mathbb{H}$ in ${\bf SP}(2)$. Since $[\mathrm{Sym}_2(V), \mathrm{Sym}_2(V)]=0$, $\mathbb{H}$ is split and ${\bf SP}(2)/\mathbb{H}\simeq\mathbb{G}_a^-$. }  
	\end{rem}
	\begin{rem}\label{[p, p] is simple}
		{\rm Similarly, the Lie superalgebra $\mathfrak{sp}(n)=[\mathfrak{p}(n), \mathfrak{p}(n)]$ is simple in odd characteristic too, provided $n>2$, but $\mathfrak{sp}(2)$ has a proper superideal
			$I\subseteq\mathfrak{sp}(2)$,
			such that $\mathfrak{sp}(2)/I$ is purely odd one-dimensional superalgebra. }
	\end{rem}
	\subsection{Queer supergroups}
	
	Let $\mathbb{G}$ be a queer supergroup ${\bf Q}(n)$. Recall that for any superalgebra $A$ there is
	\[ {\bf Q}(n)(A)=\{\left(\begin{array}{cc}
		S & S' \\
		-S' & S
	\end{array}\right)\mid S\in \mathrm{GL}_n(A_{\bar 0}), S'\in \mathrm{Mat}_n(A_{\bar 1})  \}, \text{ for } A\in\mathsf{SAlg}_{\Bbbk}. \]
	It is clear that ${\bf Q}(n)$ is a closed supersubgroup of $\mathrm{GL}(n|n)$, hence it is an algebraic supergroup.
	
	For the sake of simplicity, one can identify each ${\bf Q}(n)(A)$ with the set consisting of couples of matrices $(S|S'), S\in \mathrm{GL}_n(A_{\bar 0}), S'\in \mathrm{Mat}_n(A_{\bar 1})$, that is a group with respect to the multiplication 
	\[(S_1|S_1')(S_2|S_2')=(S_1S_2-S_1'S_2'|S_1S_2'+S_1'S_2).\]
	The Lie superalgebra $\mathfrak{q}(n)=\mathrm{Lie}({\bf Q}(n))\subseteq\mathfrak{gl}(n|n)$ consists of matrices 
	\[\left(\begin{array}{cc}
		A & B \\
		B & A
	\end{array}\right), \text{ where } A, B\in\mathrm{Mat}_n(\Bbbk).\]
	As above, we identify $\mathfrak{q}(n)$ with the set consisting of couples of matrices
	\[(A|B), \text{ where } A, B\in \mathrm{Mat}_n(\Bbbk).\]
	The corresponding Harish-Chandra pair is $(\mathrm{GL}_n, \mathrm{Mat}_n(\Bbbk))$,
	where $\mathrm{GL}_n$ acts on $\mathrm{Mat}_n(\Bbbk)$ by conjugations and the bilinear map (the restriction of Lie bracket on $\mathfrak{q}(n)_{\bar 1}$) is given by
	\[ [B, B']=BB'+B'B, \text{ for all } B, B'\in \mathrm{Mat}_n(\Bbbk).\]
	
	The supergroup ${\bf Q}(n)$ contains a central supersubgroup $\mathbb{D}\simeq\mathbb{G}_m$ consisting of all scalar matrices, i.e.
	\[\mathbb{D}(A)=\{(aI_n \, |\, 0) \, | \, a\in A_{\bar 0}^*\}\]
	for any superalgebra $A$. In terms of Harish-Chandra pairs, $\mathbb{D}$ is represented by the Harish-Chandra pair $(D, 0)$, where $D\simeq G_m$ is the center of $\mathrm{GL}_n$ consisting of scalar matrices as well. Let ${\bf PQ}(n)$ denote ${\bf Q}(n)/\mathbb{D}$, the \emph{projective queer} supergroup. The corresponding Harish-Chandra pair is
	$(\mathrm{PGL}_n, \mathrm{Mat}_n(\Bbbk)))$. 
	
	Let $\mathfrak{pq}(n)$ denote $\mathrm{Lie}({\bf PQ}(n))$. It is clear that $\mathfrak{pq}(n)=\mathfrak{pgl}_n\oplus \mathrm{Mat}_n(\Bbbk)$, where $\mathfrak{pgl}_n=\mathrm{Lie}(\mathrm{PGL}_n)=\mathfrak{gl}_n/\Bbbk I_n$.
	\begin{lm}\label{subpair in PQ}
		$(\mathrm{PGL}_n, \mathfrak{sl}_n)$ is a Harish-Chandra subpair of $(\mathrm{PGL}_n, \mathrm{Mat}_n(\Bbbk))$, that represents a normal supersubgroup of ${\bf PQ}(n)$.
	\end{lm}
	\begin{proof}
		It is clear that $\mathfrak{sl}_n$ is a $\mathrm{PGL}_n$-submodule of $\mathrm{Mat}_n(\Bbbk)$. The action of $\mathrm{PSL}_n\simeq \mathrm{PGL}_n$  on  $\mathrm{Mat}_n(\Bbbk)/\mathfrak{sl}_n$ is uniquely determined by a character of $\mathrm{PSL}_n$, or equivalently, by a character of $\mathrm{SL}_n$. It remains to note that
		$\mathrm{SL}_n$ does not have any nontrivial characters.  	
	\end{proof}	
	Let ${\bf PSQ}(n)$ denote the supergroup determined by the above pair $(\mathrm{PGL}_n, \mathfrak{sl}_n)$.
	\begin{pr}\label{almost-simple}
		The following statements hold :
		\begin{enumerate}
			\item ${\bf PSQ}(n)=\mathcal{D}({\bf PQ}(n))$;
			\item If $n\geq 3$, then 	${\bf PSQ}(n)$ is a SAS-supergroup.
			\item If $n=2$, then ${\bf PSQ}(2)$ contains a normal purely odd supersubgroup $\mathbb{R}$ isomorphic to $(\mathbb{G}_a^-)^3$. Besides, ${\bf PSQ}(2)/\mathbb{R}\simeq\mathrm{PGL}_2$. 
		\end{enumerate}	
	\end{pr}	
	\begin{proof}
		By Lemma \ref{subpair in PQ}, the supergroup ${\bf PQ}(n)/{\bf PSQ}(n)\simeq \mathbb{G}_a^-$ is abelian, hence $\mathcal{D}({\bf PQ}(n))\leq {\bf PSQ}(n)$.  On the other hand, $\mathrm{Lie}(\mathcal{D}({\bf PQ}(n)))$ contains $[\mathfrak{pq}(n), \mathfrak{pq}(n)]=\mathfrak{pgl}_n\oplus\mathfrak{sl}_n$. Thus the Harish-Chandra pair of 
		$\mathcal{D}({\bf PQ}(n))$ has a form $(R , \mathfrak{sl}_n)$, where $R$ contains $\mathcal{D}(\mathrm{PGL}_n)=\mathcal{D}(\mathrm{PSL}_n)=\mathrm{PSL}_n=\mathrm{PGL}_n$.	
		
		If $\mathrm{char}\Bbbk=0$, then \cite[1.1.4]{chengwang} states that $\mathrm{Lie}({\bf PSQ}(n))=\mathfrak{psq}(n)$ is simple, whenever $n\geq 3$. Lemma \ref{SAS is WAS} and Lemma \ref{some simple case} imply $(2)$.   
		
		Let $\mathrm{char}\Bbbk>0$. The group $\mathrm{PGL}_n\simeq \mathrm{PSL}_n$ is SAS, provided $n> 1$. By Proposition \ref{SAPS in positive char} the supergroup ${\bf PSQ}_n$ is SAS if and only if it satisfies the conditions $(1-2)$ from Lemma \ref{a very particular case}.
		
		If $V\subseteq\mathfrak{sl}_n$ is a proper submodule, such that $\mathrm{PGL}_n$ acts trivially on
		$\mathfrak{sl}_n/V$, then $[\mathfrak{pgl}_n, \mathfrak{sl}_n]=\mathfrak{sl}_n\subseteq V$, a contradiction.
		
		Let $V$ be a nonzero $\mathrm{PGL}_n$-submodule of $\mathfrak{sl}_n$, such that $[V, \mathfrak{sl}_n]=0$. Then  $V$ is a nonzero ideal of $\mathfrak{sl}_n$, regarded as an ordinary Lie algebra, that implies
		$V=\mathfrak{sl}_n$ or $V=\Bbbk I_n$. It remains to observe that if $n\geq 3$, then $[\mathfrak{sl}_n, \mathfrak{sl}_n]$ or $[\mathfrak{sl}_n, I_n]$ contain the elements $E_{11}+E_{22}=[E_{12}, E_{21}]$ or
		$2B, B\in\mathfrak{sl}_n\setminus\Bbbk I_n$ respectively, both  are not zero in $\mathfrak{pgl}_n$.
		
		Finally, if $n=2$, then the above arguments show that $[\mathfrak{sl}_2, \mathfrak{sl}_2]=0$ in $\mathfrak{pgl}_2$, hence $(1, \mathfrak{sl}_2)$ represents a normal purely odd supersubgroup.
	\end{proof}
	\begin{rem}\label{psq is simple}
		The Lie superalgebra $\mathfrak{psq}(n)$ is simple in odd characteristic too, provided $n\geq 3$. 	
	\end{rem}
	\subsection{The supergroups $\mathrm{OSp}_{\alpha}(4|2)$}
	
	Set $G=\mathrm{SL}_2\times \mathrm{SL}_2\times \mathrm{SL}_2$. As in the subsection 2.3, let $V$ denote the standard two-dimensional $\mathrm{SL}_2$-module. The algebraic group $G$ acts on $V^{\otimes 3}$ as
	\[(g_1, g_2, g_3)(u\otimes v\otimes w)=g_1 u\otimes g_2 v\otimes g_3 w, \text{ for all }g_1, g_2 ,g_3 \in \mathrm{SL}_2, \text{ and }  u, v, w\in V. \] 
	Recall that $\mathfrak{sl}_2\simeq\mathrm{Sym}_2(V)\simeq L(2)$. More precisely, this isomorphism (of $\mathrm{SL}_2$-modules) can be defined as follows. Let $\langle \,  , \, \rangle$ be a skew-symmetric $\mathrm{SL}_2$-invariant form on $V$, such that $\langle v_1, v_2 \rangle =-\langle v_2, v_1\rangle =1$. Then 
	\[[uv, w]=\frac{1}{2}(\langle u, w \rangle v+\langle v, w\rangle u), \text{ for all } u, v, w\in V.\]
	Using this identification, we determine the $G$-equivariant symmetric bilinear map
	$V^{\otimes 3}\times V^{\otimes 3}\to \mathfrak{sl}_2^{\oplus 3}$ by
	\[[u\otimes v\otimes w,  u'\otimes  v'\otimes w']=\]\[\alpha_1 \langle v, v'\rangle \langle w, w'\rangle uu'\oplus \alpha_2 \langle u, u'\rangle \langle w, w'\rangle vv'\oplus \alpha_3 \langle u, u'\rangle \langle v, v'\rangle ww',\]\[ \text{ for all } u, v, w, u', v', w'\in V.\]
	If $x=\sum_{1\leq i_1, i_2, i_3\leq 2} a_{i_1, i_2, i_3} v_{i_1}\otimes v_{i_2}\otimes v_{i_3}$, then the \emph{length} of $x$ is the number of nonzero coefficients $a_{i_1, i_2, i_3}$. The induction on the length of $x$ shows that $[[x, x], x]=0$ is valid for arbitrary $x$, whenever the following identities hold :
	\begin{enumerate}
		\item $[[x, y], z]+[[y, z], x]+[[x, z], y]=0, \text{ for all } x, y, z\in V^{\otimes 3}$;
		\item $[[x, x], x]=0$ for any $x=u\otimes v\otimes w, \text{ where } u, v, w\in V$.
	\end{enumerate} 
	Since $\langle \, , \, \rangle $ is skew-symmetric, the identity $(2)$ is obviously valid. The proof of the first identity can be found in \cite[Lemma 4.2.1]{imuss}. For the reader's convenience, we will provide this proof here. 
	
	Note that $(1)$ needs to be checked on the indecomposable tensors only. Let
	\[x=x_1\otimes x_2\otimes x_3,\, y=y_1\otimes y_2\otimes y_3, \, z=z_1\otimes z_2\otimes z_3 .\]
	Moreover, following \cite[Lemma 4.2.1]{imuss}, one can assume that any two elements from a triple $x_i, y_i, z_i$ are linearly independent, $1\leq i\leq 3$.
	In particular, we have $z_i=a_i x_i +b_i y_i, \text{ where } a_i, b_i\in\Bbbk\setminus 0, \text{ for } 1\leq i\leq 3$. Without loss of generality, one can also assume that
	$\langle x_1, y_1\rangle =\langle x_2, y_2\rangle =\langle x_3, y_3\rangle =1$.
	Then, 
	\[
	\begin{array}{lcl}
		[[x, y], z] & = &	
		\alpha_1(-a_1 x_1+b_1 y_1)\otimes z_2\otimes z_3+\alpha_2 z_1\otimes (-a_2 x_2+b_2 y_2)\otimes z_3\\[2mm]
		&&+\alpha_3  z_1\otimes z_2\otimes (-a_3 x_3+b_3 y_3).
	\end{array}
	\]
	Similarly, we obtain
	\[
	\begin{array}{lcl} [[y, z], x] & = &-\alpha_1a_2 a_3(a_1 x_1+2b_1 y_1)\otimes x_2\otimes x_3-\alpha_2 a_1 a_3x_1\otimes (a_2 x_2+2b_2 y_2)\otimes x_3\\[2mm]
		&&-\alpha_3 a_1 a_2x_1\otimes x_2\otimes (a_3 x_3+2b_3 y_3)
	\end{array}\]
	and 
	\[
	\begin{array}{lcl}
		[[x, z], y] & = & \alpha_1b_2 b_3(2a_1 x_1+b_1 y_1)\otimes y_2\otimes y_3+\alpha_2b_1 b_3 y_1\otimes (2a_2 x_2+b_2 y_2)\otimes y_3\\[2mm]
		&&+\alpha_3b_1 b_2 y_1\otimes y_2\otimes (2a_3 x_3+b_3 y_3).
	\end{array}
	\]
	Summing all up, one sees that
	\[[[x, y], z]+[[y, z], x]+[[x, z], y]=\]
	\[-2(\alpha_1+\alpha_2+\alpha_3)a_1a_2a_3 x_1\otimes x_2\otimes x_3-(\alpha_1+\alpha_2+\alpha_3)a_1b_2a_3 x_1\otimes y_2\otimes x_3\]	
	\[-(\alpha_1+\alpha_2+\alpha_3)a_1a_2b_3 x_1\otimes x_2\otimes y_3+(\alpha_1+\alpha_2+\alpha_3)a_1b_2b_3 x_1\otimes y_2\otimes y_3\]
	\[-(\alpha_1+\alpha_2+\alpha_3)b_1a_2a_3 y_1\otimes x_2\otimes x_3+(\alpha_1+\alpha_2+\alpha_3)b_1b_2a_3 y_1\otimes y_2\otimes x_3\]
	\[+(\alpha_1+\alpha_2+\alpha_3)b_1a_2b_3 y_1\otimes x_2\otimes y_3+2(\alpha_1+\alpha_2+\alpha_3)b_1b_2b_3 y_1\otimes y_2\otimes y_3=0\]
	if and only if $\alpha_1+\alpha_2+\alpha_3=0$. 
	\begin{lm}\label{HC of D(2, 1, a)}
		The above defined $(G, V^{\otimes 3})$ is a Harish-Chandra pair if and only if $\alpha_1+\alpha_2+\alpha_3=0$. The corresponding supergroup is denoted by 
		${\bf \Gamma}(\alpha_1, \alpha_2, \alpha_3)$.	
	\end{lm}
	Note that \[{\bf\Gamma}(\alpha_1, \alpha_2, \alpha_3)\simeq {\bf\Gamma}(\alpha_{\sigma(1)}, \alpha_{\sigma(2)}, \alpha_{\sigma(3)}) \ \mbox{and} \ {\bf\Gamma}(\alpha_1, \alpha_2, \alpha_3)\simeq
	{\bf\Gamma}(\lambda\alpha_1, \lambda\alpha_2, \lambda\alpha_3)\] for any permutation $\sigma\in S_3$ and any $\lambda\in\Bbbk\setminus 0$. In particular, we have a one-parameter family
	of supergroups ${\bf\Gamma}(\alpha, 1, -1-\alpha), \text{ where } \alpha\in\Bbbk$. Following  \cite{l1, sergan1} we denote the latter supergroup by $\mathrm{OSp}_{\alpha}(4|2)$, as well as its Lie superalgebra by
	$\mathfrak{osp}_{\alpha}(4|2)$. In the literature this Lie superalgebra is also known as $\mathrm{D}(2, 1, \alpha)$. 
	
	Since the Lie superalgebra $\mathfrak{osp}_{\alpha}(4|2)$ is simple, provided $\alpha\neq 0, -1$ (cf. \cite[Sections 10-12]{BGL4}), Lemma \ref{some simple case} implies that $\mathrm{OSp}_{\alpha}(4|2)$ is WAS, but it is not SAS in positive characteristic. 
	\subsection{The exceptional supergroup $\mathrm{AG}(2)$}
	
	Set $G=\mathrm{SL}_2\times \mathrm{G}_2$. Recall that the \emph{Chevalley group} $\mathrm{G}_2$ is isomorphic to the automorphism group of the algebra ${\bf O}$ of split octonions (see \cite{cart, sss} for more detail).  
	Let $t$ and $n$ denote the \emph{trace} and the \emph{norm} on ${\bf O}$, so that any element $a\in {\bf O}$ satisfies 
	\[a^2-t(a)a+n(a)=0.\]   
	
	The group $G$ acts naturally on the tensor product $W=V\otimes {\bf O}_0$, where ${\bf O}_0$ is the subspace  of traceless octonions, that is in turn an irreducible
	$\mathrm{G}_2$-module of the highest (first fundamental) weight $\omega_1$. We have
	\[V^{\otimes 2}=\mathrm{Sym}_2(V)\oplus\Lambda^2(V)=L(2)\oplus L(0).\]
	Similarly, ${\bf O}_0^{\otimes 2}=\mathrm{Sym}_2({\bf O}_0)\oplus\Lambda^2({\bf O}_0)$, where the $\mathrm{G}_2$-module $\mathrm{Sym}_2({\bf O}_0)$ is isomorphic to $L(0)\oplus L(2\omega_1)$ , whenever $\mathrm{char}\Bbbk=p\neq 7$, otherwise it is indecomposable with composition factors $L(0)$ and $L(2\omega_1)$. Moreover, in the latter case $L(0)$ is isomorphic to both the top and the socle of $\mathrm{Sym}_2({\bf O}_0)$, and $L(2\omega_1)$ appears in the middle (see \cite[Remark 5 and Proposition 5]{zubs}).  
	Further, $\Lambda^2({\bf O}_0)\simeq L(\omega_1)\oplus L(\omega_2)$, provided $p\neq 3$, otherwise it is indecomposable with composition factors $L(\omega_1)$ and $L(\omega_2)$, so that $L(\omega_1)$ is isomorphic to both the top and the socle of $\Lambda^2({\bf O}_0)$, and $L(\omega_2)$ appears in the middle.  
	\begin{rem}\rm{
			In \cite{zubs} the fundamental dominant weights are put in the opposite order, i.e. $\omega_1$ is denoted by $\lambda_2$ and $\omega_2$ is denoted by $\lambda_1$.  }	
	\end{rem}
	Let $\mathfrak{g}$ denote $\mathrm{Lie}(G)\simeq\mathfrak{sl}_2\oplus\mathfrak{g}_2$. 
	\begin{pr}\label{AG(2)}
		We have
		\[\mathrm{Hom}_{G}(\mathrm{Sym}_2(W), \mathfrak{g})=\mathrm{Hom}_{\mathfrak{g}}(\mathrm{Sym}_2(W), \mathfrak{g})\]	
		and both spaces are two-dimensional.
	\end{pr}
	\begin{proof}
		There is
		\[\mathrm{Hom}_{G}(W^{\otimes 2}, \mathfrak{g})\simeq \mathrm{Hom}_{\mathrm{G}}(W^{\otimes 2}, \mathfrak{sl}_2)\oplus \mathrm{Hom}_{G}(W^{\otimes 2}, \mathfrak{g}_2)\simeq\]
		\[\mathrm{Hom}_{\mathrm{SL}_2}(W^{\otimes 2}/(\mathrm{G}_2-1)W^{\otimes 2}, \mathfrak{sl}_2)\oplus \mathrm{Hom}_{\mathrm{G}_2}(W^{\otimes 2}/(\mathrm{SL}_2-1)W^{\otimes 2}, \mathfrak{g}_2)\]
		and
		\[\mathrm{Hom}_{\mathfrak{g}}(W^{\otimes 2}, \mathfrak{g})\simeq \mathrm{Hom}_{\mathfrak{g}}(W^{\otimes 2}, \mathfrak{sl}_2)\oplus \mathrm{Hom}_{\mathfrak{g}}(W^{\otimes 2}, \mathfrak{g}_2)\simeq\]
		\[\mathrm{Hom}_{\mathfrak{sl}_2}(W^{\otimes 2}/\mathfrak{g}_2W^{\otimes 2}, \mathfrak{sl}_2)\oplus \mathrm{Hom}_{\mathfrak{g}_2}(W^{\otimes 2}/\mathfrak{sl}_2W^{\otimes 2}, \mathfrak{g}_2).\]
		By the above remark on a composition series of ${\bf O}_0$, we have 
		\[W^{\otimes 2}/(\mathrm{G}_2-1)W^{\otimes 2}\simeq V^{\otimes 2}\otimes \mathrm{Sym}_2({\bf O}_0)/(\mathrm{G}_2-1)\mathrm{Sym}_2({\bf O}_0),\]
		where $\mathrm{Sym}_2({\bf O}_0)/(\mathrm{G}_2-1)\mathrm{Sym}_2({\bf O}_0)\simeq L(0)=\Bbbk$. Thus the space
		\[\mathrm{Hom}_{\mathrm{SL}_2}(W^{\otimes 2}/(\mathrm{G}_2-1)W^{\otimes 2}, \mathfrak{sl}_2)=\mathrm{Hom}_{\mathrm{SL}_2}(L(0)\oplus L(2), \mathfrak{sl}_2)=\mathrm{Hom}_{\mathrm{SL}_2}(L(2), \mathfrak{sl}_2)\]
		is one-dimensional. Since the fundamental weights  are always $p$-restricted, the aforementioned modules are irreducible as $\mathfrak{g}_2$-modules as well.
		Therefore, we obtain that the space
		\[\mathrm{Hom}_{\mathfrak{sl}_2}(W^{\otimes 2}/\mathfrak{g}_2W^{\otimes 2}, \mathfrak{sl}_2)=\mathrm{Hom}_{\mathfrak{sl}_2}(L(2), \mathfrak{sl}_2)\]	
		is one-dimensional also, hence
		\[\mathrm{Hom}_{\mathrm{SL}_2}(W^{\otimes 2}/(\mathrm{G}_2-1)W^{\otimes 2}, \mathfrak{sl}_2)=\mathrm{Hom}_{\mathfrak{sl}_2}(W^{\otimes 2}/\mathfrak{g}_2W^{\otimes 2}, \mathfrak{sl}_2)\]
		and morphisms from both spaces factor through $\mathrm{Sym}_2(V)\otimes\mathrm{Sym}_2({\bf O}_0)\subseteq \mathrm{Sym}_2(W)$. 
		
		Finally, we have 
		\[\mathrm{Hom}_{\mathrm{G}_2}(W^{\otimes 2}/(\mathrm{SL}_2-1)W^{\otimes 2}, \mathfrak{g}_2)\simeq \mathrm{Hom}_{\mathrm{G}_2}({\bf O}_0^{\otimes 2}, \mathfrak{g}_2)\]
		as well as 
		\[\mathrm{Hom}_{\mathfrak{g}_2}(W^{\otimes 2}/(\mathfrak{sl}_2W^{\otimes 2}, \mathfrak{g}_2)\simeq \mathrm{Hom}_{\mathfrak{g}_2}({\bf O}_0^{\otimes 2}, \mathfrak{g}_2).\]
		If $p\neq 3$, then 
		and $\mathfrak{g}_2\simeq L(\omega_2)\simeq H^0(\omega_2)\simeq V(\omega_2)$,  otherwise  $\mathfrak{g}_2\simeq V(\omega_2)$ has composition factors $L(\omega_1)$ (as the socle) and $L(\omega_2)$ (as the top).  Using the above description of a submodule structure of ${\bf O}_0^{\otimes 2}$ and the above remark on fundamental weights, we conclude that
		\[\mathrm{Hom}_{\mathrm{G}_2}({\bf O}_0^{\otimes 2}, \mathfrak{g}_2)=\mathrm{Hom}_{\mathfrak{g}_2}({\bf O}_0^{\otimes 2}, \mathfrak{g}_2) \]
		and both spaces are one-dimensional. Here only the case $p=3$ needs an additional remark that $\Lambda^2({\bf O}_0)$ remains indecomposable when regarded as the $\mathfrak{g}_2$-module. In fact, in the notations of \cite[Proposition 4]{zubs}, 
		let $M\simeq V(\omega_2)$ denote the first member of the Weyl filtration of $\Lambda^2({\bf O}_0)$ and let $N\simeq L(\omega_1)$ is its socle. Then by the remark therein the $\mathfrak{g}_2$-orbit of a highest weight vector of $M$ non-trivially intersects $N$, hence $M|_{\mathfrak{g}_2}$ is indecomposable.   
		If $\Lambda^2({\bf O}_0)|_{\mathfrak{g}_2}$ is decomposable, then 
		$\Lambda^2({\bf O}_0)|_{\mathfrak{g}_2}\simeq L(\omega_1)|_{\mathfrak{g}_2}\oplus M|_{\mathfrak{g}_2}$ and the socle of $\Lambda^2({\bf O}_0)|_{\mathfrak{g}_2}$ is isomorphic to $L(\omega_1)|_{\mathfrak{g}_2}^{\oplus 2}$. In particular, $\Lambda^2({\bf O}_0)$ contains two linearly independent highest weight vectors of weight $\omega_1$, which  contradicts \cite[Proposition 4]{zubs}.
		Finally,  all morphisms from both spaces factor through $\Lambda^2(V)\otimes\Lambda^2({\bf O}_0)\subseteq \mathrm{Sym}_2(W)$. 
	\end{proof}
	As it is well known (cf. \cite[Sections 10-12]{BGL4}), there is  $\phi\in \mathrm{Hom}_{\mathfrak{g}}(\mathrm{Sym}_2(W), \mathfrak{g})$, which determines on the superspace $\mathfrak{g}\oplus W$ the structure of a Lie superalgebra, which is simple at least for all $p\neq 3$. By Proposition \ref{AG(2)}, $\phi$ determines the structure of Harish-Chandra pair on $(G, W)$, whence an algebraic supergroup, that is denoted by $\mathrm{AG}(2)$. It is clear that
	$\mathrm{AG}(2)$ is WAS, at least if $p\neq 3$, but not SAS in positive characteristic. 
	
	The Lie superalgebra $\mathfrak{ag}(2)$ is not simple in characteristic $3$.  Indeed, the Lie algebra $\mathfrak{g}_2$ contains a seven-dimensional ideal $I$, which is nothing else but the socle $L(\omega_1)$ of $\mathfrak{g}_2$.  The proof of  Proposition \ref{AG(2)} obviously implies that $[W, W]\subseteq  \mathfrak{sl}_2\oplus I$, and hence $\mathfrak{I}=(\mathfrak{sl}_2\oplus I)\oplus W$ is a proper superideal of $\mathfrak{ag}(2)$.    
	Nevertheless, the group $\mathrm{AG}(2)$ remains WAS. In fact, let $(H, U)$ be a Harish-Chandra (sub)pair of some infinite
	smooth connected normal supersubgroup of $\mathrm{AG}(2)$.  Then $H$ is $\mathrm{SL}_2, \mathrm{G}_2$ or their product. 
	For example, let $H=\mathrm{SL}_2$.  Since $W|_{\mathrm{SL}_2}$ does not have any nonzero trivial factors, we obtain $U=W$.  Then $[U, U]=[W, W]\not\subseteq \mathfrak{sl}_2$, a contradiction.  The rest cases are  similar.   
	\begin{rem}\label{supersubalgebra in ag(2)} \rm{
			The Lie superalgebra $\mathfrak{ag}(2)$ of $\mathrm{AG}(2)$ is also denoted by $\mathrm{G}(3)$. In characteristic $3$, the superideal $\mathfrak{I}$ is a simple Lie superalgebra, first introduced in \cite[Theorem 4.22(i)]{eld}. However, there is no smooth connected supergroup $\mathbb{H}$, such that $\mathfrak{H}\simeq\mathfrak{I}$. Indeed, assume the opposite. Then $\mathrm{Lie}(H)=\mathfrak{H}_{\bar 0}\simeq\mathfrak{sl}_2\oplus\mathfrak{psl}_3$ 
			is a $p$-restricted Lie algebra, and
			since both direct terms are simple Lie algebras, $\mathfrak{H}_{\bar 0}$ contains no nonzero nilpotent ideals, and therefore no $p$-nilpotent ideals either. 
			By \cite[Corollary 4.3, II, \S 7]{dg}, $H_1$ contains no nontrivial normal unipotent subgroups, i.e. $H_1$ is $p$-reductive. Then \cite[Proposition 11.8]{ham2} infers that $H$ is a reductive group. According to \cite[Corollary 17.62 (b)]{milne}, $\mathrm{R}(H)=\mathrm{Z}(H)_t$ (the largest subtorus of
			$\mathrm{Z}(H)$), and since $\mathfrak{H}_{\bar 0}$ does not have nonzero central ideals, $\mathrm{Lie}(\mathrm{R}(H))=0$. 
			By Lemma \ref{Lie superalgebra is trivial}, $\mathrm{R}(H)=e$ and by \cite[Theorem 21.51]{milne} the semisimple group $H$ is generated by its almost-simple normal subgroups, at least one of which has the Lie algebra isomorphic to $\mathfrak{psl}_3$. The latter contradicts Proposition \ref{not algebraic}.}
	\end{rem}
	
	\subsection{The exceptional supergroup $\mathrm{AB}(3)$}
	
	Set $G=\mathrm{SL}_2\times\mathrm{Spin}(7)$. Let $q$ denote the symmetric bilinear form on ${\bf O}$, associated with the norm $n$. Then $\mathrm{SO}(8)=\mathrm{SO}(q)$ and
	$\mathrm{SO}(7)\simeq\mathrm{Stab}_{\mathrm{SO}(8)}(1_{\bf O})$. Moreover, the canonical epimorphism $\pi : \mathrm{SO}(8)\to\mathrm{PSO}(8)$ maps $\mathrm{SO}(7)$ isomorphically onto
	a closed subgroup of $\mathrm{PSO}(8)$. If $\rho$ is the \emph{graph (or triality)} automorphism of $\mathrm{PSO}(8)$ of order $3$, then $\pi^{-1}(\rho(\mathrm{SO}(7)))\simeq \mathrm{Spin}(7)$ (cf. \cite{zubs}).  Observe that $\mathrm{Lie}(\mathrm{Spin}(7))=d\rho(\mathfrak{so}(7))\simeq\mathfrak{so}(7)$ and ${\bf O}$, regarded as a \emph{twisted} $\mathfrak{so}(7))$-module, has the highest weight $\omega_3$ (denoted in \cite{zubs} by $\lambda_3'$). In other words, it is nothing else but the \emph{spinor representation} of
	$\mathfrak{so}(7)$, as well as of $\mathrm{Spin}(7)$ (cf. \cite{BGL3, schwarz, zubs}).  
	
	Consider $W=V\otimes {\bf O}$ as a natural $\mathrm{SL}_2\times\mathrm{Spin}(7)$-module. The center $\{\pm I_8\}$ of $\mathrm{Spin}(7)$ acts on ${\bf O}^{\otimes 2}$ trivially, hence
	${\bf O}^{\otimes 2}$ has the natural structure of $\mathrm{SO}(7)$-module. 
	Using \cite[Proposition 2, Propositions 14 and 15]{zubs}, one immediately sees 
	\[\mathrm{Sym}_2({\bf O})\simeq L(0)\oplus L(2\omega_3), \ \Lambda^2({\bf O})\simeq L(\omega_1)\oplus L(\omega_2).\]
	Let $\mathfrak{g}$ denote $\mathrm{Lie}(G)\simeq\mathfrak{sl}_2\oplus\mathfrak{so}(7)$.
	\begin{pr}\label{AB(3)}
		We have
		\[\mathrm{Hom}_{G}(\mathrm{Sym}_2(W), \mathfrak{g})=\mathrm{Hom}_{\mathfrak{g}}(\mathrm{Sym}_2(W), \mathfrak{g})\]	
		and both spaces are two-dimensional.	
	\end{pr}
	\begin{proof}
		Using the above comments and the fact that $\mathfrak{so}(7)\simeq L(\omega_2)$, one can mimic the proof of Proposition  \ref{AG(2)}.	
	\end{proof}
	Again, there is  $\phi\in \mathrm{Hom}_{\mathfrak{g}}(\mathrm{Sym}_2(W), \mathfrak{g})$, which determines on the superspace $\mathfrak{g}\oplus W$ the structure of a Lie superalgebra (see \cite[Sections 10-12]{BGL4}, or use an \emph{integral form} of this superalgebra from \cite{gavfi1} and the reduction modulo $p$), which is simple except the case $p=3$. 
	By Proposition \ref{AB(3)}, $\phi$ determines the structure of Harish-Chandra pair on $(G, W)$. The corresponding supergroup is denoted by $\mathrm{AB}(3)$. Since its Lie superalgebra $\mathfrak{ab}(3)$ is simple for any $p\neq 3$, $\mathrm{AB}(3)$ is WAS, whenever $p\neq 3$. It is also not SAS in any positive characteristic. 
	
	Note that any map from
	$\mathrm{Hom}_{G}(W^{\otimes 2}, \mathfrak{g})$ can be constructed as a linear combination of two independent $G$-equivariant (obviously symmetric) maps
	\[\phi_1 : (u\otimes x)\otimes (v\otimes y)\mapsto q(x, y)uv \ \mbox{and} \ \phi_2 : (u\otimes x)\otimes (v\otimes y)\mapsto \langle u, v\rangle  D_{x, y},\]
	where $u, v\in V, x, y\in {\bf O}$, $\langle \ , \ \rangle$ is the skew-symmetric form from subsection 3.4, and
	$uv\in\mathrm{Sym}_2(V)$ is identified with the element of $ \mathfrak{sl}_2$ as therein, as well as $D_{x, y}$ is the image of $x\wedge y$ under the projection $\Lambda^2({\bf O})\to\mathrm{Lie}(\mathrm{Spin}(7))= d\rho(\mathfrak{so}(7))\simeq \mathfrak{so}(7)$. 
	
	One can construct a 
	$\mathrm{Spin}(7)$-invariant polynomial $f$ in four variables, such that $f(x, y, z, t)=q(D_{x, y}(z), t)$ for any $x, y, z, t\in {\bf O}$, which is skew-symmetric in $x$ and $y$. Then \cite[Theorem 4]{zubs} infers 
	\[f(x, y, z, t)=\alpha q(x, y)q(z, t)+\beta q(x, z)q(y, t) +\gamma q(x, t)q(y, z)+\delta Q(x, y, z, t),\]
	where $Q(x, y, z, t)$ is skew-symmetric in all variables. Note that the definition of the polynomial $Q$ from \cite{zubs}, as well as the equivalent polynomial  
	from \cite{schwarz}, is not correct.  In the Appendix, we provide a correct definition.
	
	The polynomials 
	\[q(x, y)q(z, t), \ q(x, z)q(y, t), \ q(x, t)q(y, z), \ Q(x, y, z, t)\]
	are linearly independent. In particular, $f(x, y, z, t)=-f(y, x, z, t)$ implies $\alpha=0$ and $\beta=-\gamma$, that is
	\[f(x, y, z, t)=\beta(q(x, z)q(y, t) -q(x, t)q(y, z))+\delta Q(x, y, z, t).\]
	
	The projection $\Lambda^2({\bf O})\to d\rho(\mathfrak{so}(7))$ vanishes on the direct term $L(\omega_1)$, so that if $\sum_i x_i\wedge y_i$ is its highest weight vector, then $\sum_i f(x_i, y_i, z, t)$ is equal to zero for any $z, t\in {\bf O}$. 
	Using this fact, we obtain in Appendix that $\beta=3\delta$. In particular, if $p=3$, then $f(x, y, z, t)=\delta Q(x, y, z, t), \delta\in\Bbbk\setminus 0$.
	\begin{pr}\label{ab(3) in char=3}
		Let $p=3$ and $\phi=a\phi_1+b\phi_2\in\mathrm{Hom}_{\mathfrak{g}}(W^{\otimes 2}, \mathfrak{g})\setminus 0$ determines the structure of Lie superalgebra on $\mathfrak{ab}(3)$.  Then $a=0$ and $\mathfrak{so}(7)\oplus W$ is a proper ideal in $\mathfrak{ab}(3)$. Similarly, the Harish-Chandra subpair $(\mathrm{Spin}(7), W)$ determines a smooth connected normal supersubgroup $\mathbb{H}$ in $\mathrm{AB}(3)$. Furthermore, $\mathbb{H}$ is SAS and 
		$\mathrm{AB}(3)/\mathbb{H}\simeq\mathrm{SL}_2$.
	\end{pr}
	\begin{proof}
		Note that $Q(x, y, x, t)=0$ implies that $D_{x, y}(x)=0$ for any $x, y\in {\bf O}$. 
		
		Set $v=\alpha u\otimes x+\beta v\otimes y$, where $u, v\in V, \ x, y\in {\bf O}, \langle u, v\rangle =1, q(x, x)=q(y, y)=1, q(x, y)=0, \alpha, \beta\in\Bbbk\setminus 0$. Then
		\[[v, v]=a\alpha^2 u^2 +2\alpha\beta bD_{x, y}+a\beta^2 v^2\]
		and
		\[ [[v, v], v]=a(\alpha^2\beta u\otimes y-\alpha\beta^2 v\otimes x)=0\]
		infers $a=0$. 
		
		The Lie supersubalgebra $\mathfrak{so}(7)\oplus W$ is known as the simple Lie superalgebra ${\mathfrak g}(3,3)^{(1)}/{\mathfrak z}$ introduced in \cite[Theorem 4.22(ii)]{eld} (see also \cite{BGL4} for the description of the corresponding Cartan matrices). This Lie superalgebra occurs only in characteristic 3. 
		
		Note that the $\mathrm{Spin}(7)$-module $W$ is isomorphic to $L(\omega_3)^{\dim V}$. In particular, $W$ does not have nonzero trivial factors. Let $U\subseteq W$ be a nonzero $\mathrm{Spin}(7)$-submodule of $W$, such that $[U, W]=0$. Then $U\simeq L(\omega_3)^l, 0< l\leq\dim V$, so that for some $u\in V\setminus 0$ the submodule $U$ contains all elements $u\otimes x, x\in {\bf O}\setminus 0$. Choose $v\otimes y$, such that $D_{x, y}\neq 0$ and $\langle u, v\rangle =1$. We have 
		$[u\otimes x, v\otimes y]=bD_{x, y}\neq 0$,  which contradicts our assumption. Proposition \ref{SAPS in positive char} implies that $\mathbb{H}$ is a SAS-supergroup.
	\end{proof}
	\begin{rem}\rm{
			Recall that the Lie superalgebra $\mathfrak{ab}(3)$ is also known as $\mathrm{F}(4)$. Also, noting the connection with the superalgebra $\mathfrak{g}(3, 3)^{(1)}/\mathfrak{z}$, we denote $\mathbb{H}$ by $\mathrm{G}(3, 3)$.}
	\end{rem}
	
	\section{ Another one new $10|12$-dimensional SAS-supergroup}
	
	In this section, we show that the simple Lie superalgebra $\mathfrak{brj}(2;5)$ of super-dimension $10|12$, constructed in \cite{BGL3}, is algebraic. This means that there is a SAS-supergroup that has a Lie superalgebra isomorphic to $\mathfrak{brj}(2;5)$. In particular, this supergroup does not belong to Kac's list of known almost-simple supergroups.
	
	Let $G$ be the symplectic group $\mathrm{Sp}_4$. Recall that for any commutative algebra $R$, the group $G(R)$ consists of matrices of the form: 
	\[\left(\begin{array}{cc}
		A & B
		\\
		C & D	
	\end{array}\right), \text{ where } A, B, C, D\in \mathrm{M}_2(R), \text{ and } \]
	\[A^t C-C^t A=0, \ A^t D-C^t B=I_2, \ B^tD-D^t B=0.\]
	We fix the maximal torus $T$ consisting of matrices
	\[\left(\begin{array}{cc}
		A & 0 
		\\
		0 & A^{-1}	
	\end{array}\right), \ \]
	where 
	\[A=\left(\begin{array}{cc}
		t_1 & 0
		\\
		0 & t_2	
	\end{array}\right), \ t_1, t_2\in R^* .\]
	The Lie algebra $\mathfrak{g}=\mathrm{Lie}(G)$ consists of matrices of the form:
	\[\left(\begin{array}{cc}
		A & B
		\\
		C & D	
	\end{array}\right), \text{ where } A, B, C, D\in\mathrm{M}_2(\Bbbk), \text{ and }\]
	\[A=-D^t, \ B=B^t, C=C^t.\]
	
	Let $\epsilon_1, \epsilon_2$ be a standard basis of the character group $X(T)\simeq \mathrm{Z}^2$. Then we have the root decompostion
	\[\mathfrak{g}=\mathrm{Lie}(T)\oplus (\oplus_{\alpha\in\Delta}\Bbbk x_{\alpha}),\]
	where $\Delta=\{\pm 2\epsilon_1, \pm 2\epsilon_2, \pm(\epsilon_1\pm\epsilon_2)\}$ and 
	\[x_{2\epsilon_1}=E_{13}, \ x_{2\epsilon_2}=E_{24}, \ x_{\epsilon_1-\epsilon_2}=E_{12}-E_{43}, \ x_{\epsilon_1+\epsilon_2}=E_{14}+E_{23},\]
	\[x_{-2\epsilon_1}=E_{31}, \ x_{-2\epsilon_2}=E_{42}, \ x_{-\epsilon_1+\epsilon_2}=E_{21}-E_{34}, \ x_{-\epsilon_1-\epsilon_2}=E_{41}+E_{32},\]
	\[\mathrm{Lie}(T)=\Bbbk h_1\oplus \Bbbk h_2, \ h_1=E_{11}-E_{33}, \ h_2=E_{22}-E_{44}. \]

	Let $X_{\alpha}$ denote the root subgroup of $G$, such that $\mathrm{Lie}(X_{\alpha})=\Bbbk x_{\alpha}, \alpha\in\Delta$. Since each $x_{\alpha}$ satisfies $x_{\alpha}^2=0$, the group $X_{\alpha}(R)$ consits of matrices $I_4+tx_{\alpha}, t\in R$. 
	
	If $V$ is a $G$-module and $v\in V_{\lambda}$, then for any $\alpha\in\Delta$ there is 
	\[X_{\alpha}(t)(v)=\sum_{i\geq 0} t^i v_{\lambda+i\alpha}, \] where each $v_{\lambda+i\alpha}$ belongs to $V_{\lambda+i\alpha}$, $v_{\lambda}=v$ and $v_{\lambda+\alpha}=x_{\alpha}(v)$
	(cf. \cite[Proposition 22.14]{milne}).
	
	We choose the set of positive roots $\Delta^+ = \{2\epsilon_1, 2\epsilon_2, \epsilon_1\pm\epsilon_2\}$. Then the simple roots are $\alpha_1=\epsilon_1-\epsilon_2, \alpha_2=2\epsilon_2$ and the fundamental weights are $\omega_1=\epsilon_1, \omega_2=\epsilon_1+\epsilon_2$. 
	Recall also that $G$ is generated by $X_{\pm\alpha_1}, X_{\pm\alpha_2}$ and $T$ (see \cite[Theorem 27.3]{ham}, or \cite[Theorem 21.62]{milne}).
	
	Let $L(\lambda)$ denote an irreducible $G$-module of the highest weight $\lambda$. The weight $\lambda$ is dominant, i.e. $\lambda=c_1\omega_1+c_2\omega_2, \text{where } c_1, c_2\geq 0$ (cf. \cite[Section 31]{ham}). Equivalently, a weight $\lambda=l_1\epsilon_1+l_2\epsilon_2$ is domoninant if and only if $l_1\geq l_2\geq 0$.
	
	The standard $4$-dimensional representation $V$ of $G$ is the irreducible $G$-module $L(\omega_1)$ of the highest weight $\omega_1$. Its standard basic vectors $v_1, v_2, v_3, v_4$ have weights $\epsilon_1, \epsilon_2, -\epsilon_1, -\epsilon_2$ respectively. In what follows we denote them by
	$v_{\epsilon_1}, v_{\epsilon_2}, v_{-\epsilon_1}, v_{-\epsilon_2}$ correspondingly. Then the action of $G$ on $V$ is uniquely defined by
	\[X_{\alpha_1}(t)(v_{-\epsilon_1})=v_{-\epsilon_1}-tv_{-\epsilon_2}, \ X_{\alpha_2}(t)(v_{-\epsilon_2})=v_{-\epsilon_2}+tv_{\epsilon_2}, \ X_{\alpha_1}(t)(v_{\epsilon_2})=v_{\epsilon_2}+tv_{\epsilon_1},\]
	\[X_{-\alpha_1}(t)(v_{\epsilon_1})=v_{\epsilon_1}+tv_{\epsilon_2}, \ X_{-\alpha_2}(t)(v_{\epsilon_2})=v_{\epsilon_2}+tv_{-\epsilon_2}, \ X_{-\alpha_1}(t)(v_{-\epsilon_2})=v_{-\epsilon_2}-tv_{-\epsilon_1},\]
	and by the fact that for any other couple of a simple root $\beta$ and a basic vector $v_{\gamma}$, we have $X_{\beta}(v_{\gamma})=v_{\gamma}$. 
	
	These formulas are particular case of the famous Matsumoto's theorem (cf. \cite{mats}). More precisely, if $L(\pi)$ is so-called \emph{basic} (irreducible) representation in the sense of  \cite{plotsemvav}, then there is a weight basis $\{v_{\lambda}\}$ of $L(\pi)$, such that the action of a root subgroup $X_{\alpha}$  is given as : 
	\begin{enumerate}
		\item if $\lambda\neq 0$ and $L(\pi)_{\lambda+\alpha}=0$, then $X_{\alpha}(t)(v_{\lambda})=v_{\lambda}$;
		\item if $\lambda, \lambda+\alpha\neq 0$ and $L(\pi)_{\lambda+\alpha}\neq 0$, then $X_{\alpha}(t)(v_{\lambda})=v_{\lambda}\pm  tv_{\lambda+\alpha}$;
		\item if $L(\pi)_{\alpha}=0$, then $X_{\alpha}(t)(v_0)=v_0$;
		\item if $L(\pi)_{\alpha}\neq 0$, then $X_{\alpha}(t)(v_0)=v_0\pm t\alpha_*(v_0)v_{\alpha}$ and $X_{\alpha}(t)(v_{-\alpha})=v_{-\alpha}\pm tv_0(\alpha)\pm t^2v_{\alpha}$, where
		$\alpha_*\in L(\pi)^*, v_0(\alpha)\in L(\pi)$.
	\end{enumerate}
	Besides, by the above remark, one needs to consider the action of $X_{\pm\alpha}$ only, where  $\alpha$ runs over simple roots.
	
	The weights $\omega_1$ and $\omega_2$ are the highest weights of basic (even \emph{minimal}, see \cite[Table 2]{plotsemvav}) representations. The irreducible $G$-module
	$L(\omega_2)$ seats in the following exact sequence 
	\[0\to \Bbbk\to \Lambda^2(V)\to L(\omega_2)\to 0,\]
	where the leftmost trivial module is spanned by the vector $v_{\epsilon_1}\wedge v_{-\epsilon_1}+v_{\epsilon_2}\wedge v_{-\epsilon_2}$. The module $L(\omega_2)$ is \emph{multiplicity-free} and its weight decomposition is
	\[L(\omega_2)_{\omega_2}=\oplus_{\lambda\in\{\pm(\epsilon_1+\epsilon_2), 0, \pm(\epsilon_1-\epsilon_2)\}} \Bbbk w_{\lambda},\]
	where
	\[w_{\epsilon_1+\epsilon_2}=v_{\epsilon_1}\wedge v_{\epsilon_2}, \ w_{\epsilon_1-\epsilon_2}=v_{\epsilon_1}\wedge v_{-\epsilon_2}, \ w_0=v_{\epsilon_1}\wedge v_{-\epsilon_1},\]
	\[w_{-\epsilon_1+\epsilon_2}=v_{-\epsilon_1}\wedge v_{\epsilon_2}, \ w_{-\epsilon_1-\epsilon_2}=v_{-\epsilon_1}\wedge v_{-\epsilon_2}.\]  
	Following the weight diagram of $L(\omega_2)$ (see \cite[Fig.15]{plotsemvav}), the action of $G$ on $L(\omega_2)$ can be described by the following formulas:
	\[
	\begin{array}{lcl}X_{\alpha_2}(t)(w_{-\epsilon_1-\epsilon_2})& = & w_{-\epsilon_1-\epsilon_2}+tw_{-\epsilon_1+\epsilon_2},\\[1mm]
		X_{\alpha_1}(t)(w_{-\epsilon_1+\epsilon_2})& = &w_{-\epsilon_1+\epsilon_2}-2tw_0+t^2 w_{\epsilon_1-\epsilon_2},\\[1mm]
		X_{\alpha_1}(t)(w_0) & = & w_0-tw_{\epsilon_1-\epsilon_2},\\[1mm]
		X_{\alpha_2}(t)(w_{\epsilon_1-\epsilon_2}) & = &w_{\epsilon_1-\epsilon_2}+tw_{\epsilon_1+\epsilon_2},\\[1mm]
		X_{-\alpha_2}(t)(w_{\epsilon_1+\epsilon_2})&=&w_{\epsilon_1+\epsilon_2}+tw_{\epsilon_1-\epsilon_2},\\[1mm]
		X_{-\alpha_1}(t)(w_{\epsilon_1-\epsilon_2}) & = & w_{\epsilon_1-\epsilon_2}-2tw_0+t^2 w_{-\epsilon_1+\epsilon_2},\\[1mm]
		X_{-\alpha_1}(t)(w_0)&=&w_0-tw_{-\epsilon_1+\epsilon_2},\\[1mm]
		X_{-\alpha_2}(t)(w_{-\epsilon_1+\epsilon_2})&=&w_{-\epsilon_1+\epsilon_2}+tw_{-\epsilon_1-\epsilon_2}.
	\end{array}\]
	We will denote $L(\omega_1)\otimes L(\omega_2)$ by $N$ for convenience. Then, $N_{\omega_1+\omega_2}=\Bbbk (v_{\epsilon_1}\otimes w_{\epsilon_1+\epsilon_2})$ and the vector $v_{\epsilon_1}\otimes w_{\epsilon_1+\epsilon_2}$ generates the submodule $M$ of $N$, such that its top is isomorphic to
	$L(\omega_1+\omega_2)=L(2\epsilon_1+\epsilon_2).$ Note also that the weights of $N$ are $\pm 2\epsilon_1\pm\epsilon_2, \pm \epsilon_1\pm 2\epsilon_2$ and $\pm\epsilon_1, \pm\epsilon_2$, obtained by adding the weights of $L(\omega_1)$ to the weights of $L(\omega_2)$.

	Moreover, the first eight weights have multiplicity one, but the rest ones have multiplicity three.
	\begin{lm}\label{a required simple}
		If $\mathrm{char}\Bbbk =5$, then $\mathrm{soc}(N)=\mathrm{soc}(M)=\mathrm{rad}(M)\simeq L(\omega_1)$ and $M/\mathrm{rad}(M)\simeq L(\omega_1+\omega_2)$, but
		$N/M\simeq L(\omega_1)$.  Moreover, $\mathrm{soc}(M)$ has a (weight) basis
		\[v_{\epsilon_1}\otimes w_0 +v_{\epsilon_2}\otimes w_{\epsilon_1-\epsilon_2}-v_{-\epsilon_2}\otimes w_{\epsilon_1+\epsilon_2},\]
		\[-v_{\epsilon_2}\otimes w_0 -v_{\epsilon_1}\otimes w_{-\epsilon_1+\epsilon_2}+v_{-\epsilon_1}\otimes w_{\epsilon_1+\epsilon_2},\]
		\[-v_{-\epsilon_2}\otimes w_0 -v_{\epsilon_1}\otimes w_{-\epsilon_1-\epsilon_2}+v_{-\epsilon_1}\otimes w_{\epsilon_1-\epsilon_2},\]
		\[v_{-\epsilon_1}\otimes w_0 +v_{\epsilon_2}\otimes w_{-\epsilon_1-\epsilon_2}-v_{-\epsilon_2}\otimes w_{-\epsilon_1+\epsilon_2}.\]
		Similarly, $M/\mathrm{soc}(M)\simeq L(\omega_1+\omega_2)$ has a basis consisting of the vectors
		\[u_{2\epsilon_1+\epsilon_2}=v_{\epsilon_1}\otimes w_{\epsilon_1+\epsilon_2},\, u_{2\epsilon_1-\epsilon_2}=v_{\epsilon_1}\otimes w_{\epsilon_1-\epsilon_2},\]
		\[u_{\epsilon_1+2\epsilon_2}=v_{\epsilon_2}\otimes w_{\epsilon_1+\epsilon_2}, \, u_{\epsilon_1-2\epsilon_2}=v_{-\epsilon_2}\otimes w_{\epsilon_1-\epsilon_2},\]
		\[u_{-2\epsilon_1+\epsilon_2}=v_{-\epsilon_1}\otimes w_{-\epsilon_1+\epsilon_2}, \, u_{-2\epsilon_1-\epsilon_2}=v_{-\epsilon_1}\otimes w_{-\epsilon_1-\epsilon_2}\]
		\[u_{-\epsilon_1-2\epsilon_2}=v_{-\epsilon_2}\otimes w_{-\epsilon_1-\epsilon_2},\, u_{-\epsilon_1+2\epsilon_2}=v_{\epsilon_2}\otimes w_{-\epsilon_1+\epsilon_2}\]
		of weights  $\pm 2\epsilon_1\pm\epsilon_2, \pm \epsilon_1\pm 2\epsilon_2$, and the vectors
		\[u_{\epsilon_1}=v_{\epsilon_2}\otimes w_{\epsilon_1-\epsilon_2}-2v_{\epsilon_1}\otimes w_0,\]
		\[u_{\epsilon_2}=v_{\epsilon_1}\otimes w_{-\epsilon_1+\epsilon_2}-2v_{\epsilon_2}\otimes w_0,\]
		\[u_{-\epsilon_2}=v_{\epsilon_1}\otimes w_{-\epsilon_1-\epsilon_2}-2v_{-\epsilon_2}\otimes w_0,\]
		\[u_{-\epsilon_1}=v_{\epsilon_2}\otimes w_{-\epsilon_1-\epsilon_2}-2v_{-\epsilon_1}\otimes w_0 \]
		of weights $\{\pm\epsilon_1, \pm\epsilon_2\}$.
	\end{lm}
	\begin{proof}
		Among the weights of $N$ only $2\epsilon_1+\epsilon_2$ and $\epsilon_1$ are dominant, hence the possible composition factors 
		of $N$ are only $L(2\epsilon_1+\epsilon_2)$ and $L(\epsilon_1)$, and $[N : L(\omega_1+\omega_2)]=[M : L(\omega_1+\omega_2)]=1$. Observe also that the weights $\pm 2\epsilon_1\pm\epsilon_2, \pm \epsilon_1\pm 2\epsilon_2$ can appear in $L(\omega_1+\omega_2)$ only, hence $\oplus_{\alpha\in\{\pm 2\epsilon_1\pm\epsilon_2, \pm \epsilon_1\pm 2\epsilon_2\}} N_{\alpha}\subseteq M$.  
		Consider a vector 
		\[x=\alpha v_{\epsilon_1}\otimes w_0 +\beta v_{\epsilon_2}\otimes w_{\epsilon_1-\epsilon_2} +\gamma v_{-\epsilon_2}\otimes w_{\epsilon_1+\epsilon_2}\]
		of weight $\epsilon_1$. Then
		\[X_{\alpha_1}(t)(x)=x+t(\beta-\alpha)v_{\epsilon_1}\otimes w_{\epsilon_1-\epsilon_2},\]
		\[X_{\alpha_2}(t)(x)=x+t(\beta+\gamma)v_{\epsilon_2}\otimes w_{\epsilon_1+\epsilon_2}.\]
		Thus the vector 
		\[y_{\epsilon_1}=v_{\epsilon_1}\otimes w_0 +v_{\epsilon_2}\otimes w_{\epsilon_1-\epsilon_2}-v_{-\epsilon_2}\otimes w_{\epsilon_1+\epsilon_2}\]
		is the highest weight vector of an irreducible submodule $S$, that is isomorphic to $L(\omega_1)$. 
		
		Further, we have
		\[X_{-\alpha_1}(t)(y_{\epsilon_1})=y_{\epsilon_1}+ty_{\epsilon_2}, \ X_{-\alpha_2}(t)(y_{\epsilon_2})=y_{\epsilon_2}+ty_{-\epsilon_2}, \ X_{-\alpha_1}(t)(y_{-\epsilon_2})=y_{-\epsilon_2}-t y_{-\epsilon_1},\]
		where
		\[y_{\epsilon_2}=-v_{\epsilon_2}\otimes w_0 -v_{\epsilon_1}\otimes w_{-\epsilon_1+\epsilon_2}+v_{-\epsilon_1}\otimes w_{\epsilon_1+\epsilon_2},\]
		\[y_{-\epsilon_2}=-v_{-\epsilon_2}\otimes w_0 -v_{\epsilon_1}\otimes w_{-\epsilon_1-\epsilon_2}+v_{-\epsilon_1}\otimes w_{\epsilon_1-\epsilon_2},\]
		\[y_{-\epsilon_1}=v_{-\epsilon_1}\otimes w_0 +v_{\epsilon_2}\otimes w_{-\epsilon_1-\epsilon_2}-v_{-\epsilon_2}\otimes w_{-\epsilon_1+\epsilon_2}.\]
		Recall that $L(\omega_1)=H^0(\omega_1)$ and $L(\omega_2)=H^0(\omega_2)$ (see \cite[3.2.3]{hag-mc}).  By Donkin-Mathieu theorem \cite[Proposition II.4.21]{jan} and \cite[2.2.4]{hag-mc}, $N$ has a good filtration  
		with the top quotient $H^0(\omega_1+\omega_2)$. Moreover, the submodule $\ker(N\to H^0(\omega_1+\omega_2))$ has a good filtration with layers isomorphic to $H^0(\omega_1)$ only, which is nothing else but $S$ (use \cite[II.2.12(1)]{jan}). 
		Besides, $H^0(\omega_1+\omega_2)$ is either irreducible or $H^0(\omega_1+\omega_2)/L(\omega_1+\omega_2)\simeq L(\omega_1)$. 
		
		Recall that for any dominant weight $\lambda$, the induced module $H^0(\lambda)$ is irreducible if and only if the Weyl module $V(\lambda)\simeq H^0(-w_0(\lambda))^*$ is. In our case $w_0=-1$, hence $V(\lambda)\simeq H^0(\lambda)^*$. If $\mathrm{char}\Bbbk =5$, then applying \cite[Proposition II.8.19]{jan} to $V(\omega_1+\omega_2)$, one easily sees that the sum $(1)$ therein contains the unique term $\chi(\omega_1)=\mathrm{ch}(H^0(\omega_1))=\mathrm{ch}(L(\omega_1))$, hence $V(\omega_1+\omega_2)$ is not irreducible.

		If $M\cap S=0$, then $\dim M_{\pm\epsilon_i}=1, i=1, 2$. Computing $X_{\alpha_1}(t)(v_{\epsilon_2}\otimes w_{-\epsilon_1+\epsilon_2})$, we obtain
		\[v_{\epsilon_1}\otimes w_{-\epsilon_1+\epsilon_2}-2v_{\epsilon_2}\otimes w_0\in M_{\epsilon_2}, \ v_{\epsilon_2}\otimes w_{\epsilon_1-\epsilon_2}-2v_{\epsilon_1}\otimes w_0\in M_{\epsilon_1}.\] 
		Next, computing $X_{-\alpha_2}(t)(v_{\epsilon_2}\otimes w_{\epsilon_1+\epsilon_2})$, we obtain
		\[v_{-\epsilon_2}\otimes w_{\epsilon_1+\epsilon_2}+v_{\epsilon_2}\otimes w_{\epsilon_1-\epsilon_2}\in M_{\epsilon_1},\]
		that is $\dim M_{\epsilon_1}\geq 2$. This contradiction implies $S\subseteq M$. 
		Furthermore, since $M$ is indecomposable, we have $\mathrm{rad}(M)=\mathrm{soc}(M)=\mathrm{soc}(N)=S$. 
		Finally, calculating $X_{-\alpha_2}(t)(v_{\epsilon_1}\otimes w_{-\epsilon_1+\epsilon_2}-2v_{\epsilon_2}\otimes w_0)$, and then 
		$X_{-\alpha_1}(t)(v_{\epsilon_1}\otimes w_{-\epsilon_1-\epsilon_2}-2v_{-\epsilon_2}\otimes w_0)$, we complete the proof. 
	\end{proof}
	The following formulas describe the action of $X_{\pm\alpha_1}$ and $X_{\pm\alpha_2}$ on $L(\omega_1+\omega_2)$ :
	\[X_{\alpha_1}(t)(u_{-2\epsilon_1-\epsilon_2})=u_{-2\epsilon_1-\epsilon_2}-tu_{-\epsilon_1-2\epsilon_2}, \]
	\[X_{\alpha_2}(t)(u_{-2\epsilon_1-\epsilon_2})=u_{-2\epsilon_1-\epsilon_2}+tu_{-2\epsilon_1+\epsilon_2},\]
	\[X_{\alpha_1}(t)(u_{-2\epsilon_1+\epsilon_2})=u_{-2\epsilon_1+\epsilon_2}-t u_{-\epsilon_1}+t^2 u_{-\epsilon_2}-t^3 u_{\epsilon_1-2\epsilon_2},\]
	\[X_{\alpha_2}(t)(u_{-\epsilon_1-2\epsilon_2})=u_{-\epsilon_1-2\epsilon_2}+2t u_{-\epsilon_1}+t^2 u_{-\epsilon_1+2\epsilon_2},\]
	\[X_{\alpha_1}(t)(u_{-\epsilon_1})= u_{-\epsilon_1}-2 t u_{-\epsilon_2}-2t^2 u_{\epsilon_1-2\epsilon_2}  \]
	\[X_{\alpha_2}(t)(u_{-\epsilon_1})=u_{-\epsilon_1}+t u_{-\epsilon_1+2\epsilon_2},\]
	\[X_{\alpha_1}(t)(u_{-\epsilon_2})=u_{-\epsilon_2}+2t u_{\epsilon_1-2\epsilon_2},\]
	\[X_{\alpha_2}(t)(u_{-\epsilon_2})=u_{-\epsilon_2}+t u_{\epsilon_2},\]
	\[X_{\alpha_1}(t)(u_{-\epsilon_1+2\epsilon_2})=u_{-\epsilon_1+2\epsilon_2}+t u_{\epsilon_2}+t^2 u_{\epsilon_1}+t^3 u_{2\epsilon_1-\epsilon_2},\]
	\[X_{\alpha_1}(t)(u_{\epsilon_2})=u_{\epsilon_2}+2t u_{\epsilon_1}-2t^2 u_{2\epsilon_1-\epsilon_2},\]
	\[X_{\alpha_2}(t)(u_{\epsilon_1-2\epsilon_2})=u_{\epsilon_1-2\epsilon_2}+2t u_{\epsilon_1}+t^2 u_{\epsilon_1+2\epsilon_2},\]
	\[X_{\alpha_1}(t)(u_{\epsilon_1})=u_{\epsilon_1}-2t u_{2\epsilon_1-\epsilon_2},\]
	\[X_{\alpha_2}(t)(u_{\epsilon_1})=u_{\epsilon_1}+t u_{\epsilon_1+2\epsilon_2},\]
	\[X_{\alpha_1}(t)(u_{\epsilon_1+2\epsilon_2})=u_{\epsilon_1+2\epsilon_2}+t u_{2\epsilon_1+\epsilon_2},\]
	\[X_{\alpha_2}(t)(u_{2\epsilon_1-\epsilon_2})=u_{2\epsilon_1-\epsilon_2}+t u_{2\epsilon_1+\epsilon_2};\]
	
	\[X_{-\alpha_1}(t)(u_{2\epsilon_1+\epsilon_2})=u_{2\epsilon_1+\epsilon_2}+t u_{\epsilon_1+2\epsilon_2},\]
	\[X_{-\alpha_2}(t)(u_{2\epsilon_1+\epsilon_2})=u_{2\epsilon_1+\epsilon_2}+t u_{2\epsilon_1-\epsilon_2},\]
	\[X_{-\alpha_1}(t)(u_{2\epsilon_1-\epsilon_2})=u_{2\epsilon_1-\epsilon_2}+tu_{\epsilon_1}+t^2 u_{\epsilon_2}+t^3 u_{-\epsilon_1+2\epsilon_2},\]
	\[X_{-\alpha_2}(t)(u_{\epsilon_1+2\epsilon_2})=u_{\epsilon_1+2\epsilon_2}+2t u_{\epsilon_1}+t^2 u_{\epsilon_1-2\epsilon_2},\]
	\[X_{-\alpha_1}(t)(u_{\epsilon_1})=u_{\epsilon_1}+2t u_{\epsilon_2}-2t^2 u_{-\epsilon_1+2\epsilon_2},\]
	\[X_{-\alpha_2}(t)(u_{\epsilon_1})=u_{\epsilon_1}+t u_{\epsilon_1-2\epsilon_2},\]
	\[X_{-\alpha_1}(t)(u_{\epsilon_1-2\epsilon_2})=u_{\epsilon_1-2\epsilon_2}-tu_{-\epsilon_2}+t^2 u_{-\epsilon_1}-t^3 u_{-2\epsilon_1+\epsilon_2},\]
	\[X_{-\alpha_1}(t)(u_{\epsilon_2})=u_{\epsilon_2}-2t u_{-\epsilon_1+2\epsilon_2},\]
	\[X_{-\alpha_2}(t)(u_{\epsilon_2})=u_{\epsilon_2}+t u_{-\epsilon_2},\]
	\[X_{-\alpha_1}(t)(u_{-\epsilon_2})=u_{-\epsilon_2}-2t u_{-\epsilon_1}-2t^2 u_{-2\epsilon_1+\epsilon_2},\]
	\[X_{-\alpha_2}(t)(u_{-\epsilon_1+2\epsilon_2})=u_{-\epsilon_1+2\epsilon_2}+2t u_{-\epsilon_1}+t^2 u_{-\epsilon_1-2\epsilon_2},\]
	\[X_{-\alpha_1}(t)(u_{-\epsilon_1})=u_{-\epsilon_1}+2t u_{-2\epsilon_1+\epsilon_2}, \]
	\[X_{-\alpha_2}(t)(u_{-\epsilon_1})=u_{-\epsilon_1}+t u_{-\epsilon_1-2\epsilon_2},\]
	\[X_{-\alpha_1}(t)(u_{-\epsilon_1-2\epsilon_2})=u_{-\epsilon_1-2\epsilon_2}-t u_{-2\epsilon_1-\epsilon_2},\]
	\[X_{-\alpha_2}(t)(u_{-2\epsilon_1+\epsilon_2})=u_{-2\epsilon_1+\epsilon_2}+t u_{-2\epsilon_1-\epsilon_2}.\]
	
	The Lie algebra $\mathfrak{g}$,  regarded as a $G$-module with respect to the adjoint action, is isomorphic to $L(2\omega_1)$. 
	
	To determine a Harish-Chandra pair structure on $(G, L(\omega_1+\omega_2))$, one needs to determine a certain $G$-module morphism
	$\mathrm{Sym}_2(L(\omega_1+\omega_2))\to L(2\omega_1)$. 
	\begin{lm}\label{the unique bilinear}
		We have 
		\[\dim\mathrm{Hom}_G(\mathrm{Sym}_2(L(\omega_1+\omega_2)), L(2\omega_1))=\dim\mathrm{Hom}_{\mathfrak{g}}(\mathrm{Sym}_2(L(\omega_1+\omega_2)), L(2\omega_1))=1.\]	
		In particular, there holds
		\[\mathrm{Hom}_G(\mathrm{Sym}_2(L(\omega_1+\omega_2)), L(2\omega_1))=\mathrm{Hom}_{\mathfrak{g}}(\mathrm{Sym}_2(L(\omega_1+\omega_2)), L(2\omega_1)).\]	
	\end{lm}
	\begin{proof}
		Arguing as in Proposition \ref{the unique form}, we have isomorphisms \[\mathrm{Sym}_2(L(\omega_1+\omega_2))^*\simeq \mathrm{Sym}_2(L(\omega_1+\omega_2)^*)\simeq \mathrm{Sym}_2(L(\omega_1+\omega_2))\]
		of $G$ and $\mathfrak{g}$-modules, simultaneously.   
		Thus 
		\[\mathrm{Hom}_G(\mathrm{Sym}_2(L(\omega_1+\omega_2)), L(2\omega_1))\simeq \mathrm{Hom}_G(L(2\omega_1), \mathrm{Sym}_2(L(\omega_1+\omega_2)))\]
		and
		\[\mathrm{Hom}_{\mathfrak{g}}(\mathrm{Sym}_2(L(\omega_1+\omega_2)), L(2\omega_1))\simeq \mathrm{Hom}_{\mathfrak{g}}(L(2\omega_1), \mathrm{Sym}_2(L(\omega_1+\omega_2))).\]
		Note that
		\[\mathrm{Sym}_2(L(\omega_1+\omega_2))_{2\omega_1}=\Bbbk u_{\epsilon_1+2\epsilon_2}u_{\epsilon_1-2\epsilon_2}+\Bbbk u_{2\epsilon_1+\epsilon_2}u_{-\epsilon_2}+\Bbbk u_{2\epsilon_1-\epsilon_2}u_{\epsilon_2}+\Bbbk u_{\epsilon_1}^2\]
		and any (nontrivial) $G$-module morphism from $L(2\omega_1)$ to $\mathrm{Sym}_2(L(\omega_1+\omega_2))$ is uniquely determined by a $X_{\alpha}$-stable element from
		$\mathrm{Sym}_2(L(\omega_1+\omega_2))_{2\omega_1}$,  where $\alpha\in\{\alpha_1, \alpha_2 \}$. Indeed, any such element uniquely determines a $G$-module morphism $V(2\omega_1)\to \mathrm{Sym}_2(L(\omega_1+\omega_2))$.  
		It remains to note that $\mathrm{rad}(V(2\omega_1))$ is either trivial or isomorphic to a direct sum of several copies of $L(\omega_1)$, but $\mathrm{Sym}_2(L(\omega_1+\omega_2))$ does not have
		such composition factors. Let $W$ denote the subspace of $\mathrm{Sym}_2(L(\omega_1+\omega_2))_{2\omega_1}$ consisting of elements annihilated by $x_{\alpha}$, where $\alpha\in\{\alpha_1, \alpha_2\}$. 
		It is obvious that
		\[\dim\mathrm{Hom}_{\mathfrak{g}}(L(2\omega_1), \mathrm{Sym}_2(L(\omega_1+\omega_2)))\leq\dim W. \]
		Let 
		\[f=\alpha u_{\epsilon_1+2\epsilon_2}u_{\epsilon_1-2\epsilon_2}+\beta u_{2\epsilon_1+\epsilon_2}u_{-\epsilon_2}+\gamma u_{2\epsilon_1-\epsilon_2}u_{\epsilon_2}+\delta u_{\epsilon_1}^2.\]	
		Then 
		\[X_{\alpha_1}(t(f))=f +t[(\alpha+2\beta) u_{2\epsilon_1+\epsilon_2}u_{\epsilon_1-2\epsilon_2}+(2\gamma -4\delta)u_{2\epsilon_1-\epsilon_2}u_{\epsilon_1} ]+t^2(-2\gamma+4\delta)u_{2\epsilon_1-\epsilon_2}^2 .\]
		By the above remark, the second term is equal to $t[x_{\alpha_1}, f]$, hence $f$ is $X_{\alpha_1}$-stable if and only if it is $x_{\alpha_1}$-stable if and only if
		\[\alpha+2\beta=\gamma -2\delta=0.\]
		Similarly, we have
		\[X_{\alpha_2}(t(f))=f+t[2(\alpha+\delta)u_{\epsilon_1+2\epsilon_2}u_{\epsilon_1}+(\beta+\gamma)u_{2\epsilon_1+\epsilon_2}u_{\epsilon_2}]+t^2(\alpha+\delta)u_{\epsilon_1+2\epsilon_2}^2,\]
		that is $f$ is $X_{\alpha_2}$-stable if and only if it is $x_{\alpha_2}$-stable if and only if
		\[\alpha+\delta=\beta+\gamma=0. \]
		We obtain $\dim\mathrm{Hom}_G(\mathrm{Sym}_2(L(\omega_1+\omega_2)), L(2\omega_1))=\dim W=1$, that concludes the proof.
	\end{proof}
	\begin{tr}\label{it is HC}
		The pair $(\mathrm{Sp}_4, L(\omega_1+\omega_2))$ is a Harish-Chandra pair with respect to the nonzero symmetric bilinear $G$-equivariant map $L(\omega_1+\omega_2)\times L(\omega_1+\omega_2)\to\mathfrak{g}$ from Lemma \ref{the unique bilinear}. The corresponding algebraic supergroup, denoted by $\mathrm{BRJ}(2; 5)$, is SAS. 
	\end{tr}
	\begin{proof}
		The known non-zero bracket of $\mathfrak{brj}(2;5)$ is $\mathfrak g$-equivariant. By Lemma \ref{the unique bilinear} it is $G$-equivariant. Since the relevant space is one-dimensional, every nonzero map under consideration is a scalar multiple of this bracket. 
		In particular, $(\mathrm{Sp}(4), L(\omega_1+\omega_2))$ satisfies all axioms of Harish-Chandra pairs with respect to this map, that proves the first statement.
		
		Since $G$ is a SAS-group and $\mathfrak{G}_1$ is an irreducible $G$-module, Proposition \ref{SAPS in positive char} concludes the proof. 
	\end{proof}
	
	\begin{rem}\label{Brown algebra} {\rm 
			In characteristic $p=3$, there is another simple Lie superalgebra $\mathfrak{brj}(2; 3)$, whose even part is isomorphic to so-called Brown algebra of dimension $10$, see \cite{BGL4, eld}. In contrast to $\mathfrak{brj}(2; 5)$,  the Lie superalgebra 
			$\mathfrak{brj}(2; 3)$ is not the Lie superalgebra of a smooth connected algebraic supergroup; in particular it has no SAS-supergroup integration. In fact, let $\mathbb{G}$ be an algebraic supergroup, such that
			$\mathfrak{G}\simeq \mathfrak{brj}(2; 3)$. Since 
			$\mathfrak{G}_0=\mathfrak{brj}(2; 3)_0$ is simple, $G$ is a SAS-group and the proof of Proposition \ref{Lie algebra of SAS} shows that $\mathfrak{G}_0$ is a simple Lie algebra of type $B_2$ or $C_2$, which is a contradiction.}
	\end{rem}
	
	\section{Appendix}
	
	The purpose of this appendix is to correct inaccuracies in the definition of a certain spinor invariant given in the articles \cite{schwarz, zubs}. Then we calculate the coefficients of the polynomial 
	$f(x, y, z, t)$ from Subsection 3.6.

	\subsection{Definitions and conventions}
	
	Let $\bf O$ be the split octonion algebra over a
	field $\Bbbk$ of characteristic zero, with standard basis
	\[
	\mathcal B=
	\{e_1,e_2,u_1,u_2,u_3,v_1,v_2,v_3\}.
	\]
	We use the multiplication relations
	\[
	e_1u_i=u_i,\qquad u_ie_2=u_i,
	\qquad 
	e_2v_i=v_i,\qquad v_ie_1=v_i,
	\]
	and
	\[
	u_ie_1=e_2u_i=v_ie_2=e_1v_i=0.
	\]
	The relevant products are
	\[
	u_1u_2=v_3,\quad u_2u_1=-v_3,
	\quad 
	u_1u_3=-v_2,\quad u_3u_1=v_2,
	\quad 
	u_iv_i=e_1,\quad v_iu_i=e_2.
	\]
	For the trace, we use Schwarz's normalization:
	\[
	\operatorname{tr}(e_1)
	=\operatorname{tr}(e_2)=\frac12,
	\qquad
	\operatorname{tr}(u_i)
	=\operatorname{tr}(v_i)=0.
	\]
	Define
	\[
	M(x,y)=
	\frac12(xy-yx)
	+\operatorname{tr}(x)y-\operatorname{tr}(y)x
	\]
	and
	\[
	G(x,y,z,t)=
	\operatorname{tr}\bigl(M(x,y)M(z,t)\bigr).
	\]
	Note that in \cite{zubs} the agreement on trace normalization was incorrectly omitted in the definition of the polynomials $M$ and $G$.
	The complete skew symmetrization of $G(x, y, z, t)$ is defined as
	\begin{equation*}
		\begin{aligned}
			Q(x,y,z,t)=\frac16\bigl[
			&G(x,y,z,t)-G(z,y,x,t)-G(t,y,z,x)-G(x,t,z,y)\\
			&-G(x,z,y,t)-G(z,t,y,x)
			\bigr].
		\end{aligned}
		\label{eq:Q}
	\end{equation*}
	In \cite{schwarz}, it was stated that $Q(x, y, z, t)$ is invariant under the action of one of the realizations of the spinor Lie algebra, namely, the algebra $d\rho^2(\mathfrak{so}_7)$. Our version of the spinor Lie algebra,
	$d\rho(\mathfrak{so}_7)$, is related to Schwarz's one as (see the next subsection): 
	\[\overline{d\rho^2(h)[x]}=d\rho(h)[\overline{x}], \ h\in\mathfrak{so}_7, \ x\in {\bf O}.\]
	In particular, $Q(a, b, c, d)$ would be a $d\rho^2(\mathfrak{so}_7)$-invariant if and only if $Q(\bar a, \bar b, \bar c, \bar d)$ is  $d\rho(\mathfrak{so}_7)$-invariant. 
	However, it turned out that $Q$ is not invariant!
	
	Let \(L_a\) and \(R_a\) denote left and right
	multiplication by \(a\):
	\[
	L_a(x)=ax,\qquad R_a(x)=xa.
	\]
	Recall that $\mathfrak{so}_7$ is generated by
	the Lie subalgebra $\mathfrak{g}_2$ and by the elements $L_a-R_a, a\in {\bf O}_0$. Furthermore, we have
	\[d\rho(g)=g, \ g\in\mathfrak{g}_2, \ d\rho(L_a)=R_a, \ d\rho(R_a)=-(L_a+R_a).\]
	Thus $d\rho(\mathfrak{so}_7)$ is generated by
	$\mathfrak{g}_2$ and by the elements $2 R_a+L_a$, as well as $d\rho^2(\mathfrak{so}_7)$ is generated by
	$\mathfrak{g}_2$ and by the elements $2 L_a+R_a$ respectively. 
	
	We consider the element
	$B=2L_{u_1}+R_{u_1}\in d\rho^2(\mathfrak{so}_7)$.
	The infinitesimal invariance condition for \(Q\) is
	\begin{equation*}
		\begin{aligned}
			0={}&Q(Bx,y,z,t)+Q(x,By,z,t)+Q(x,y,Bz,t)+Q(x,y,z,Bt).
		\end{aligned}
		\label{eq:invariance}
	\end{equation*}
	Testing this identity at
	$(x,y,z,t)=(e_1,e_2,u_2,u_3)$
	we obtain 

	\[
	\begin{aligned}
		&Q(Be_1,e_2,u_2,u_3)
		+Q(e_1,Be_2,u_2,u_3)
		+Q(e_1,e_2,Bu_2,u_3)\\&
		+Q(e_1,e_2,u_2,Bu_3)
		=\frac32\neq 0.
	\end{aligned}
	\]
	
	\subsection{Corrected definition of $Q$}
	
	From now on, we assume $\mathrm{char}\Bbbk=p\neq 2$, unless stated otherwise. Set 
	\[
	\tilde{M}(x, y):=\frac{1}{2}(yx-xy)+\mathrm{tr}(x)y-\mathrm{tr}(y)x=\frac{1}{2}(\bar x y-\bar y x).
	\]
	As above, we define $\tilde{G}(x, y, z, t)=\mathrm{tr}(\tilde{M}(x, y)\tilde{M}(z, t))$.
	
	It is known that $\mathrm{SO}(8)$ is generated by the operators $L_a$ and $R_a$, $a\in {\bf O}_1$, where ${\bf O}_1$ consists of all octonions of norm one. If it does not lead to confusion, we denote their images in $\mathrm{PSO}(8)$ by the same symbols. Let $Z\simeq\mu_2$ denote the center of $\mathrm{SO}(8)$.
	
	Let $\sigma$ and $\rho$ be the outer automorphisms of $\mathrm{PSO}(8)$ defined in  \cite{zubs}. They are
	\[\sigma(L_a)=R_{a^{-1}}=R_a^{-1}, \ \sigma(R_a)=L_{a^{-1}}=L_a^{-1},\]
	\[\rho(L_a)=R_a, \ \rho(R_a)=L_{a}^{-1} R_a^{-1}=R_a^{-1}L_a^{-1}.\]
	In addition, $\rho^3=\sigma^2=\mathrm{id}, \ \sigma\rho\sigma=\rho^{-1}=\rho^2$. 
	Note also that the automorphism $\sigma$ can be lifted to an automorphism of $\mathrm{SO}(8)$.
	In fact, let $g\in\mathrm{GL}(8)$ be the involution $x\mapsto\bar x, x\in{\bf O}$. 
	One easily sees that the conjugation by $g$ acts on the generators $L_a$ and $R_a$ of $\mathrm{SO}(8)$ as $\sigma$ does. In the equivalent form, we have
	\[h[\bar x]=\overline{\sigma(h)[x]}, \quad h\in\mathrm{SO}(8), x\in {\bf O}.\]

	Further, the Lie algebra $\mathfrak{so}(8)\simeq\mathfrak{pso}(8)$ has the outer automorphisms $d\sigma$ and $d\rho$, such that
	\[d\sigma(L_a)=-R_a, \ d\sigma(R_a)=-L_a, \ d\rho(L_a)=R_a, \ d\rho(R_a)=-(L_a+R_a)\]
	and $(d\rho)^3=(d\sigma)^2=\mathrm{id}, d\sigma d\rho d\sigma=(d\rho)^{-1}=(d\rho)^2$.
	
	Recall that $\mathrm{SO}(7)\simeq\mathrm{Stab}_{\mathrm{SO}(8)}(1_{\bf O})$ and since $\mathrm{SO}(7)\cap Z=e$, $\mathrm{SO}(7)$ is naturally embedded into $\mathrm{PSO}(8)$.
	Moreover, $g\in\mathrm{PSO}(8)$ belongs to $\mathrm{SO}(7)$ if and only if $\sigma(g)=g$.
	
	Let $G_{(i)}\simeq\mathrm{Spin}(7)$ denote the preimage of $\rho^i(\mathrm{SO}(7))$ in $\mathrm{SO}(8), i=1, 2$. In what follows we consider $\bf O$ with respect to the action of $G_{(2)}$. 
	
	Note that the decomposition of ${\bf O}^{\otimes 2}$ into a direct sum of irreducible $\mathrm{Spin}(7)$-modules does not depend on the choice of implementation of $\mathrm{Spin}(7)$.
	\begin{lm}\label{spinor invariants of degree 4}
		We have 
		\[({\bf O}^{\otimes 4})^{\mathrm{Spin}(7)}=({\bf O}^{\otimes 4})^{\mathfrak{spin}(7)}=({\bf O}^{\otimes 2})^{d\rho^2(\mathfrak{so}(7))}.\]
	\end{lm}
	\begin{proof}
		Recall that $\mathrm{SO}(8)$-module $\bf O$ is self-dual. Therefore, we have
		\[({\bf O}^{\otimes 4})^{\mathrm{Spin}(7)}\simeq (({\bf O}^*)^{\otimes 2}\otimes {\bf O}^{\otimes 2})^{\mathrm{Spin}(7)}\simeq\]
		\[\mathrm{End}_{\mathrm{Spin}(7)}({\bf O}^{\otimes 2})\simeq \mathrm{End}_{\mathrm{Spin}(7)}(L(0)\oplus L(\omega_1)\oplus L(\omega_2)\oplus L(2\omega_3) ).\]
		Since the weights $0, \omega_1, \omega_2, 2\omega_3,$ are $p$-restricted and pairwise different, we obtain that the latter space is four-dimensional and isomorphic to
		\[\mathrm{End}_{\mathfrak{spin}(7)}({\bf O}^{\otimes 2})\simeq ({\bf O}^{\otimes 4})^{\mathfrak{spin}(7)}.\]
	\end{proof}
	This lemma shows that the space of multilinear $\mathrm{Spin}(7)$-invariants of degree $4$ coincides with the space of multilinear $\mathfrak{spin}(7)$-invariants of the same degree. 
	
	\begin{lm}\label{a form of triality}
		For any $h\in\mathfrak{so}(8)$ and $x, y\in {\bf O}$ we have the following form of local triality principle:
		\[h [xy]=d\rho d\sigma(h)[x]y+xd\sigma d\rho(h)[y].\]
	\end{lm}
	\begin{proof}
		By \cite[Theorem 3.5.5 (i)]{springveld}, for any $h_1\in\mathfrak{so}(8)$ there are $h_2, h_3\in\mathfrak{so}(8)$, such that for any $x, y\in {\bf O}$ we have
		\[h_1[xy]=h_2[x]y+ xh_3[y].\] 	
		Moreover, the maps $h_1\mapsto h_2$ and $h_1\mapsto h_3$ are Lie algebra morphisms. 
		
		Without loss of generality, one can assume that $h$ equals to a left-normed commutator in the generators $L_a$ and $R_b, a, b\in {\bf O}_0$. Accordingly to  the last statement in \cite[Theorem 3.5.5]{springveld}, if $h_1=L_a$ or $h_1=R_a$, then 
		\[h_2=L_a+R_a=d\rho d\sigma(L_a), \ h_3=-L_a=d\sigma d\rho(L_a),\]
		or \[h_2=-R_a=d\rho d\sigma(R_a), \ h_3=L_a +R_a=d\sigma d\rho(R_a),\]
		respectively. The above remark and the induction on the length of the commutator representing $h$ entail the statement.   
	\end{proof}
	\begin{lm}\label{Schwarz's invariant}
		$\tilde{G}(x, y, z, t)$ is invariant with respect to the action of the group $G_{(2)}$.	
	\end{lm}
	\begin{proof}
		By Lemma \ref{spinor invariants of degree 4} one needs to prove that $\tilde{G}$ is invariant with respect to the action of Lie subalgebra $\mathfrak{spin}(7)\simeq d\rho^2(\mathfrak{so}(7))$.
		Let $h\in\mathfrak{so}(7), g=d\rho^2(h)$ and $x, y\in {\bf O}$.  Since $d\sigma(h)=h$, Lemma \ref{a form of triality} implies
		\[\overline{g[x]} y+\bar x g[y]=\overline{d\rho^2(h)[x]}y+\bar x d\rho^2(h)[y]=\]\[d\sigma d\rho^2(h)[\bar x]y+\bar x d\rho^2 d\sigma(h)[y]=d\rho d\sigma(h)[\bar x]y+\bar x d\sigma d\rho(h)[y]=h[\bar x y]. \]
		Therefore, $M(g[x], y)+M(x, g[y])=h[M(x, y)]$ for any octonions $x, y$. It remains to note that the element $M(x, y)$ belongs to ${\bf O}_0$ and for any
		traceless octonions $x', y'$ we have $\mathrm{tr}(x' y')=-q(x', y')$ is even $\mathrm{SO}(8)$-invariant. 
	\end{proof}
	Similarly \cite[Proposition 17]{zubs}, one can show that the complete skew symmetrization $\tilde{Q}(x, y, z, t)$ of $\tilde{G}(x, y, z, t)$ is defined over $\mathbb{Z}[\frac{1}{2}]$, hence over any field of odd or zero characteristic.
	
	\subsection{Determining $\beta$ and $\delta$}
	
	We consider the polynomial
	\[
	f(x,y,z,t)
	=
	\beta\bigl(
	q(x,z)q(y,t)-q(x,t)q(y,z)
	\bigr)
	+\delta\widetilde Q(x,y,z,t).
	\]
	Recall that in \cite[Proposition 15]{zubs} it was proved that with respect to the action of $G_{(1)}$, the direct term of $\Lambda^2({\bf O})$ isomorphic to $L(\omega_1)$ is generated by the highest weight vector $e_1\wedge u_1-v_2\wedge v_3$. In terms of the action of $G_{(2)}$ this means that $e_2\wedge u_1+v_2\wedge v_3$ generates the corresponding direct term of $\Lambda^2({\bf O})$. In particular, we require
	\[
	f(e_2,u_1,z,t)+f(v_2,v_3,z,t)=0, 
	\]
	for all $z,t\in {\bf O}$.
	
	\medskip
	\noindent
	\textbf{Case 1: \(z=e_1\) and \(t=v_1\).} We have
	\[
	q(e_2,e_1)=\frac12,
	\; 
	q(u_1,v_1)=-\frac12,
	\; 
	\text{ and } \; 
	\widetilde Q(e_2,u_1,e_1,v_1)=\frac14,
	\;
	\widetilde Q(v_2,v_3,e_1,v_1)=\frac12.
	\]
	Consequently,
	\[
	\begin{aligned}
		0
		&=f(e_2,u_1,e_1,v_1)
		+f(v_2,v_3,e_1,v_1)=-\frac{\beta}{4}
		+\frac{\delta}{4}
		+\frac{\delta}{2}=-\frac{\beta}{4}
		+\frac{3\delta}{4}.
	\end{aligned}
	\]
	Therefore,
	$
	\beta=3\delta.
	$
	
	\medskip
	\noindent
	\textbf{Case 2: \(z=u_2\) and \(t=u_3\).} The bilinear-form contribution is
	\[
	\begin{aligned}
		&q(e_2,u_2)q(u_1,u_3)
		-q(e_2,u_3)q(u_1,u_2)
		+q(v_2,u_2)q(v_3,u_3)
		-q(v_2,u_3)q(v_3,u_2)=\frac14.
	\end{aligned}
	\]
	Using the previously obtained relation
	\[
	\widetilde Q(e_2,u_1,u_2,u_3)
	+\widetilde Q(v_2,v_3,u_2,u_3)
	=-\frac34,
	\]
	we obtain
	\[
	\begin{aligned}
		0
		&=f(e_2,u_1,u_2,u_3)
		+f(v_2,v_3,u_2,u_3)=\frac{\beta}{4}
		-\frac{3\delta}{4}.
	\end{aligned}
	\]
	Therefore,
	$
	\beta=3\delta.
	$ 
	
	
	

\end{document}